\documentclass[10pt,a4paper,reqno]{amsart}

\pdfoutput=1

\usepackage{amsmath,bbm}
\usepackage{amssymb}
\usepackage{graphicx}
\usepackage{color}
\usepackage{latexsym}
\usepackage[T1]{fontenc}
\usepackage{stmaryrd,comment}
\usepackage{url}
\usepackage[plainpages=false,pdfpagelabels,colorlinks=true,citecolor=blue,hypertexnames=false]{hyperref}

\newcommand{\ns}{{\mathbb N}} 
\newcommand{\qs}{{\mathbb Q}}  
\newcommand{\cs}{{\mathbb C}} 
\newcommand{\rs}{{\mathbb R}} 
\newcommand{\al}{\alpha}
\newcommand{\be}{\beta}

\newcommand{\la}{\lambda}

\newcommand{\vareps}{\varepsilon}

\newcommand{\by}{\bar y}

\newcommand{\bC}{C}
\newcommand{\bD}{D}
\newcommand{\bN}{N}
\newcommand{\hP}{\widehat{P}}
\newcommand{\hQ}{\widehat{Q}}
\newcommand{\hR}{\widehat{R}}
\newcommand{\hN}{\mathcal N}
\newcommand{\hD}{\mathcal D}

\newcommand{\bP}{P}
\newcommand{\bQ}{Q}
\newcommand{\bR}{R}
\newcommand{\tQ}{\tilde{Q}}
\newcommand{\tR}{\tilde{R}}
\newcommand{\tT}{T}
\newcommand{\tM}{\tilde{M}}
\newcommand{\tP}{\tilde{P}}

\newcommand{\Q}{T}

\newcommand{\GK}{\mathbb{K}}

\DeclareMathOperator{\Eq}{Eq}
\DeclareMathOperator{\tEq}{\tilde Eq}
\DeclareMathOperator{\Rat}{Rat}
\DeclareMathOperator{\Pol}{Pol}

\DeclareMathOperator{\df}{drf}

\DeclareMathOperator{\cc}{c}
\DeclareMathOperator{\vv}{v}
\DeclareMathOperator{\ff}{f}
\DeclareMathOperator{\ee}{e}

\DeclareMathOperator{\Tpol}{T}
\DeclareMathOperator{\Tch}{T}
\DeclareMathOperator{\Ppol}{P}

\newcommand{\gM}{\bar M}

\newcommand{\R}{R}

\newcommand{\cC}{\mathcal C}

\newcommand{\cM}{\mathcal M}

\newcommand{\hM}{\hat{M}}

\newcommand{\cS}{\mathcal S}
\newcommand{\tcS}{\tilde{\mathcal S}}

\newtheorem{Theorem}{Theorem}
\newtheorem*{Theorem*}{Theorem~\ref{thm:ED} (repeated)}
\newtheorem{Proposition}[Theorem]{Proposition}

\newtheorem{Lemma}[Theorem]{Lemma}

\newcommand{\beq}{\begin{equation}}
\newcommand{\eeq}{\end{equation}}

\newcommand{\gf}{generating function}

\newcommand{\fps}{formal power series}

\def\emm#1,{{\em #1}}

\catcode`\@=11
\def\section{\@startsection{section}{1}%
 \z@{.7\linespacing\@plus\linespacing}{.5\linespacing}%
 {\normalfont\bfseries\scshape\centering}}

\def\subsection{\@startsection{subsection}{2}%
 \z@{.5\linespacing\@plus\linespacing}{.5\linespacing}%
 {\normalfont\bfseries\scshape}}

\def\subsubsection{\@startsection{subsubsection}{3}%
 \z@{.5\linespacing\@plus\linespacing}{-.5em}
 {\normalfont\bfseries\itshape}}
\catcode`\@=12

\def\qee{$\hfill{\Box}$}

\begin{document}
\title
[Counting coloured planar maps: differential equations]
{Counting coloured planar maps: differential equations}

\author[O. Bernardi]{Olivier Bernardi}
\address{O. Bernardi: Brandeis University, Department of Mathematics, 415 South Street, Waltham, MA 02453, USA}
\email{bernardi@brandeis.edu}

\author[M. Bousquet-M\'elou]{Mireille Bousquet-M\'elou}
\address{M. Bousquet-M\'elou: CNRS, LaBRI, Universit\'e de Bordeaux,
351 cours de la Lib\'eration, 33405 Talence, France}
\email{mireille.bousquet@labri.fr}

\thanks{The authors were partially supported by the French ``Agence Nationale
de la Recherche'', first via project A3 ANR-08-BLAN-0190 and then
project Graal ANR-14-CE25-0014. OB was partially supported by the NSF
grants DMS-1308441 and  DMS-1400859. MBM also acknowledges the hospitality
of the Schrödinger Institute in Vienna, where part of this work was
accomplished during the programme ``Combinatorics, Geometry and
Physics'' in 2014.}

\keywords{Enumeration -- Coloured planar maps -- Tutte polynomial --
 Differentially algebraic series}
\subjclass[2000]{05A15, 05C30, 05C31}

\begin{abstract}
We address the enumeration of $q$-coloured planar maps counted by
the   number of edges and the number  of monochromatic edges.
We prove that the associated \gf\ is \emm differentially algebraic,,
that is, satisfies a
non-trivial polynomial differential equation with
respect to the edge variable. We give explicitly a differential system
that characterizes this series.   We then
prove a similar result for planar triangulations, thus generalizing
a result of Tutte  dealing with their proper $q$-colourings. In
statistical physics terms, we solve
the $q$-state Potts model on random planar lattices.

This work follows a first paper by the same authors, where the \gf\
was proved to be algebraic for certain values of $q$,
including $q=1, 2$ and $3$. It is
known to be transcendental in general. In contrast, our
differential system holds for an indeterminate $q$.

For certain special cases of combinatorial  interest (four colours; proper
$q$-colourings; maps equipped with a spanning forest), we
derive from this system, in the case of triangulations, an explicit
 differential equation of order
$2$ defining the \gf. For general planar maps, we also obtain a
 differential equation of order 3 for the four-colour case and for the
 self-dual Potts model.
\end{abstract}

\date{\today}
\maketitle

\section{Introduction}
A planar map is a connected planar graph, given with one of its proper
embeddings in the sphere, taken up to continuous deformation
(Figure~\ref{fig:example-map}).  The enumeration of planar maps is a combinatorial
problem that has attracted a lot
of interest since the sixties, in connection with graph theory~\cite{tutte-4-colored,gimenez-noy-planar}, algebra~\cite{goulden-jackson-KP,jackson-visentin}, theoretical
physics~\cite{BIZ,BIPZ,DFGZJ}, and computational geometry~\cite{castelli,schnyder-embedding}. Several important
approaches have been developed: recursive~\cite{tutte-general}, bijective~\cite{Sch97}, using matrix
integrals~\cite{BIPZ}, or using connections with the characters of the symmetric
group~\cite{Jackson:Harer-Zagier}.
We only give very few references, as  a complete bibliography would
take dozens of pages.

From the combinatorial and physical point of view, it is  natural
to count planar maps equipped with an additional structure: for
instance a spanning tree~\cite{mullin-boisees}, a proper
colouring~\cite{lambda12,tutte-differential},  an independent set of
vertices~\cite{mbm-jehanne,mbm-schaeffer-ising,BDG-blocked}, a configuration of the Ising or Potts model~\cite{Ka86,eynard-bonnet-potts}, a self-avoiding
walk~\cite{DK88}. (Again, we  give very few of the relevant
references.) The first result of this nature  probably dates back to
1967 with Mullin's
enumeration of planar maps equipped with a spanning tree~\cite{mullin-boisees}. The
second attempt is due to Tutte, who, in the early seventies, started
to study maps --- more precisely, triangulations --- equipped with a
proper colouring~\cite{lambda12}.
In the decade that followed, he devoted at least eight other papers
to  this problem~\cite{lambda3,lambda-tau,tutteIV,tutteV,tutte-pair,tutte-chromatic-sols,tutte-chromatic-solsII,tutte-differential}. His
work culminated in 1982, when he proved that the series $H(w)$
counting $q$-coloured rooted triangulations by vertices
satisfies a  polynomial differential
equation~\cite{tutte-chromatic-solsII,tutte-differential}:
\beq\label{Tutte-ED}
2q^2(1-q)w +(qw+10H-6wH')H''+q(4-q)(20H-18wH'+9w^2H'')=0.
\eeq
We say that $H(w)$ is \emm differentially algebraic,.
Equivalently,  the number $h_n$ of rooted triangulations
with $n$ vertices  satisfies the following
simple recurrence relation:
$$
q(n+1)(n+2)h_{n+2}=q(q-4)(3n-1)(3n-2)h_{n+1} + 2\sum_{i=1}^n
i(i+1)(3n-3i+1)h_{i+1}h_{n+2-i},
$$
with the initial condition $h_2=q(q-1)$. For instance,
$h_3=q(q-1)(q-2)$ is the number of proper $q$-colourings of a triangle.
To date, this recursion remains entirely mysterious,
and Tutte's \emm tour de force, has remained isolated.

Let us be more precise about
the content of this \emm tour de force,.
  Tutte started from  a \gf
\ $G(w;x,y)\equiv G(x,y)$ which counts a larger family of $q$-coloured
maps according to three
parameters (as before, $w$ counts vertices). The above
series $H(w)$
is $G(w;1,0)$.
Using the
deletion/contraction properties of the chromatic polynomial,
he easily established  the following functional equation:
\begin{multline}\label{eq-Tutte}
G(x,y)=xq(q-1)w^2+\frac{xy}{qw}G(1,y)G(x,y)-x^2yw\frac{G(x,y)-G(1,y)}{x-1}+x\frac{G(x,y)-G(x,0)}{y}.
\end{multline}
Observe that  the
  divided differences (or discrete derivatives)
$$
\frac{G(x,y)-G(1,y)}{x-1}\qquad\hbox{ and } \qquad
\frac{G(x,y)-G(x,0)}{y}
$$
prevent us from simply specializing $x$ to 1 and $y$ to 0 to obtain an
equation for $H=G(1,0)$. In fact,
 the variables $x$ and $y$, called nowadays \emm catalytic variables,~\cite{zeil-umbral,mbm-jehanne},
are essential to write  this equation, even though they are not very
important combinatorially. The whole challenge is to determine the
 nature of  the series $H=G(1,0)$: is it algebraic, as the \gf\ of many classes of
planar maps? Is it D-finite, that is, a solution of a linear
differential equation, as the \gf\ of maps equipped with a spanning
tree~\cite{mullin-boisees}? Is it at least differentially algebraic? Tutte
answered the latter question positively by deriving from the functional
equation~\eqref{eq-Tutte} the
 differential equation~\eqref{Tutte-ED} satisfied by $H$.

It is known that the solutions of polynomial equations with  \emm one,
catalytic variable
are systematically algebraic~\cite{mbm-jehanne}. Such equations
commonly occur
in the enumeration of families of planar maps
with no additional structure. Moreover, in the past decade, progress has been made
on \emm linear, equations with \emm two,
catalytic variables,   which one typically  encounters when counting
plane lattice
walks confined to a quadrant. In this context,  it is precisely understood when
the \gf\ is D-finite (see e.g.~\cite{BoRaSa12,BoMi10,KuRa12} and references therein). Hence the key difficulty with Tutte's
equation~\eqref{eq-Tutte} is the occurrence of two catalytic variables,
combined with non-linearity.

\medskip
 A few years ago, we started to resurrect Tutte's technique in
order to
solve a more general problem: {\em the Potts model on planar maps}. In
combinatorial terms, this means that we count all $q$-colourings,
not necessarily proper, but with a weight $\nu^m$, where $\nu$ is an
indeterminate and $m$ is the  number of \emm monochromatic,  edges (edges whose
endpoints have the same colour). The case $\nu=0$ is thus the problem
solved by Tutte. Moreover, we do not only  study
degree-constrained maps (triangulations) as Tutte did, but also general planar
maps (counted by edges and vertices).

In a first paper on this topic~\cite{bernardi-mbm-alg}, we wrote the counterpart
of~\eqref{eq-Tutte} for each of these two problems. Then, we
performed a first step, by establishing an equation with only \emm
one, catalytic variable,~$y$. But this equation holds only for values
of $q$  of the form $2+2\cos (j\pi/m)$, for integers $j$
and $m$ (with  $q\not = 0, 4$).  Moreover, the size of this equation grows with
$m$. Nevertheless, equations with  one catalytic variable are much
better understood that those with two~\cite{mbm-jehanne}, and we were able to prove that the \gf\ of $q$-coloured maps is algebraic for
all such values of~$q$, including $q=2$ and $q=3$. It is known to be
transcendental in general~\cite{bernardi-mbm-alg}.

In this paper, we perform the second step, and prove that for  $q$ an
indeterminate, the \gf\ of $q$-coloured planar maps (or
triangulations) is differentially algebraic.  We give an explicit
differential system that characterizes it.  This system has the same form for
general maps and triangulations.
In some special cases (proper colourings; four colours; spanning
forests; self-dual Potts model) we derive from it an explicit differential
equation of small order defining the \gf. In particular, we
recover Tutte's result~\eqref{Tutte-ED} for properly coloured
triangulations.
For the convenience of the reader, all calculations are performed in
an accompanying Maple session, available on the web pages of the
authors.

 \medskip

Here is  an outline of the paper.
In Section~\ref{sec:def} we give
definitions about maps, the Potts model and its combinatorial
counterpart, the Tutte
polynomial, as well as notation on power series.
In Section~\ref{sec:example} we state our main
result and give an idea of Tutte's sophisticated
approach (some would say obscure) on a simple example. In Sections~\ref{sec:general} and~\ref{sec:triang} we establish  differential
systems for $q$-coloured planar maps and $q$-coloured triangulations,
respectively. We simplify these systems in Section~\ref{sec:non-diff}. The next sections deal with special cases: proper
colourings (Section~\ref{sec:nu=0}), four colours (Section~\ref{sec:q=4}),
then the $q=0$ case, which corresponds to the enumeration of maps equipped
with a spanning forest (Section~\ref{sec:q=0}), and finally the
self-dual Potts model  (Section~\ref{sec:sd}).

Our differential system for planar maps appeared (without proof) in a
survey  on map enumeration published in 2011~\cite{mbm-survey}.
\section{Definitions and notation}
\label{sec:def}

\subsection{Planar maps}
%
A \emph{planar map} is a proper
 embedding of a connected planar graph in the
oriented sphere, considered up to orientation preserving
homeomorphism. Loops and multiple edges are allowed
(Figure~\ref{fig:example-map}). The \emph{faces} of a map are the
connected components of  its complement. The numbers of
vertices, edges and faces of a planar map $M$, denoted by $\vv(M)$,
$\ee(M)$ and $\ff(M)$,  are related by Euler's relation
$\vv(M)+\ff(M)=\ee(M)+2$.
 The \emph{degree} of a vertex or face is the number
of edges incident to it, counted with multiplicity. A \emph{corner} is
a sector delimited by two consecutive edges around a vertex;
hence  a vertex or face of degree $k$ defines $k$ corners.
Alternatively, a corner can be described as an incidence between a
vertex and a face. The \emph{dual} of a
map $M$, denoted $M^*$, is the map obtained by placing a
vertex of $M^*$ in each face of $M$ and an edge of $M^*$ across each
edge of $M$; see Figure~\ref{fig:example-map}. A \emph{triangulation}
is a map in which every face has degree 3. Duality transforms
triangulations into \emph{cubic} maps, that is, maps in which every vertex has
degree 3.

For counting purposes it is convenient to consider \emm rooted, maps.
A map is rooted by choosing a corner, called  the \emm root-corner,.
The vertex and face that are incident at this corner are respectively
the \emm root-vertex, and the \emm root-face,.
 The \emm root-edge, is the edge that follows the
root-corner in counterclockwise
order around the root-vertex.
 In figures, we  indicate the rooting by
 an arrow pointing to the root-corner, and take the root-face
as the infinite face (Figure~\ref{fig:example-map}).
This explains why we often call the root-face the \emm outer face, and
its degree the \emm outer degree, (denoted $\df(M)$).

From now on, every {map}
is \emph{planar} and \emph{rooted}. By convention,
we include among rooted planar maps the \emph{atomic map}
 having one vertex and no edge.

\begin{figure}[h]
  \centering
  \includegraphics{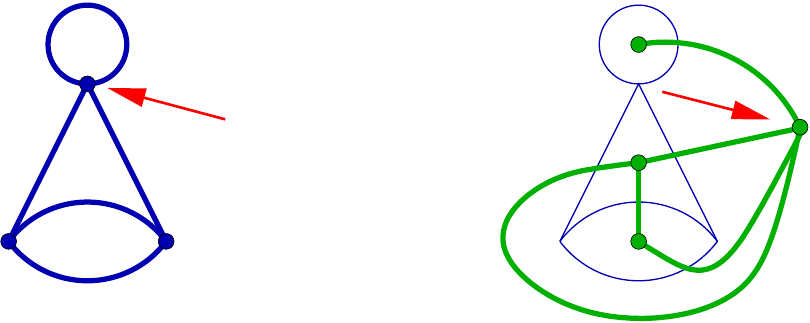}
\caption{A rooted planar map and its dual
(rooted at the dual corner).}
\label{fig:example-map}
\end{figure}

%

\subsection{The Potts model}
Let $G$ be a graph with vertex set $V(G)$ and edge set $E(G)$.
Let $\nu$ be an indeterminate, and take $q\in \ns$.
A \emm colouring, of the vertices of $G$ in $q$ colours is a map $c : V(G)
\rightarrow \{1, \ldots, q\}$. An edge of $G$ is \emm monochromatic,
if its endpoints share the same colour. Every loop is thus
monochromatic. The number of monochromatic edges is denoted by $m(c)$.
The \emm partition function of the  Potts model,
on $G$ counts colourings by the number of monochromatic edges:
$$
\Ppol_G(q, \nu)= \sum_{c  : V(G)\rightarrow \{1, \ldots, q\}}
\nu^{m(c)}.
$$
The Potts model is a classical magnetism model in statistical physics, which
includes (when $q=2$) the famous Ising model (with no magnetic
field)~\cite{wu,welsh-merino}.
Of course, $\Ppol_G(q,0)$ is the chromatic
polynomial of $G$, which counts proper colourings (no monochromatic edge).

It is not hard to see that  $\Ppol_{G}(q,\nu)$ is
 a  polynomial in $q$ and $\nu$.
We call it the \emm Potts polynomial, of $G$.  Observe that it is
a multiple of $q$.
We will often consider $q$
as an indeterminate, or evaluate  $\Ppol_{G}(q,\nu)$ at
real values $q$.

We define the \emm Potts \gf,\ of planar maps by:
\beq\label{M-def}
M(y)\equiv M(q,\nu,w,t;y)
=\frac 1 q \sum_{M} \Ppol_M(q,\nu) w^{\vv(M)} t^{\ee(M)} y^{\df(M)},
\eeq
where the sum runs over all planar maps $M$.
Since there are finitely many maps with a given number of edges,
and $\Ppol_M(q,\nu)$ is a multiple of $q$,
the generating function $M(y)$ is a power series in $t$ with
coefficients in $\qs[q,\nu,w,y]$, the ring of polynomials in $q, \nu,
w$ and $y$ with rational coefficients. The expansion of $M$ at order 2
reads
\begin{multline*}
M(y)=\ w+\left(w^2y^2(q-1+\nu)+w y \nu \right)t
+ \left(
2w^3y^4(q-1+\nu)^2
+w^2y^2(q-1+\nu^2)\right.
\\
\left. +w^2\nu(y+3y^3)(q-1+\nu)
+w\nu^2(y+y^2)
\right)t^2+O(t^3),
\end{multline*}
as illustrated in Figure~\ref{fig:small-maps}. In combinatorial terms,
$M$ counts $q$-coloured planar maps by vertices, edges,
monochromatic edges and outer degree, with the convention that the root-vertex is
coloured in a prescribed colour (this accounts for the division by $q$).

\begin{figure}[h]
  \centering
  \includegraphics{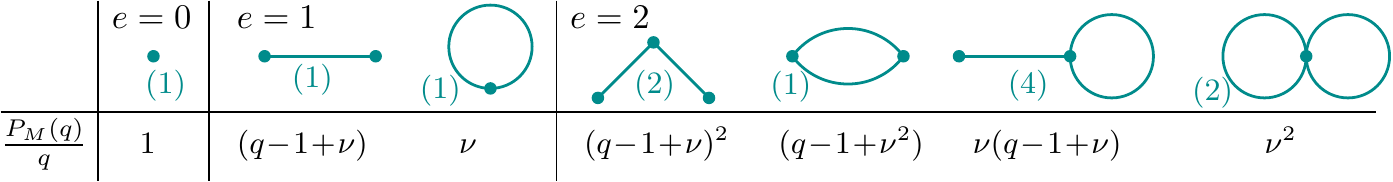}
\caption{The planar maps with $e=0,1,$ and $2$ edges, and their Potts
  polynomials (divided by $q$). The maps are shown \emph{unrooted},
  and the numbers in parentheses
indicate the
  number of corresponding \emph{rooted} maps.}
\label{fig:small-maps}
\end{figure}

\subsection{Specializations}
\label{sec:specializations}
We consider in this paper several specializations of the Potts
polynomial. For some of them, the combinatorial meaning is obvious:
for instance, $\Ppol_G(q,0)$ counts proper $q$-colourings, and $\Ppol_G(4, \nu)$
counts 4-colourings by monochromatic edges. To understand the
significance of some other specializations, it is useful to relate
$\Ppol_G$ to
another invariant of graphs: the \emm Tutte polynomial,
$\Tpol_G(\mu, \nu)$. It is defined as follows (see e.g. \cite{Bollobas:Tutte-poly}):
\beq\label{Tutte-def}
\Tpol_G(\mu,\nu):=\sum_{S\subseteq
  E(G)}(\mu-1)^{\cc(S)-\cc(G)}(\nu-1)^{\ee(S)+\cc(S)-\vv(G)},
\eeq
where the sum is over all spanning subgraphs of $G$ (equivalently,
over all subsets of edges) and  $\vv(.)$, $\ee(.)$ and $\cc(.)$ denote
respectively the number of vertices, edges and connected
components.
The equivalence with the Potts polynomial was established  by
Fortuin and Kasteleyn~\cite{Fortuin:Tutte=Potts}:
\beq\label{eq:Tutte=Potts}
\Ppol_G(q,\nu)=\sum_{S\subseteq E(G)}q^{\cc(S)}(\nu-1)^{\ee(S)}=(\mu-1)^{\cc(G)}(\nu-1)^{\vv(G)}\,\Tpol_G(\mu,\nu),
\eeq
for $q=(\mu-1)(\nu-1)$.
The Tutte polynomial satisfies an interesting duality property:
if $G$ and $G^*$ are dual
connected planar graphs  (that is, if $G$ and $G^*$ can be embedded
as dual planar maps) then
\beq\label{eq:duality-Tutte-poly}
\Tpol_{G^*}(\mu,\nu)=\Tpol_G(\nu,\mu).
\eeq
This gives a duality relation on the Potts \gf\ defined by~\eqref{M-def}:
\beq\label{dual-M}
M(q,\nu,w,t;1)=w^2q M(q,\mu, (wq)^{-1},tw(\nu-1);1).
\eeq
The connection~\eqref{eq:Tutte=Potts} between $\Ppol_G$ and $\Tpol_G$
allows us to understand combinatorially the limit $q=0$ of
$\Ppol_G/q$: for a planar map $M$,
\begin{align*}
  \lim_{q\rightarrow 0} \frac 1 q
  \Ppol_M(q,\nu)&=(\nu-1)^{\vv(M)-1}\Tpol_M(1, \nu)\\
&=\sum_{C \ \small{\rm connected\  on\  } M} (\nu-1)^{\ee(C)} & \hbox{by \eqref{Tutte-def},}\\
&=(\nu-1)^{\ff(M^*)-1}\Tpol_{M^*}(\nu,1)& \hbox{by \eqref{eq:duality-Tutte-poly},}\\
&=(\nu-1)^{\ff(M^*)-1}\sum_{F\ \small{\rm forest\  on\  } M^*}
(\nu-1)^{\cc(F)-1}& \hbox{by \eqref{Tutte-def} again,}
\end{align*}
where the first sum runs over all connected subgraphs $C$  of $M$, and the
second over all spanning \emm forests, $F$ of $M^*$ (subsets of edges
of $M^*$ with no cycle).
Hence, if $S(q, \nu,w,t)$
is the Potts \gf\ of some family of planar maps $\cM$
(meaning that each map $M$ is given a weight $\frac 1 q
P_M(q,\nu)w^{\vv(M)}t^{\ee(M)}$)
and $\cM^*$ denotes
the set of duals of maps of $\cM$,
\begin{align}
  S(0,1+\be, w,t)&=\sum_{M\in \cM, \, C} \be^{\ee(C)} w^{\vv(M)} t^{\ee(M)}\nonumber\\
&=\frac 1 {\be} \sum_{M^* \in \cM^*,\, F} \be^{\cc(F)-1}
\left(\be w\right)^{\ff(M^*)} t^{\ee(M^*)},\label{forests}
\end{align}
where the first sum runs over all maps $M$ of $\cM$ equipped with a
connected subgraph $C$, and the second over all maps $M$ of $\cM^* $
equipped with a spanning forest $F$.

Our final specialization is the self-dual Potts model, in which a map
and its dual get the same weight. In view of the duality relation~\eqref{dual-M}, this
means that $\mu=\nu$, so that $q=(\nu-1)^2$, and that $w=1/\sqrt q
=1/(\nu-1)$. Denoting $\beta=\nu-1$, we will thus consider
\beq\label{M-sd}
M(\beta^2, \beta+1,  \be^{-1}, t;1) = \frac 1 \be \sum_{M} \Tpol_M(1+\be,
1+\be)\, t^{\ee(M)}
\eeq
where the sum runs over all planar maps.
One motivation for studying this specialization is that it should coincide with a critical random-cluster model~\cite{sheffield,chen}, in which
one chooses a map $M$ and a subset $S$ of its edges  with probability
proportional to $q^{\cc(S)}\beta^{\ee(S)}t^{\ee(M)}\be^{-\vv(M)}$.
The partition function of this model is
$$
\hM(q,\beta,t):=\sum_{M}t^{\ee(M)}\be^{-\vv(M)}\sum_{S\subseteq
 E(M)}q^{\cc(S)}\beta^{\ee(S)}= q M(q, \be+1, \be^{-1},t;1).
$$
The duality relation~\eqref{dual-M} implies
$$
\hM(q,\be,t)=q \be^{-2} \hM(q, q/\be,t),
$$
and one can thus expect the critical point $\be_c(q)$ to satisfy
$\be_c(q)^2=q$. The critical partition function would then be given
by~\eqref{M-sd}. Note that the connection between criticality and
self-duality has recently been proved for  the regular square
lattice~\cite{beffara-duminil}.

\subsection{Power series}
Let $A$ be a commutative ring and $x$ an indeterminate. We denote by
$A[x]$ (resp. $A[[x]]$) the ring of polynomials (resp. \fps) in $x$
with coefficients in $A$. The coefficient of $x^n$ in a   series $F(x)$ is denoted by
$[x^n]F(x)$. If $A$ is a field, then $A(x)$ denotes the field
of rational functions in $x$.
This notation is generalized to polynomials, fractions
and series in several indeterminates. For instance, the Potts \gf\ of
planar maps $M(q, \nu, w,t;y)$ defined by~\eqref{M-def} belongs to
$\qs[q,\nu,w,y][[t]]$. To lighten notation, we often omit the
dependence of our series in $q$, $\nu$ and $w$, writing for instance
$M(t;y)$, or even $M(y)$.


If $\GK$ is a
field,  a power series $F(t) \in \GK[[t]]$
  is \emm algebraic, (over $\GK(t)$) if it satisfies a
non-trivial polynomial equation $P(t, F(t))=0$ with coefficients in
$\GK$. It is \emm differentially algebraic, if it satisfies a non-trivial polynomial
differential equation $P(t, F(t), F'(t), \ldots, F^{(k)}(t))=0$ with
coefficients in $\GK$.

For a series $F$ in several variables $x,y, \ldots$, we denote by
$F'_x$ the derivative of $F$ with respect to $x$.

\section{Main result, and outline of the proof}
\label{sec:example}

Our study of  coloured maps started in~\cite{bernardi-mbm-alg} with
the following equation with two catalytic variables $x$ and $y$:
\begin{align*}
\gM(x,y)&= 1
+xywt\left((\nu-1)(y-1)+qy\right)\gM(x,y)\gM(1,y)
\nonumber\\
&\ \ \ +xyt(x\nu-1)\gM(x,y)\gM(x,1)\\
&
\ \ \ +xywt(\nu-1)\frac{x\gM(x,y)-\gM(1,y)}{x-1}+xyt\frac{y\gM(x,y)-\gM(x,1)}{y-1}.
\nonumber
\end{align*}
It defines uniquely  a power series in $t$, denoted
$  \gM(x,y) \equiv \gM(q,\nu,w,t;x,y)$, which has polynomial coefficients in
$q, \nu, w, x$ and $y$.
The Potts \gf\ $M(y)$ defined by~\eqref{M-def} is essentially the
specialization $x=1$ of $\bar M(x,y)$. More precisely,
$$
M(y)\equiv M(q,\nu,w,t;y)= w\gM(q,\nu,w,t;1,y)= w\bar M(1,y).
$$
We then  proved that when $q\not =
0, 4$ is of the form $2 +2 \cos (j\pi/m)$, for integers $j$ and $m$,
the series $ M(y)$ satisfies an equation of the form:
\beq\label{polcat}
P(\nu, w,t,y,M(y), M_1, \ldots, M_r)=0,
\eeq
for some polynomial $P$ depending on $j$ and $m$, and $r$ (unknown)
series in $t$,
denoted by $M_1=M(1), \ldots, M_r$,  \emm which are independent of,
$y$~\cite[Cor.~10]{bernardi-mbm-alg}.
We call~\eqref{polcat} a \emm
polynomial equation with one catalytic variable,,~$y$. In particular, when $q=1$, every edge is monochromatic
and~\eqref{polcat} coincides with the standard functional
equation obtained by deleting recursively  the root-edge in planar
maps~\cite{tutte-general}:
 \beq\label{1cat-planaires1}
M(y)= w+ y^2 t\nu   M(y)^2 + \nu t y \, \frac{yM(y)-M_1}{y-1}.
\eeq
The classical way of solving~\eqref{1cat-planaires1} is Brown's \emm
quadratic method,, presented below in Section~\ref{sec:quad-bip}.
It derives from~\eqref{1cat-planaires1} a
polynomial equation satisfied by $M_1$.
The quadratic method was generalized
 in~\cite{mbm-jehanne} to arbitrary (well-posed) polynomial equations
 with one catalytic variable: under minimal assumptions,
their solutions are \emm always algebraic,. Using this approach,
we  proved in~\cite{bernardi-mbm-alg} that for
$q=2+2\cos (j\pi/m)$, the series $M(q,\nu,w,t;y)$ is algebraic over
$\rs(\nu,w,t,y)$ (and even over $\qs(q,\nu,w,t,y)$).
However, we could not obtain an explicit polynomial
equation satisfied by $M(y)$, nor even by $M(1)=M_1$. We only constructed one for the three \emm integer,
values of $q$ of the above form, namely $q=1,2,$ and $3$. And for $q=3$, we
were only able to do it when $\nu=0$, that is, when one counts
\emm proper, 3-colourings. We expect the degree of $M_1$ to grow with
$m$. For $q$ an indeterminate, the series $M_1$ is
transcendental.

The main result of this paper is that when $q$ is an indeterminate,
and consequently  for any real~$q$, the series
$M_1$ is differentially algebraic over $\qs(q,\nu, w,t)$. Moreover, we construct an explicit
differential system that characterizes $M_1$.

\begin{Theorem}\label{thm:ED}
Let $q$ be an indeterminate, $\beta=\nu-1$ and
$$
\bD(t,x)=(q\nu+\be^2)x^2-q (\nu+1 ) x+ \be t ( q-4 ) ( wq+\be ) +q.
$$
There exists a unique triple $(\bP(t,x), \bQ(t,x),\bR(t,x))$ of
 polynomials in $x$ with coefficients in
$\qs[q,\nu,w][[t]]$,
having
 degree $4, 2$ and $2$ respectively in $x$, such that
\beq\label{init-0}
 \begin{array}{ll}
[x^4]\bP(t,x)=1,& \quad \bP(0,x)=x^2(x-1)^2, \\
 {[x^2]}\bR(t,x)=\nu+1-w(q+2\beta), &\quad\bQ(0,x)= x(x-1),
 \end{array}
\eeq
and
\beq\label{de}
\frac {1}{\bQ}\frac{\partial }{\partial t} \left( \frac{ \bQ^2}{\bP \bD^2}\right)=\frac 1{\bR}\frac{\partial }{\partial x} \left( \frac{ \bR^2}{\bP \bD^2}\right).\eeq

Let $\bP_j(t)\equiv P_j$ (resp. $\bQ_j$, $R_j$) denote the coefficient of $x^j$ in
$\bP(t,x)$ (resp. $\bQ(t,x)$, $\bR(t,x)$).
The Potts \gf\ of planar maps,
$M_1\equiv M(q,\nu,  w,t;1)$, can be expressed in terms of the series $\bP_j$
and $\bQ_j$ using $\tM_1:=t^2M_1$ and
\begin{multline}\label{M11-PQ-0}
12  \left( {\beta}^{2}+q\nu \right) \tM_1  + P_3 ^{2}/4
+2 t \left(1+\nu -w(2 \beta+q) \right)P_3  -P_2 +2 Q_0
=4 t \left(1+  w(3 \beta +q) \right) .
\end{multline}
An alternative characterization of $M_1$ is in terms of the derivative
of $\tM_1$:
$$
2\left( {\beta}^{2}+q\nu\right) \tM_1'
+ \left( 1+\nu- w\left( 2\beta+q \right)  \right) P_3 /2-R_1
=2+2 \beta w .
$$
The series $M_1$ is differentially algebraic, that is,  satisfies a
non-trivial differential equation  with respect to the edge variable
$t$.
The same holds for each series $P_j$, $Q_j$ and $R_j$.
\end{Theorem}

Eq.~\eqref{de}  looks like a \emm partial, differential equation in $t$ and $x$, but it
is in fact a system of nine differential equations
in $t$ written in a compact form.
Indeed, its numerator reads:
$$
2Q_t'PD-QP_t'D-2QPD_t'=2R_x'PD-RP_x'D-2RPD_x'.
$$
This is a polynomial in $x$ of degree  8, and each of its nine coefficients  gives a differential equation
(with respect to the variable $t$)
relating the eleven series $P_j$, $Q_j$ and $R_j$. For instance, extracting the coefficient of $x^8$ gives :
\beq\label{ed-simple}
2Q'_2P_4-Q_2P'_4=0.
\eeq
The fact that we only obtain nine equations for eleven unknown series look
alarming, but the initial conditions~\eqref{init-0} give explicitly
the two series $P_4$ and $R_2$.
We also observe that~\eqref{ed-simple}, combined with the initial
conditions~\eqref{init-0}, implies that $Q_2=1$.
In fact, we will prove that the differential system~\eqref{de}, together with its
initial conditions, determines the series $P_j$, $Q_j$ and $R_j$
uniquely.
For instance,
\begin{multline*}
  P_0=-4t+(q^2w^2+16\beta w-4qw+8\beta )t^2\\+2(-\beta q^2w^3+q^3w^3+2\beta qw^2-4q^2w^2+16\beta ^2w+4\beta qw+2\beta ^2-6qw+4\beta )t^3+ O(t^4).
\end{multline*}

Using~\eqref{M11-PQ-0},
one can in principle construct a differential
equation (DE) for $M_1$, but it would probably be very large. However,
we will  work out
some special cases in details, and obtain for instance a DE of order 3 for four-coloured
planar maps (Section~\ref{sec:q=4}).  We
predict the order of $M_1$ to be 5 in general (Section~\ref{sec:simpl-gen}).

Our solution differs significantly from other published results on the Potts
model, in that it ignores the catalytic variable $y$ (which is set to
its natural value 1) and describes the dependence of the series
$M(1)$ in the
size variable $t$. In contrast, the results of~\cite{borot3,guionnet-jones} give a precise
description of the dependence of $M(y)$ in $y$ (in terms of elliptic
functions), but the dependence in the size variable seems to remain
elusive. This is this dependence that we have characterized, in
differential terms.

\subsection{A toy example: uncoloured maps}
\label{sec:toy}
In this section, we illustrate on an example how one can derive differential equations, rather
than polynomial equations, from equations with one catalytic
variable like~\eqref{polcat}. Our example is  Tutte's equation for the enumeration of
planar maps~\cite{tutte-general}, counted by the number of edges (variable $t$) and
the outer degree (variable $y$):
\beq\label{eq-bip}
M(t;y)\equiv M(y) = 1+ ty^2M(y)^2+ ty\, \frac{yM(y)-M_1}{y-1},
\eeq
where $M_1$ stands for $M(1)\equiv M(t;1)$.
 This equation characterizes  $M(y)$ in the ring
 of \fps\ in $t$. Indeed, the coefficient of
$t^n$ can be computed by induction
 on $n$, yielding:
\beq\label{bip-exp}
M(t;y)= 1 +(y+y^2)t +(2y+2y^2+3y^3+2y^4)t^2+O(t^3).
\eeq
The reader can find this expansion, and much more, performed in the accompanying Maple session, available on the web pages
of the authors.

The standard method for
solving~\eqref{eq-bip} is Brown's  quadratic method (see~\cite{brown-square}, or
\cite[Section~2.9]{goulden-jackson} for a modern account). It
constructs from~\eqref{eq-bip} a polynomial
  equation satisfied by $M_1$. We describe this solution in
  Section~\ref{sec:quad-bip} below. We then present in
  Section~\ref{sec:diff-bip} an alternative solution, which derives
from~\eqref{eq-bip} a \emm system of differential equations, satisfied by
$M_1$. This gives the spirit of our
construction of differential equations for $q$-coloured planar maps.
%
%
The spirit, but not the details: in Section~\ref{sec:differences} we
 underline a few important differences between
the solution of our toy example and that of $q$-coloured maps.

\subsubsection{The quadratic method: a polynomial equation for
  $\boldsymbol{M_1}$}
\label{sec:quad-bip}
We first
form in~\eqref{eq-bip} a perfect square, as if we wanted to solve
it for $M(y)$:
\beq\label{carre}
\left( 2ty^2(y-1)M(y)+ty^2-y+1\right)^{ 2}=
(y-1-y^2t)^2-4ty^2(y-1)^2+4t^2y^3(y-1) M_1.
\eeq
Clearly, there exists a unique \fps\ in $t$, denoted below $Y$, than
cancels the left-hand side. Its $n$th coefficient can
be computed by induction, using the expansion~\eqref{bip-exp}
of $M(y)$:
$$
Y=1+t+4t^2+25t^3+O(t^4).
$$
 Let us denote the right-hand side of~\eqref{carre} by
$$
\Delta(y)\equiv \Delta(t;y)= (y-1-y^2t)^2-4ty^2(y-1)^2+4t^2y^3(y-1)
M_1.
$$
 This is a \emm quartic polynomial, in
$y$, with coefficients in $\qs[[t]]$. Since $Y$ cancels the left-hand
side of~\eqref{carre}, it also cancels its right-hand side. Thus
$Y$ is a root of $\Delta(y)$. By differentiating~\eqref{carre} with
respect to $y$, we see that $Y$ is also a root of $\Delta'_y$, and
hence a double root of $\Delta$.

Since $\Delta(y)$ has a double root, its discriminant (with respect to
$y$) vanishes.
This gives a polynomial equation satisfied by $M_1$:
$$
27\,  {t}^{2}M_1  ^{2}+ \left(1 -18
\,t \right) M_1 +16\,t-1=0.
$$
Of course this can be easily solved, and we obtain the well-known \gf\
of planar  maps counted by edges~\cite{tutte-general}:
$$
M_1=\frac{(1-12t)^{3/2}-1+18t}{54\,t^2}.
$$
%

\subsubsection{An alternative approach: a differential system}
\label{sec:diff-bip}
We now describe another consequence of the fact that $\Delta(y)$ has a
double root, this time in terms of  differential equations. The
procedure looks cumbersome,  compared to the neat fact that
the discriminant of $\Delta(y)$ is zero. But, applied to coloured planar
maps, it will yield a system of  differential equations that looks much
nicer than any polynomial
equations we could possibly compute.

Denote $Z=1/Y$.
Since $Y$ is a double root of $\Delta(y)$,  there exists a polynomial
$P(t;y)\equiv P(y)$ of degree 2 such that
$$
\Delta(y)= (1-yZ)^2P(y).
$$
Since $Z$ is a \fps\ in $t$, the same holds for the coefficients of $P(y)$.
Similarly, there exists a polynomial $Q(y)$ of degree 2 such that
\beq\label{DyQ-bip}
\Delta'_y(y)=(1-yZ) Q(y),
\eeq
and a polynomial $R(y)$ of degree 3 such that
\beq\label{DtR-bip}
\Delta'_t(y)=(1-yZ)R(y).
\eeq
We now want to eliminate $\Delta$ and $Z$ in these three equations so
as to find a (differential) relation between $P$, $Q$ and $R$. We
first eliminate $Z$, writing
\beq\label{elimZ-bip}
\frac{{\Delta'_y}^2}{\Delta}= \frac{Q^2}P \quad \hbox{ and } \quad
\frac{{\Delta'_t}^2}{\Delta}= \frac{R^2}P.
\eeq
We now differentiate the first equation with respect to $t$, and the
second one with respect to $y$:
$$
\frac{2\Delta'_y \Delta''_{y,t}\Delta-{\Delta'_y}^2\Delta'_t}{\Delta^2}=
\frac{\partial}{\partial t}\left(  \frac{Q^2}P \right),\qquad\hbox{
} \qquad
\frac{2\Delta'_t \Delta''_{y,t}\Delta-{\Delta'_t}^2\Delta'_y}{\Delta^2}=
\frac{\partial}{\partial y} \left(  \frac{R^2}P\right) .
$$
The ratio of the left-hand sides is $\Delta'_y/\Delta'_t$, which,
according to~(\ref{DyQ-bip}-\ref{DtR-bip}), is also $Q/R$. This gives:
$$
\frac 1 Q \frac{\partial}{\partial t}\left(  \frac{Q^2}P \right)
=\frac 1{R}\frac{\partial}{\partial y} \left(  \frac{R^2}P\right).
$$
Note the striking analogy with our differential system~\eqref{de} for coloured maps.
As in the coloured case, this equation gives, in a compact form, a
system of  differential equations
in $t$ relating the coefficients of $P(y)$, $Q(y)$ and $R(y)$.
We will not discuss which additional properties  are needed to
characterize them 
uniquely, but let us
express the series $M_1$ in terms of them. The first
equation of~\eqref{elimZ-bip} reads
$$
P{\Delta'_y}^2=Q^2\Delta.
$$
This is a polynomial in $y$ of degree 8. Extracting the coefficient of $y^8$ gives the identity
$$
64t^2P_2M_1+16t^2P_2-Q_2^2-64tP_2=0,
$$
which determines $M_1$ in terms of the coefficients $P_j$ and $Q_j$ of
$P(y)$ and $Q(y)$.

\subsection{Some differences between the toy example and coloured
  maps}
\label{sec:differences}
Before embarking on the proof of Theorem~\ref{thm:ED}, we want to
underline a number of differences between the treatment just applied
to uncoloured maps and our
treatment of coloured maps below.
Let us list the main steps of our approach.
\begin{itemize}
\item We take $q=2+2\cos (2\pi/m)$, and  start from a
  polynomial equation in one catalytic variable
  satisfied by $M(y)$, obtained in~\cite{bernardi-mbm-alg}.
This is~\eqref{eq:inv}.
\item This equation has some similarities with~\eqref{carre}: in the
  latter equation, the right-hand side is a polynomial in $y$ with
  coefficients in $\qs[[t]]$, while in~\eqref{eq:inv},  the right-hand side is a
  polynomial in $M(y)$ (or rather, in a  series $I(y)$ which contains the same information
  as $M(y)$), with coefficients in $\qs(q,\nu,w)[[t]]$.
  Hence the roles of $y$ and $M(y)$ are exchanged.
\item We derive from~\eqref{eq:inv}
that a certain polynomial  $\Delta(x)$, of
  degree $2m$ and involving $m-2$ unknown series in $t$, has $m-2$ double roots
  $I_1, \ldots , I_{m-2}$. This is Lemma~\ref{lem:rootsCD}.
\item We derive our differential system  from this property. The
  counterparts of~\eqref{DyQ-bip} and~\eqref{DtR-bip} do not have
  left-hand sides $\Delta'_y$ and $\Delta'_t$, but variants of these
  two polynomials (Proposition~\ref{prop:factor}).
\end{itemize}

\section{Differential system for coloured planar maps}
\label{sec:general}

\subsection{An equation with one catalytic variable}
\label{sec:cat-M}
Our starting point is an equation  of~\cite{bernardi-mbm-alg}. We take $q=2+2\cos
(2k\pi/m)$, with $k$ and $m$ coprime and $0<2k<m$. We write $\beta=\nu-1$.
We introduce the following notation:
\begin{itemize}
\item $I(t,y)\equiv I(q,\nu,w,t;y)$ is a  variant of the \gf \
  $M(y)\equiv M(q,\nu,w,t;y)$ of $q$-coloured planar maps defined
  by~\eqref{M-def}:
\beq\label{I-def}
I(t,y)=tyqM(y)+\frac{y-1}{y}+\frac{ty}{y-1}.
\eeq
\item $N(y,x)$ and $D(t,x)$ are the following (Laurent) polynomials:
\beq\label{N-expr}
\bN(y,x)= \beta(4-q) (\by-1)+( q+2\,\beta ) x -q,
\eeq
with $\by=1/y$, and
\beq\label{D-expr}
\bD(t,x)=(q\nu+\be^2)x^2-q (\nu+1 ) x+ \be t ( q-4 ) ( wq+\be ) +q.
\eeq
\end{itemize}
Let $\Tch_m$ be the $m$th Chebyshev polynomial of the first kind, defined
by
$$
\Tch_m(\cos \theta) = \cos(m\theta).
$$
Then  there exists $m+1$ formal power series in $t$ with coefficients in
$\qs(q,\nu,w)$,
denoted $C_0(t), \ldots$, $C_m(t)$, such that
\beq\label{eq:inv}
\bD(t,I(t,y))^{m/2} \,\Tch_m\left(\frac{\bN(y,I(t,y))}{2\sqrt {\bD(t,I(t,y))}}
\right)= \sum_{r=0}^m
C_r(t) I(t,y)^r.
\eeq
This is a combination of Corollary~10 and Lemma~16 from~\cite{bernardi-mbm-alg}.
We call~\eqref{eq:inv} \emm the invariant equation,, since it is
derived from a certain \emm theorem of invariants,, in Tutte's
terminology~\cite{tutte-chromatic-revisited}.
We find convenient to denote
\beq\label{C-expr}
\bC(t,x)=\sum_{r=0}^m C_r(t)\, x^r.
\eeq
Since $\Tch_m(u)$ is even (resp. odd) in $u$ if $m$ is even
(resp. odd), the left-hand side of~\eqref{eq:inv} only involves \emm
integer, powers of $D$.

As the invariant equation may look
intimidating, let us consider an example.

\medskip
\noindent {\bf Example: bipartite maps.} Let us take $q=2$ (that is,
$m=4$ and $k=1$), $w=1$ and $\nu=0$. Then $M(y)$ is simply the \gf\ of
bipartite maps, counted by the edge number (variable $t$) and the
outer degree (variable $y$). Let
us show that in this case, the invariant equation~\eqref{eq:inv}
coincides with  the standard equation obtained by deleting the root-edge.

With our choice of $q$, $\nu$ and $w$, we have
$$
N(y,x)=-2\by, \qquad
D(t,x)=x^2-2x+2t+2,
$$
while $\Tch_m(u)=\Tch_4(u)=1-8u^2+8u^4$.  The invariant
equation thus reads
$$
\left(I(y)^{2}-2\,I(y) +2\,t+2\right) ^{2}-8\by^2
  (I(y)^{2}-2\,I(y)+2\,t+2)+8\,\by^4=
\sum _{r=0}^{4} C_r\, I(y)^r.
$$
In this equation, let us replace $I(y)$ by its
expression~\eqref{I-def} in terms of $M(y)$ (with $q=2$),
and then
expand the equation in the neighbourhood of $y=1$.
Saying that the coefficients of $(y-1)^{-4}, \ldots, (y-1)^0$ must vanish gives the values of  the  series $C_r$:
\beq\label{Cr-bip}
  C_4=1, \quad
C_3=-4,\quad
C_2=4t,\quad
C_1=8(1+t)
\eeq
and
$$
C_0=
-4-40\,t-4\,{t}^{2}+32\,{t}^{2}M (1).
$$
In the invariant equation, let us  replace  each series $C_r$ by its
expression: we obtain
$$
{y}^{2}t ( y^2-1 )  M ( y )^{2}
+ (1 -{y}^{2}+{y}^{2}t ) M ( y )
-t {y}^{2}M ( 1 )+  y^2-1 =0.
$$
Equivalently,
$$
M(y)= 1 + t y^2 M(y)^2+ty^2\, \frac{M(y)-M(1)}{y^2-1},
$$
which can be obtained by deleting the root-edge in a bipartite
map~\cite{tutte-general}. \qee

\bigskip
Let us return to the general case $q=2+2\cos(2k\pi/m)$.
By expanding~\eqref{eq:inv} near $y=1$, we can express the series
$C_r(t)$ in terms of the derivatives of $M$ with respect to $y$,
evaluated at $y=1$. We will need the expressions of $C_m, \ldots,
C_{m-3}$. Note that the first three are completely explicit in terms
of $q$, $\nu$, $t$, $w$ and $m$. The fourth one involves the
unknown series $M(1)$.
\begin{Lemma}
 \label{lem:initial-Cr}
Denote $\beta=\nu-1$, $\delta=\beta^2+q\nu$ and
$
\gamma= \frac{q+2\beta}{2\sqrt{\delta}}.
$
We have:
\begin{equation*}
\begin{split}
C_m&= \delta ^{m/2}\ \Tch_m(\gamma),\\
C_{m-1}&= -\frac q 4 \delta ^{(m-3)/2} \left(2 m(\nu+1)
 \sqrt{\delta}\ \Tch_m(\gamma)+\beta(q-4) \Tch'_m(\gamma)\right),\\
C_{m-2} &= \frac 1 8\, \delta ^{(m-5)/2}\left(
2 m a_0
 \sqrt{\delta}\ \Tch_m(\gamma)
 +q\beta(q-4)a_1\Tch'_m(\gamma)\right),\\
C_{m-3} &= \frac q{24} \delta^{(m-7)/2}
 \left(2 mb_0 \sqrt{\delta}\ \Tch_m(\gamma)+\beta(q-4)b_1 \Tch'_m(\gamma)\right)
+\frac{q(q-4)\beta  t^2}2 M(1)\, \delta^{(m-1)/2}\Tch'_m(\gamma),
\end{split}
\end{equation*}
with
\begin{equation*}
\begin{split}
 a_0&=2\,t\beta\,\delta \left( q-4 \right) \left( \beta+wq \right) +q
 \left( m-1 \right) \left( \left( 1+{\nu}^{2} \right) q-2\,{\beta}^{
2} \right) ,
\\
a_1&=2\,t\delta\, \left( \nu+1-w \left( 2\,\beta+q \right) \right) +q
 \left( m-1 \right) \left( \nu+1 \right) ,
\\
b_0&=6\,tw \left( q-4 \right) \beta\,\delta\, \left( \left( {\beta}^{2}-q
 \right) m+q \left( \nu+1 \right) \right) -6\,  {\beta}^{2}t\delta \left( m-1 \right)
 \left( \nu+1 \right) \left( q-4 \right)
\\ &\hskip 30mm -q \left(
m-1 \right) \left( m-2 \right) \left( \nu+1 \right) \left( \left(
{\nu}^{2}-\nu+1 \right) q-3\,{\beta}^{2} \right),
 \\
b_1&=6\,tw\delta\, \left( q( m-1 ) ( {\nu}^{2}+2\,\nu+q-3
) - \left( 2\,\beta+q \right) \delta \right)\\
&\hskip 20mm -6\,t \delta \left( m-1
 \right) \left( q ( 1+{\nu}^{2} ) -2\,{\beta}^{2}
 \right) +q ( m-1 ) ( m-2 ) \left( {\beta}^{2}
-( {\nu}^{2}+\nu+1 ) q \right).
\end{split}
\end{equation*}
\end{Lemma}

\smallskip
\noindent{\bf Example: bipartite maps (continued).} When $k=1$,
$m=4$, $q=2$, $w=1$ and $\nu=0$, we have $\beta=-1$, $\delta=1$ and
$\gamma=0$. Moreover, $\Tch_m(\gamma)=1$ and $\Tch'_m(\gamma)=0$. We
can then check that the above lemma specializes to~\eqref{Cr-bip}. In
particular, $C_1=C_{m-3}$ does not involve $M(1)$ because $\Tch_m'(\gamma)=0$. \qee

\begin{proof}[Proof of Lemma~\ref{lem:initial-Cr}.]
 We multiply the invariant equation~\eqref{eq:inv} by $(y-1)^m$, so that it
 becomes non-singular at $y=1$, and expand it around $y=1$. We
 extract successively the coefficients of $(y-1)^0, \ldots, (y-1)^3$
 and obtain in this way expressions of $C_m, \ldots, C_{m-3}$. For
 instance, since
 \begin{align*}
  (y-1) I(t,y)&= t+O(y-1),\\
(y-1)N(y, I(t,y))&= (q+2\beta) t +O(y-1),\\
(y-1)^2 D(t,I(t,y))&= \delta t^2+O(y-1),
 \end{align*}
extracting the coefficient of $(y-1)^0$ gives
$$
\delta^{m/2} t^m \Tch_m(\gamma)= C_m t^m,
$$
from which the expression of $C_m$ follows. As we extract the
coefficients of higher powers of $(y-1)$, derivatives of $\Tch_m$, taken
at the point $\gamma$,
occur. We systematically express them in terms of $\Tch_m(\gamma)$ and
$\Tch_m'(\gamma)$ using the differential equation
$$
(1-u^2)\Tch''_m(u)-u\Tch'_m(u)+m^2\Tch_m(u)=0.
$$
Similarly, derivatives of $(1-y)I(t,y)$ with respect to $y$ occur, and
this is why the expression of $C_{m-3}$ involves the series $M(1)$.
\end{proof}

\subsection{Some special values of $\boldsymbol{y}$}

The invariant equation~\eqref{eq:inv} reads
\beq\label{eqinv-Rat}
\Rat(t,y,I(y),C_0, \ldots,C_m)=0
\eeq
where $\Rat$ is the following rational function, with coefficients in
$\qs(q,\nu, w)$:
$$
\Rat(t,y,x,c_0, \ldots, c_r)=
D(t,x)^{m/2}\Tch_m\left(\frac{N(y,x)}{2\sqrt{D(t,x)}}\right)-\sum_{r=0}^m c_r x^r.
$$
Observe that the first term of $\Rat$
is the only one that involves $y$. In our
toy model (Section~\ref{sec:toy}), we were dealing with~\eqref{carre}, which read  $
\Rat(t,y,M(y),M_1)=0$ with
$$
\Rat(t,y,x,c)=
\left( 2ty^2(y-1)x+ty^2-y+1\right)^{ 2}-
(y-1-y^2t)^2+4ty^2(y-1)^2-4t^2y^3(y-1)c.
$$
There, the first term was the only one that depended on $x$. To solve
this toy equation, we considered a series $Y\equiv Y(t)$ cancelling this
term, or equivalently, satisfying
$$
\Rat'_x(t,Y,M(Y),M_1)=0,
$$
where the derivative is taken with respect to the third variable of $\Rat$.
Similarly, we are going to solve~\eqref{eqinv-Rat} by cancelling
$\Rat'_y$ (this change from $x$ to $y$ explains why we wrote in
Section~\ref{sec:differences} that the roles of $y$ and $M(y)$ are exchanged). More precisely, we focus on
$$
\Tch'_m\left( \frac{N(y,I(t,y))}{2\sqrt {D(t,I(t,y))}}\right)=0.
$$
Denoting $n=\lfloor m/2\rfloor$, the set of roots of $\Tch_m'$ is easily
seen to be:
\beq\label{roots}
\left\{ \cos(j\pi/m), 0<j<m\right\} =
\left\{\pm\cos(j\pi/m), j=1,\ldots, n\right\}.
\eeq
Hence we are interested in the equation
\begin{equation}\label{eqj-square}
\bN(Y,I(t,Y))^2 =4\cos(j\pi/m)^2\bD(t,I(t,Y)),
\end{equation}
for some $j \in \llbracket 1, n\rrbracket$.
 \begin{Lemma}\label{lem:count}
 Let us assume
$k=1$, so that $q=2+2\cos (2\pi/m)$.
There exist  $m-2$ distinct formal power series in $t$, denoted $Y_1,
 \ldots, Y_{m-2}$, that have coefficients in $\rs(\nu,w)$,
 constant term $1$, and satisfy~\eqref{eqj-square}
for some $j \in \llbracket 1, n\rrbracket$.

Let us denote $I_i(t):=I(t,Y_i)$.  This is a \fps\ in $t$ with
coefficients in $\rs(\nu,w)$. The $m-2$ series
$I_i(t)$ are distinct, and for $1\le i \le m-2$,
$$
I_i(0)\not\in\{0, 1\}, \qquad \bD(t,I_i)\not = 0,
\qquad \frac{\partial I}{\partial y}(t,Y_i) \not = 0.
$$
 \end{Lemma}

\noindent{\bf Example: one-coloured maps.} We take $q=1$, that is, $m=3$ (and
$n=1$). The only relevant value of $j$ is $1$, and we thus want to solve
$$
N(Y,I(t,Y))^2=D(t, I(t,Y)),
$$
for $Y$ a series in $t$ with constant term $1$. We write $Y=1+tZ$.
Using the definitions~\eqref{N-expr} and~\eqref{D-expr} of $N$ and $D$, the
equation satisfied by the series $Z$ can be written as
$$
Z=\nu+t\Pol\left(\nu,w,t,Z,M(1+tZ)\right),
$$
for some polynomial $\Pol$ with integer coefficients. This equation
allows us to compute by induction on $n$ the coefficient of $t^n$ in $Z$,
and shows the existence and uniqueness of $Z$ (and $Y$). More precisely,
$$
Y=1+\nu t+2\nu^2(w+1)t^2+\nu^3(5+14w+6w^2)t^3+O(t^4).
$$
This gives
\begin{align*}
  I_1(t):= I(t,Y)&=  1 /\nu +(\nu-w-1) t +O(t^2),
\\
D(t,I_1)&=\left(1-1/\nu\right)^2+O(t),
\\
\frac{\partial I}{\partial y}(t,Y)&=-\frac 1{\nu^2 t} +O(1),
\end{align*}
so that the last three properties of the lemma hold. \qee

\medskip
\begin{proof}[Proof of Lemma~\ref{lem:count}]
 We denote
$q_j=4\cos(j\pi/m)^2=2+2\cos(2j\pi/m)$. Observe that $q
=q_1$.
 We assume that $m=2n+1$ is odd. The proof is similar in the
even case. 

\medskip
\paragraph{\bf Existence of the series $\boldsymbol{Y_i}$.}
We are looking for series $Y\equiv Y(t)$ solutions of~\eqref{eqj-square},
of the form $Y=1+tZ$ for some series $Z\equiv Z(t)$.
Multiplying~\eqref{eqj-square} by $Z^2$ gives the equation
$
\Phi_j(t,Z)=0,
$
where
$$
\Phi_j(t,z)=z^2\bN(1+tz,I(t,1+tz))^2-q_jz^2\bD(t,I(t,1+tz)).
$$
We are interested in non-zero solutions (since for $Z=0$, we have
$Y=1$ and $I(t,Y)$ is not well defined).
Observe that $zI(t,1+zt)$ is a \fps\ in $t$ with coefficients in
$\rs[\nu,w,z]$. Hence the same holds for $z\bN(1+tz,I(t,1+tz))$ and
$z^2\bD(t,I(t,1+tz))$, and finally for $\Phi_j(t,z)$.  Recall that
$\nu=\beta+1$. Expanding $I$, $N$ and $D$  at  first order in $t$ gives
\begin{align}
zI(t,1+tz) &= 1+O(t),\label{I-exp}
\\
z\bN(1+tz,I(t,1+tz))&= q+2\beta -qz+O(t),\nonumber
\\
z^2\bD(t,I(t,1+tz))&=q\beta+q+\beta^2-q (\beta+2)z +qz^2 +O(t).
\label{D-exp}
\end{align}
 The equation $\Phi_j(t,Z)=0$ thus reads
\beq\label{eq-Phi-S}
P_j(Z) +t S_j(t,Z)=0,
\eeq
where
$$
P_j( z)=\left( q+2\beta -qz\right)^2-q_j\left(q\beta+q+\beta^2-q
  (\beta+2)z +qz^2\right)
$$
and $S_j(t,z)$ is a power series in $t$ with coefficients in $\rs[\nu, w,
z]$.
By~\eqref{eq-Phi-S}, the coefficient of $t^0$ in~$Z$, denoted $z_0$,
must satisfy $P_j(z_0)=0$.
Equivalently, $z_0=1+\beta v$ with
\begin{equation}\label{eq-quadratic}
q(q-q_j)v^2+q(q_j-4)v+4-q_j =0.
\end{equation}
This equation has degree  1 in $v$ if $q_j=q$, that is if $j=1$, and
degree 2 otherwise.
 Given the expressions of $q_j$ and $q=q_1$,
its roots  are found to be
%
%
$$v_j^{+}=\frac{\sin(j\pi/m)}{2\cos(\pi/m)\sin((j+1)\pi/m)}
\quad \hbox{ and } \quad
v_j^{-}=\frac{\sin(j\pi/m)}{2\cos(\pi/m)\sin((j-1)\pi/m)},
$$
for $j\not = 1$. For  $j=1$ the only root is $v_1^+=1/q$.
The roots $v_j^+$ and $v_j^-$ are distinct
(we use here the assumption that $m$ is odd).

Now having fixed $z_0=1+\beta v$ with $v=v_j^{\pm}$, let us return
to~\eqref{eq-Phi-S}
and extract from it the coefficient of $t^p$, for $p\ge 1$. This gives
$$
P'_j(z_0)[t^p]Z+ \big(\hbox{expression involving only } [t^i]Z \hbox{ for }
i<p\big) =0.
$$
Since $v$ is not
a double root of $P_j$, this allows us to compute
the coefficient of $t^p$ in $Z$ by induction on $p$, and thus
determines $Z$ completely. Moreover, $[t^p]Z$ has a rational
expression in $\nu$ and $w$ (with real coefficients).

It is easy to see that
$$
0 < v_1^+ < \cdots <v_n^+ =\frac 1{2\cos(\pi/m)}<v_{n}^-<\cdots< v_{2}^-,
$$
so that we have exactly  $(m-2)$ distinct values $v_j^\pm$.
Let us denote them
$
v_1, \ldots, v_{m-2}.
$
They give rise to exactly $(m-2)$ distinct series $Y$ satisfying $Y(0)=1$
and~\eqref{eqj-square} for some $j\in \llbracket 1, n\rrbracket$. We
denote them $Y_1,
\ldots, Y_{m-2}$, with $Y_i=1+(1+\beta v_i)t+O(t^2)$.
We note that $v_i$  is a real number, while
$\beta=\nu-1$ where $\nu $ is a formal parameter. Hence our $m-2$
series $Y_i$ all differ from the constant $1$, and $I(t,Y_i)$ is
well-defined.
We have thus proved the first statement of the lemma.

\medskip

\medskip
\paragraph{\bf Properties of $\boldsymbol{Y_i}$ and
  $\boldsymbol{I_i:=I(t,Y_i)}$.}
Let us  prove that the series $I_i:=I(t,Y_i)$ are
distinct. Returning to~\eqref{I-exp} gives $I_i=(1+\beta
v_i)^{-1}+O(t)$. Hence the $(m-2)$ constant terms $I_i(0)$ are distinct since  the $(m-2)$ numbers $v_1, \ldots, v_{m-2}$ are distinct.
Moreover  $I_i(0)$ is obviously non-zero, and it is
distinct from 1 because $v_i\neq 0$.

Let us now prove that
$D(t,I_i)\neq 0$. Returning to~\eqref{D-exp}, we find
\beq\label{D-constant-term}
(1+\beta v_i)^2 D(t,I_i)=\beta^2(qv_i^2-qv_i+1)+O(t)=\left( \beta \frac{\sin
    (\pi/m)}{\sin ((j\pm 1)\pi/m)}\right)^2+O(t),
\eeq
where the value of the denominator depends on whether $v_i$ is of the
form $v_j^+$ or $v_j^-$. At any rate, this is non-zero.

We finally check that
$$
\displaystyle \frac{\partial I}{\partial y}(t,Y_i)=
-\frac{1}{t(1+\be v_i)^2}+O(1)$$
is non-zero, and this completes the proof of the lemma.
\end{proof}

\subsection{Some polynomials with common roots}
We still assume $q=2+2\cos (2\pi/m)$.
Recall that  $D(t,x)$
and $C(t,x)$ are polynomials in $x$ with coefficients in $\qs(q, \nu,
w)[[t]]$,  given by~\eqref{D-expr} and~\eqref{C-expr}
respectively. Their degrees in $x$ are $2$ and $m$, respectively.  The
coefficients of $D$ are explicit, but those of $C$ are unknown, apart from the three leading ones (Lemma~\ref{lem:initial-Cr}). We
now prove  that several polynomials, related
to $C(t,x)$ and $D(t,x)$, admit the series $I_i$ of Lemma~\ref{lem:count} as
common roots.

\begin{Lemma}\label{lem:rootsCD}
Each of the series $I_i:=I(t,Y_i)$ defined by Lemma~{\rm\ref{lem:count}} satisfies
\begin{align}
 \bC(t,I_i)^2&=\bD(t,I_i)^m,\label{eq1}
\\
\bD(t,I_i) \frac{\partial \bC}{\partial x}(t,I_i)&= \frac m 2 \bC(t,I_i)
\frac{\partial \bD}{\partial x}(t,I_i),\label{eq2}
\\
\bD(t,I_i) \frac{\partial \bC}{\partial t}(t,I_i)&= \frac m 2 \bC(t,I_i) \frac{\partial \bD}{\partial t}(t,I_i).\label{eq3}
  \end{align}
Since $D(t,I_i)\not = 0$, it follows from~\eqref{eq1} and~\eqref{eq2}
that $I_i$ is actually a double root of $C^2-D^m$.
\end{Lemma}
\begin{proof}
Let us denote
$$
G(t,y,x)=\frac{N(y,x)}{2\sqrt{D(t,x)}}.
$$
Recall  that $D(t,I_i)$ is a \fps\ in $t$ with coefficients in
$\rs(\nu,w)$. By~\eqref{D-constant-term}, its constant term is of the
form $\beta^2 r^2/(1+\beta v)^2$, for some  real
numbers $r\not = 0$ and $v$. So we can
 define $\sqrt{D(t,I_i)}$
as a \fps\ in $t$ with coefficients in $\rs(\nu,w)$.

By construction of the series $Y_i$, we have
$G(t,y,I_i)=\pm \cos(j\pi/m)$, for some $j\in \llbracket 1, n\rrbracket$.  Given that $\Tch_m(\cos x)= \cos(mx)$,
each series $Y_i$ satisfies
$$
\Tch_m\big(G(t,Y_i,I(t,Y_i))\big)= \vareps \quad \hbox{and} \quad
\Tch'_m\big(G(t,Y_i,I(t,Y_i))\big)= 0,
$$
where $\vareps=\pm1$.
Hence the invariant equation~\eqref{eq:inv} gives
\beq\label{eq:1}
\bC(t,I_i)=\vareps \bD(t,I_i)^{m/2}
\eeq
from which~\eqref{eq1} follows.

 Let us now differentiate the invariant equation~\eqref{eq:inv} with respect to $y$:
\begin{multline*}
 \frac m 2 \bD(t,I(t,y))^{m/2-1} \frac{\partial \bD}{\partial x}(t,I(t,y)) \frac {\partial I}{ \partial y}(t,y)
\,\Tch_m\big(G(t,y,I(t,y))\big)
\\
+\bD(t,I(t,y))^{m/2}\left( \frac{\partial G}{\partial y}(t,y,I(t,y))+ \frac{\partial G}{\partial x}(t,y,I(t,y)) \frac {\partial I}{ \partial y}(t,y)\right)
\,\Tch_m'\big(G(t,y,I(t,y))\big)
\\= \frac{\partial \bC}{\partial x} (t,I(t,y))\frac {\partial I}{ \partial y}(t,y),
\end{multline*}
or, if we want to keep things in a reasonable volume,
$$
\frac m 2 \bD^{m/2-1} \bD'_x I'_y \,\Tch_m(G)+ \bD^{m/2}\left( G'_y+ G'_x
 I'_y\right) \Tch'_m(G)= \bC'_x I'_y.
$$
When $y=Y_i$, then $\Tch_m(G)=\vareps$, $\Tch'_m(G)=0$ and $I'_y \not = 0$
(by Lemma~\ref{lem:count}), so that
\beq\label{eq:2.0}
 \frac m 2 \vareps\bD^{m/2-1} \bD'_x = \bC'_x.
\eeq
Combined with~\eqref{eq:1}, this gives~\eqref{eq2}.

Similarly, differentiating the invariant equation with respect to $t$ gives
$$
\frac m 2 \bD^{m/2-1} \left( \bD'_t+\bD'_x I'_t \right)
\Tch_m(G)+ \bD^{m/2}\left( G'_t+ G'_x I'_t\right) \Tch'_m(G)= \bC'_t + \bC'_x I'_t.
$$
When $y=Y_i$, then $\Tch_m(G)=\vareps$, $\Tch'_m(G)=0$, and $\bC'_x$ is given by~\eqref{eq:2.0}. Hence
$$
\frac m 2 \vareps \bD^{m/2-1} \bD'_t= \bC'_t .
$$
Combined with~\eqref{eq:1}, this gives~\eqref{eq3}.
\end{proof}

\begin{Proposition}\label{prop:factor}
 Let $I_i:=I(t,Y_i)$ be the $(m-2)$
series defined in  Lemma~{\rm\ref{lem:count}}.
 There exists a  triple $(\hP(t,x), \hQ(t,x),\hR(t,x))$ of
 polynomials in $x$ with coefficients in $\rs(\nu,w)[[t]]$, having
 degree at most $4$, $2$ and $2$ respectively in $x$, such that
 \begin{align}
 \bC(t,x)^2-\bD(t,x)^m&= \displaystyle \hP(t,x)
 \prod_{i=1}^{m-2}(x-I_i)^2,
\label{factor1}
\\
\bD(t,x)\bC'_x(t,x)- \frac m 2 \bD'_x(t,x)\bC(t,x)&= \displaystyle \hQ(t,x) \prod_{i=1}^{m-2}(x-I_i),\label{factor2}
\\
\bD(t,x)\bC'_t(t,x)- \frac m 2 \bD'_t(t,x)\bC(t,x)&= \displaystyle \hR(t,x) \prod_{i=1}^{m-2}(x-I_i).\label{factor3}
 \end{align}
\end{Proposition}

\begin{proof}
The polynomial $\bC(t,x)^2-\bD(t,x)^m$ has degree
(at most) $2m$ in $x$, and admits each of the $(m-2)$ distinct series
$I_i$ as a double root (Lemma~\ref{lem:rootsCD}).
Thus, there exists a polynomial $\hP(t,x)$
of degree at most~4 such that~\eqref{factor1} holds.
Let us write
\beq\label{hP-expr}
\hP(t,x) =\left(\bC(t,x)^2-\bD(t,x)^m\right)
\prod_{i=1}^{m-2} \frac1{I_i^2(1-xI_i^{-1})^2}.
\eeq
Recall that $C$ and $D$ are polynomials in $x$ with coefficients in
$\rs(\nu,w)[[t]]$, that
the series $I_i$ belong to  $\rs(\nu,w)[[t]]$ and have a non-zero
constant term
(Lemma~\ref{lem:count}). Expanding  in $x$ the above expression thus  proves
that the coefficients of $\hP$ also belong to
$\rs(\nu,w)[[t]]$.

Similarly, $\bD (t,x) \bC'_x(t,x) -\frac m 2 \bD'_x (t,x) \bC(t,x)$ is a polynomial of degree at most $m$ in $x$, since
$\bD (t,x) \bC'_x(t,x)$ and $\frac m 2 \bD'_x (t,x) \bC(t,x)$ are
polynomials of degree  $m+1$ having the same leading
coefficient.
Hence by~\eqref{eq2}, there exists a polynomial $\hQ(t,x)$ in
$\rs(\nu,w)[[t]][x]$, of degree at most~2, such that~\eqref{factor2} holds.

Finally, given that $\bD'_t(t,x)$ has degree 0 in $x$, and
$\bC'_t(t,x)$ degree at most $m-2$ (Lemma~ \ref{lem:initial-Cr}), the
polynomial $\bD (t,x) \bC'_t(t,x) -\frac m 2 \bD'_t (t,x) \bC(t,x)$
has degree at most $m$. Equation~\eqref{eq3} implies the existence of
a polynomial $\hR(t,x)$ in $\rs(\nu,w)[[t]][x]$, of degree at most~2,
such that~\eqref{factor3} holds.
\end{proof}

\subsection{Differential system}
\label{sec:diff-general}

We still assume  that $q=2+2\cos
(2\pi/m)$.
We will prove that the series $\hP$, $\hQ$ and $\hR$ of Proposition~\ref{prop:factor},
suitably normalized into series $P$, $Q$ and $R$, satisfy all
equations of  Theorem~\ref{thm:ED}.  We repeat this theorem  for
convenience.

\begin{Theorem*}
Let $q$ be an indeterminate, $\beta=\nu-1$ and
\beq\label{D-encore}
\bD(t,x)=(q\nu+\be^2)x^2-q (\nu+1 ) x+ \be t ( q-4 ) ( wq+\be ) +q.
\eeq
There exists a unique triple $(\bP(t,x), \bQ(t,x),\bR(t,x))$ of
 polynomials in $x$ with coefficients in $\qs[q,\nu,w][[t]]$, having
 degree $4, 2$ and $2$ respectively in $x$, such that
\beq\label{init-general}
 \begin{array}{ll}
   [x^4]\bP(t,x)=1,& \quad\bP(0,x)=x^2(x-1)^2,\\
  {[x^2]}\bR(t,x)=\nu+1-w(q+2\beta), &\quad\bQ(0,x)= x(x-1),
 \end{array}
\eeq
and
\beq\label{de-encore}
\frac {1}{\bQ}\frac{\partial }{\partial t} \left( \frac{ \bQ^2}{\bP
    \bD^2}\right)=\frac 1{\bR}\frac{\partial }{\partial x} \left(
  \frac{ \bR^2}{\bP \bD^2}\right).
\eeq

Let $\bP_j(t)\equiv P_j$ (resp. $\bQ_j$, $R_j$) denote the coefficient
of $x^j$ in
$\bP(t,x)$ (resp. $\bQ(t,x)$, $\bR(t,x)$).
 The Potts \gf\ of planar maps,
$M_1\equiv M(q,\nu, t, w;1)$, can be expressed in terms of the series
$\bP_j$ and $\bQ_j$
using $\tM_1:=t^2M_1$ and
\begin{multline}\label{M11-PQ}
12  \left( {\beta}^{2}+q\nu \right) \tM_1  +
 P_3 ^{2}/4
+2 t \left(1+\nu -w(2 \beta+q) \right)P_3  -P_2
 +2 Q_0
=4 t \left(1+  w(3 \beta +q) \right) .
\end{multline}
An alternative characterization of $M_1$ is in terms of the derivative
of $\tM_1$:
\beq\label{M1-alt}
2\left( {\beta}^{2}+q\nu\right) \tM_1'
+ \left( 1+\nu- w\left( 2\beta+q \right)  \right) P_3 /2-R_1
=2+2 \beta w .
\eeq
The series $M_1$ is differentially algebraic, that is,  satisfies a
non-trivial differential equation  with respect to the edge variable
$t$.
The same holds for each series $P_j$, $Q_j$ and $R_j$.
\end{Theorem*}

\medskip
\noindent
{\bf A differential equation relating {\boldmath$\hP, \hQ$} and {\boldmath$\hR$}.}
We first prove that~\eqref{de-encore} holds for $\hP, \hQ$ and $\hR$. This results from the elimination of $\bC$ and $\prod_i(x-I_i)$ in the three
equations of Proposition~\ref{prop:factor}. The derivation is
analogous to that of Section~\ref{sec:diff-bip}.
Let us first eliminate the product, in such a way that $\hQ^2/(\hP D^2)$ and
$\hR^2/(\hP D^2)$ naturally appear:
$$
\frac{\left(DC'_x-\frac m 2 D'_xC\right)^2}{D^2(C^2-D^m)}= \frac{\hQ^2}{\hP
  D^2}
\qquad \hbox{ and } \qquad
\frac{\left(DC'_t-\frac m 2 D'_tC\right)^2}{D^2(C^2-D^m)}= \frac{\hR^2}{\hP
  D^2}.
$$
Let us differentiate the first equation with respect to $t$, and the
second one with respect to $x$.
The ratio of the two resulting identities
is
$$
\frac{DC'_x-\frac m 2 D'_xC}{DC'_t-\frac m 2 D'_tC},
$$
which, according to the last two equations of Proposition~\ref{prop:factor}, is
$\hQ/\hR$. This gives
 \beq\label{ED-hat}
\frac{1}{\hQ} \frac{\partial}{\partial t} \left(\frac{ \hQ^2}{\hP\bD^2}\right)
=\frac{1}{\hR} \frac{\partial}{\partial x}
\left(\frac{\hR^2}{\hP\bD^2}\right).
\eeq
This equation coincides with~\eqref{de-encore} but with hats over the letters
$P, Q$ and $R$.  The series $\bP, \bQ$ and $\bR$ occurring in the
theorem will simply be normalizations of $\hP, \hQ$ and $\hR$ by
multiplicative constants independent from $t$ and $x$.

\medskip
\noindent {\bf The leading coefficients of {\boldmath$\hP,\hQ,\hR$}.}
Proposition~\ref{prop:factor} gives
\begin{align*}
  \displaystyle [x^4]\hP(t,x)&=  [x^{2m}]\left(\bC(t,x)^2-\bD(t,x)^m\right),\\
 {[x^2]}\hQ(t,x)&=
 [x^m]\left(\bD(t,x)\bC'_x(t,x)- \frac m 2 \bD'_x(t,x)\bC(t,x)\right),\\
 {[x^2]}\hR(t,x)&= 
 [x^m]\left(\bD(t,x)\bC'_t(t,x)- \frac m 2 \bD'_t(t,x)\bC(t,x)\right).
\end{align*}
Recall that $\bD$ is defined  by~\eqref{D-encore}, and that
Lemma~\ref{lem:initial-Cr} gives the
leading coefficients of
$C(t,x)$. This allows to determine the leading coefficients of $\hP,\hQ,\hR$:
\beq \label{tCPQR}
\begin{cases}
 \hP_4:=[x^4] \hP(t,x)&= \delta^m \left(\Tch_m(\gamma)^2-1\right) = \frac q{m^2} ( q/ 4 -1) \delta^{m-1} \Tch'_m(\gamma)^2,
\\
\hQ_2:=[x^2]\hQ(t,x)&= q \beta (q/4- 1 )\delta^{(m-1)/2}
\Tch'_m(\gamma),
\\
\hR_2:=[x^2]\hR(t,x)&=  q \beta (q/4-1) \delta ^{(m-1)/2}\Tch'_m(\gamma)(\nu+1-w(q+2\beta)),
\end{cases}
\eeq
where $\delta=\beta^2+q\nu$ and
$\gamma= (q+2\be)/(2\sqrt \delta)$.
 In the expression of $\hP_4$, we have used the fact that, for all $x$,
\beq\label{eT01}
(u^2-1){\Tch'_m}(u)^2=m^2\left(\Tch_m(u)^2-1\right).
\eeq
Observe that these coefficients  are independent of $t$.
Let us define
\beq\label{PQR-nohat}
\bP(t,x)=\frac{\hP(t,x)}{\hP_4}, \qquad
\bQ(t,x)=\frac{\hQ(t,x)}{\hQ_2}
\qquad \hbox{ and } \qquad \bR(t,x)=\frac{\hR(t,x)}{\hQ_2}.
\eeq
 We have intentionally  normalized $\hR(t,x)$ by $\hQ_2$ rather than $\hR_2$.
 It then follows from~\eqref{ED-hat} that $\bP$, $\bQ$, $\bR$
satisfy the differential system~\eqref{de-encore}. Moreover
\beq\label{in}
[x^4]\bP(t,x)=1,  \qquad [x^2]\bQ(t,x)=1 \qquad \hbox{and} \qquad
[x^2]\bR(t,x)=\nu+1-w(q+2\beta),
\eeq
so that the left-hand side of the initial
conditions~\eqref{init-general} hold.
(We do not give the value of $[x^2]\bQ(t,x)$ in the
statement of the
theorem because it is a consequence of the differential system and the
initial conditions, as will be seen in
Section~\ref{sec:unique-general}.)

\medskip

\noindent
{\bf The case {\boldmath$t=0$}.}
Let us now establish the right-hand side of the initial
conditions~\eqref{init-general} by  determining the values $\bP(0,x)$
and $\bQ(0,x)$. Proposition~\ref{prop:factor} gives
\begin{align}
\bC(0,x)^2-\bD(0,x)^m&= \displaystyle \hP(0,x) \prod_{i=1}^{m-2}(x-I_i(0))^2,\label{e1t0}
\\
\bD(0,x)\bC'_x(0,x)- \frac m 2 \bD'_x(0,x)\bC(0,x)&= \displaystyle \hQ(0,x) \prod_{i=1}^{m-2}(x-I_i(0)).\label{e2t0}
 \end{align}
We will now combine our knowledge of $\bC(0,x)$ obtained
in~\cite{bernardi-mbm-alg} with the properties of the series $I_i(t)$
gathered in Lemma~\ref{lem:count} to determine $\hP(0,x)$ and $\hQ(0,x)$.

Recall that $C_r(t) $ denotes the coefficient of $x^r$ in $C(t,x)$
(see~\eqref{C-expr}). The constant term of the series $C_r(t)$ has been determined
in~\cite[Lemma~16]{bernardi-mbm-alg}. With the notation used there,
$C_r(0)=L_r(0;1)$, where
$$
\sum_{r=0}^m L_r(t;y) x^r= \bD(t,x)^{m/2}
\,\Tch_m \left(
\frac{\bN(y,x)}{2\sqrt{\bD(t,x)}}
\right),$$
where $\bN$ and $\bD$ are given by~\eqref{N-expr} and~\eqref{D-expr}.
Hence
\beq\label{C0-explicit}
C(0,x)\equiv \sum_{r=0}^m C_r(0) x^r= \hD^{m/2}
\,\Tch_m \left( \frac{\hN}{2\sqrt{\hD}}\right),
\eeq
where $\hN=\bN(1,x)=(q+2\beta)x -q$ and
$\hD=\bD(0,x)=(q\nu+\be^2)x^2-q (\nu+1 ) x+ q$.
Thus we first want to factor
$$
C(0,x)^2-D(0,x)^m= \hD^m\left( \Tch_m\left( \frac{\hN}{2\sqrt{\hD}}\right)^2-1\right).
$$
Denote $n=\lfloor m/2\rfloor$. By~\eqref{eT01} and~\eqref{roots},
\beq\label{T2m1}
\Tch_m(u)^2-1=\frac 1 {m^2} (u^2-1) \Tch'_m(u)^2
=
4^{m-1} (u^2-1) (u^2)^{ \mathbbm{1}_{m=2n}}
\prod_{j=1}^{ \lfloor {\frac{m-1}2}\rfloor}
 \left( u^2-\cos^2 \frac{j\pi} m\right)^2.
\eeq
(We have also used the fact that the dominant coefficient of $\Tch_m$
is $2^{m-1}$.)
Thus 
\begin{align}\label{pol-deg}
 \bC(0,x)^2-\bD(0,x)^m&= \frac 1 4
(\hN^2-4\hD)( \hN^2)^{ \mathbbm{1}_{m=2n}} \prod_{j=1}^{ \lfloor {\frac{m-1}2}\rfloor} \left( \hN^2-4\hD\cos^2 \frac{j\pi} m\right)^2\nonumber \\
&= \frac q 4 (q-4)
(x-1)^2( \hN^2)^{ \mathbbm{1}_{m=2n}}\prod_{j=1}^{ \lfloor
  {\frac{m-1}2}\rfloor} \left( \hN^2-4\hD\cos^2 \frac{j\pi}
  m\right)^2.
\end{align}
Let us now compare this to~\eqref{e1t0}. By Lemma~\ref{lem:count},
$I_i(0)\not =1$ for all $i$, so that $(x-1)^2$ is necessarily a factor
of $\hP(0,x)$. We will now prove that $x^2$ is also a factor of $\hP(0,x)$.

Consider the term obtained for $j=1$ in~\eqref{pol-deg}. Using $4\cos^2( {\pi}/ m)=q$, we get
$$
\hN^2-4\hD\cos^2 \frac{\pi} m  = \hN^2-q\hD =x\beta (q-4) (q -x(q+\beta)),
$$
which has a factor $x$.
By Lemma~\ref{lem:count},
$I_i(0)\not =0$ for all $i$, so that $x^2$ is necessarily a factor of
$\hP(0,x)$. We have proved that $\hP(0,x)$, hence also $\bP(0,x)$, is
divisible by $x^2(x-1)^2$. Moreover $\bP(0,x)$ has degree 4 and
constant term 1, thus $\bP(0,x)=x^2(x-1)^2$.

\medskip
We now wish to determine $\hQ(0,x)$.
We have just seen that $x=0$ and $x=1$ are double roots of
$C(0,x)^2-D(0,x)^m$. So they cancel as well the derivative
$2C(0,x)C(0,x)'_x-mD(0,x)^{m-1}D'_x(0,x)$, and since $D(0,0)\not =0$ and $D(0,1)\not=0$,
they must also cancel $\bD(0,x)\bC'_x(0,x)- \frac m 2
\bD'_x(0,x)\bC(0,x)$. Let us now return to the
factorisation~\eqref{e2t0}. Since $I_i(0)\not = 0,1$, we see that
$x(x-1)$ must divide $\hQ(0,x)$. Hence it  also divides
$\bQ(0,x)$. Moreover we have proved above that $[x^2]\bQ(0,x)=1$ so
that $\bQ(0,x)=x(x-1)$  as stated in the theorem.

\subsection{The Potts \gf\ of planar maps}
\label{sec:M1-PQR}
We still assume  that $q=2+2\cos
(2\pi/m)$. Let us 
 prove that the Potts \gf\ $M_1$ is related to the $P_j$'s
and $Q_j$'s by~\eqref{M11-PQ} and~\eqref{M1-alt}. It follows from the first two
identities of Proposition~\ref{prop:factor} that
$$
\hP\left( \bD\bC'_x- \frac m 2 \bD'_x\bC\right)^2
= \hQ^2\left(\bC^2-\bD^m\right),
$$
where all series and polynomials are evaluated at $(t,x)$. Using the
normalization~\eqref{PQR-nohat}, this gives
\beq\label{id}
\bP\left(\bD\bC'_x- \frac m 2 \bD'_x\bC\right)^2
= qm^2\beta^2(q/4-1) \bQ^2\left(\bC^2-\bD^m\right),
\eeq
since ${\hQ_2}^2= qm^2\beta^2(q/4-1)\hP_4$ by~\eqref{tCPQR}.
Recall that $\bD$ is defined by~\eqref{D-encore},  that the
leading coefficients of $\bC$ are given in
Lemma~\ref{lem:initial-Cr}, and that $C^2-D^m$ has degree $2m$.  Recall also the known values~\eqref{in} of $P_4$ and
$Q_2$.
Extracting from~\eqref{id} the coefficients of $x^{2m+4}$, $x^{2m+3}$, and $x^{2m+2}$
 gives:
 \begin{itemize}
 \item for the coefficient of $x^{2m+4}$, a tautology, equivalent
   to~\eqref{eT01} taken at $u=\gamma$,
\item for the coefficient of $x^{2m+3}$, an interesting relation between $\bP_3$ and $\bQ_1$, namely
\beq\label{P3Q1-init}
\bP_3=2\bQ_1+4t(1+\nu)-4tw(2\beta+q),
\eeq
\item for the coefficient of $x^{2m+2}$, the
identity~\eqref{M11-PQ}. In the calculation we use once
again~\eqref{eT01} to relate $\Tch_m(\gamma)$ and $\Tch'_m(\gamma)$,
as well as~\eqref{P3Q1-init} to express $Q_1$ in terms of $P_3$.
 \end{itemize}
The second characterization~\eqref{M1-alt} of $M_1$ is obtained in a similar
fashion by combining  instead  the second
and third identities of Proposition~\ref{prop:factor}: they imply
$$
R\left( DC'_x-\frac m 2 D'_x C\right) =Q\left( DC'_t-\frac m 2 D'_t
  C\right),
$$
and extracting the coefficient of $x^{m+1}$
gives an expression of $\tM_1'$ in terms
of $Q_1$ and $R_1$, which we transform into~\eqref{M1-alt} using~\eqref{P3Q1-init}.

\subsection{Uniqueness of the solution}
\label{sec:unique-general}

The arguments in this subsection apply whether $q$ is an
indeterminate, or $q=2+2\cos (2\pi/m)$.
The differential  system of Theorem~\ref{thm:ED}  can be written as
\beq \label{syst-line}
2Q_t'PD-QP_t'D-2QPD_t'=2R_x'PD-RP_x'D-2RPD_x'.
\eeq
Since $P$, $Q$ and $R$ have respective degree 4, 2 and 2 in $x$, this
identity relates two polynomials in $x$ of degree at most
8. Recall that $P_4=R_2=1$. Extracting the coefficient of $x^8$ gives $Q'_2(t)=0$, which, with the
initial condition $Q_2(0)=1$, implies $Q_2(t)=1$. Hence  the
leading coefficients of $P$, $Q$ and $R$ (that is, the series $P_4$, $Q_2$ and $R_2$) are independent of $t$,
and the left-hand side of~\eqref{syst-line}, as well as  its right-hand side,
 has degree at most 7 in $x$. And we are left with eight unknown series.

We denote   $P_{i,j}:=[t^i]P_j$, and similarly for $Q$ and $R$, so that
$$
P(t,x)= \sum_{i,j} P_{i,j} t^i x^j.
$$
Let $\cC_i$ be the following 8-tuple of coefficients:
$$
\cC_i=\left(P_{i,0},P_{i,1},P_{i,2},P_{i,3}; Q_{i,0},Q_{i,1}; R_{i-1,0},R_{i-1,1}\right).
$$
The right-hand side of~\eqref{init-general} gives us the values of $P(t,x)$
and $Q(t,x)$ at $t=0$, so that
$$
\cC_0=(0,0,1, -2;  0,-1;  0,0).
$$
 We will  show by induction on $i\ge 1$ that the differential system determines  the eight
coefficients of $\cC_i$, and that these coefficients
are rational functions of $q,\be,w$.

For $i\ge 1$ and $0\le j \le 7$, the equation $\Eq_{i,j}$  obtained by extracting the coefficient of
$t^{i-1}x^j$ in~\eqref{syst-line} reads
$$
\sum_{i_1+i_2+i_3=i\atop j_1+j_2+j_3=j}
(2i_1-i_2-2i_3)Q_{i_1,j_1}P_{i_2,j_2}D_{i_3,j_3}=\sum_{i_1+i_2+i_3=i-1\atop
  j_1+j_2+j_3=j+1}(2j_1-j_2-2j_3)R_{i_1,j_1}P_{i_2,j_2}D_{i_3,j_3},
$$
where $D_{i,j}=[t^ix^j]D(t,x)$. This is a linear equation in the
unknowns of $\cC_i$, of the form
\begin{multline}\label{eq:induction-ij}
\sum_{j_1+j_2+j_3=j}\left(
  2iQ_{i,j_1}P_{0,j_2}D_{0,j_3}-iQ_{0,j_1}P_{i,j_2}D_{0,j_3}
\right) \\- \sum_{j_1+j_2+j_3=j+1}(2j_1-j_2-2j_3)R_{i-1,j_1}P_{0,j_2}D_{0,j_3}=K_{i,j},
\end{multline}
where $K_{i,j}$ is a polynomial in the coefficients of $\cup_{s<i}
\cC_s$, with coefficients in $\qs[q, \be, w]$. It would be convenient if the eight equations $\Eq_{i,j}$, for
$j\in\llbracket 0, 7\rrbracket$, could define the eight unknown coefficients
of $\cC_i$, but this is not exactly what happens, for two reasons.

First, the equation $\Eq_{i,0}$  involves none of the coefficients of
$\cC_i$. Indeed, we see on~\eqref{eq:induction-ij}, specialized at $j=0$, that
the only coefficients of $\cC_i$ that $\Eq_{i,0}$  may involve are  $Q_{i,0}$,
$P_{i,0}$, $R_{i-1,0}$ and $R_{i-1,1}$.  But they do not occur, because
$P_{0,0}=P_{0,1}=Q_{0,0}=0$ (this follows from the initial
conditions~\eqref{init-general}). Hence this equation reads $K_{i,0}=0$
and only involves coefficients of $\cup_{s<i}\cC_s$.
We leave it to the reader to check that it is linear in the coefficients
of $\cC_{i-1}$, provided $i>2$.

Then, a similar  problem happens with  the sum of the eight equations $\Eq_{i,j}$, for
$j\in\llbracket 0, 7\rrbracket$: it does not involve any of the coefficients of
$\cC_i$ either. Indeed, it reads
\begin{multline*}
\sum_{j_1,j_2,j_3}\left(2iQ_{i,j_1}P_{0,j_2}D_{0,j_3}-iQ_{0,j_1}P_{i,j_2}D_{0,j_3}\right)
\\- \sum_{j_1,j_2,j_3}(2j_1-j_2-2j_3)R_{i-1,j_1}P_{0,j_2}D_{0,j_3}=\sum_{j=0}^7K_{i,j}.
\end{multline*}
But the left-hand side is the coefficient of $t^{i-1}$ in
\begin{multline*}
  2Q'_t(t,1)P(0,1)D(0,1) - Q(0,1)P'_t(t,1)D(0,1)
- 2R'_x(t,1)P(0,1)D(0,1)\\+R(t,1)P'_x(0,1)D(0,1)+2R(t,1)P(0,1)D'_x(0,1),
\end{multline*}
and this series is zero because $P(0,1)=Q(0,1)=P_x'(0,1)=0$ (see the
initial conditions~\eqref{init-general}). Hence
this sum of equations only involves coefficients of $\cup_{s<i}\cC_s$.
We leave it to the reader to check that again, this sum is linear in
the coefficients of $\cC_{i-1}$, provided $i>2$.

These observations lead us to consider the following system of eight equations:
\beq\label{cSi}
\cS_i=\left\{
\sum_{j=0}^7\Eq_{i+1,j}, \Eq_{i+1,0},\Eq_{i,2},\Eq_{i,3},\ldots,\Eq_{i,7}\right\}.
\eeq
Solving the system $\cS_1$ gives
\beq\label{C1-sol}
\cC_1=
\left(
-4, 8-2wq,4w(q-\be) -2\be-4, 2\be-2wq ;
wq+2\be+4, 4w\be-\be+wq-4 ; 2, wq-\be-4
\right).
\eeq
%
Moreover for $i>1$,  $\cS_i$ is a system of eight  linear equations for
the eight unknowns of $\cC_i$ in terms of the rational functions in
$\cup_{s<i} \cC_s$. The determinant of this linear system is
$$
 256\,{i}^{6}{q}^{3}{\be}^{7}w \left( q-4 \right)
\left( q\nu+\be^2\right) ^{2},
$$
which is non-zero when $q$ is an indeterminate but also when  $q\neq 0,4$ is of
the form $2+2\cos (2\pi/m)$.
By induction, this proves that the
coefficients $P_{i,j},Q_{i,j},R_{i,j}$ are uniquely determined by the
differential system
and its initial conditions.
Moreover
these coefficients lie in  $\qs(q,\beta,w)$, and their denominators
are products of  terms
$q$, $\beta$, $w$, $(q-4)$ and
$(q\nu+\be^2)$.

\subsection{About possible singularities}
\label{sec:poles}
It remains to prove that the coefficients of the series $P_j, Q_j$ and
$R_j$ are polynomials in $q, \nu$ and $w$.

We will use three identities that we will establish later.  Their proofs do not assume anything on the singularities of the
coefficients of $P_j, Q_j$ and $R_j$. The first one
is the characterization~\eqref{M1-alt} of $\tM_1$,  established in
Section~\ref{sec:conclusion}:
\beq\label{M1bis-2}
2\left( {\beta}^{2}+q\nu\right) \tM_1'
+ \left( 1+\nu- w\left( 2\beta+q \right)  \right) \bar P_3 /2-\bar R_1
=2+2 \beta w .
\eeq
The other two are established  in Section~\ref{sec:simpl-gen}:
\begin{multline}
  \beta  \left( wq+\beta \right)  \left( q-4 \right) Q_0
+q \left( \beta+2 \right) R_0
+ 2\left(
 \beta  \left( q-4 \right)  \left( wq+\beta \right) t+ q \right)
R_1 =\\
2 \beta  \left( q-4 \right)  \left( wq-2
 \right)  \left( wq+\beta \right) t+2 q \left( wq-2 \right) ,
\label{Q0R-M1}
\end{multline}
\begin{multline}
\beta  \left( wq+\beta \right)  \left( q-4 \right) Q_1
-2 \left(  {\beta}^{2}+ q\beta+ q \right) R_0
  -q \left( \beta+2 \right)R_1 = \\
2 \beta  \left( q-4 \right)  \left( 2 \beta w+wq-\beta-2 \right)
 \left( wq+\beta \right) t-2 q \left( \beta qw-2 \beta w+wq-\beta-
2 \right) .
\label{Q1R-M1}
\end{multline}

Let us now prove by induction on $i$ that the coefficients of $\cC_i$
have no singularity at $q=4$ or $q=-\be^2/\nu$. This holds for $i=0$
and $i=1$. From~\eqref{M1bis-2}, we derive that this holds for $R_{i-1,1}$ if it holds for
$P_{i-1,3}$, which we assume by the induction hypothesis. Now if we
remove the last
equation  from the system $\cS_i$ (given by~\eqref{cSi}), we obtain
seven polynomial equations between the coefficients of $\cup_{s\le i}
\cC_s$. Once the values of $\cC_1$ are known, they are  linear  in $
P_{i,0},  P_{i,1},  P_{i,2},  P_{i,3},  Q_{i,0}, Q_{i,0}, R_{i-1,0}$, with  determinant:
\begin{multline*}
  -128 i^6q^2  \be^4 w
 \left( 6q^4\nu^3+\be q^3(\be^4+21 \be^3+36\be^2+11\be-6)
\right.\\
\left. +\be^3q^2(3\be^3 +34\be^2 +42 \be +12) +\be^5 q(3\be^2+24\be +16)+\be^7(\be+6)\right).
\end{multline*}
The last factor is irreducible, and this determinant contains no factor
$(q-4)$ nor $(q\nu+\be^2)$. This proves that the coefficients in
$\cC_i$ are not singular at $ q=4$ nor $q=-\be^2/\nu$.

Let $\Eq_i^{(1)}, \Eq_i^{(2)}, \Eq_i^{(3)}$ be the equations obtained
by extracting the coefficient of $t^i$ in the
equations~\eqref{M1bis-2}, \eqref{Q0R-M1} and \eqref{Q1R-M1}. The
system
$$
\left\{\sum_{j=0}^7 \Eq_{i+1,j}, \Eq_{4,i}, \ldots, \Eq_{7,i}, \Eq_i^{(1)}, \Eq_i^{(2)},
\Eq_i^{(3)}\right\}
$$
relates polynomially the coefficients $P_{s,j}, Q_{s,j}, R_{s,j} $ for
$s\le i$ and
the coefficients  (in $t$) of the series $M_1(t)$. For $i>1$,
it is linear in $P_{i,0}, \ldots, P_{i,3}, Q_{i,0}, Q_{i,1},
R_{i,0}, R_{i,1}$, with determinant
$$
-16 i^5\be^5(q-4)w(q\nu+\be^2)^5,
$$
and this excludes singularities at $q=0$.

Observe that the sum of~\eqref{Q0R-M1} and~\eqref{Q1R-M1} is divisible by
$\be$.  If we consider now the system
$$
\left\{\Eq_{i+1,0}, \Eq_{4,i}, \ldots, \Eq_{7,i}, \Eq_i^{(1)}, \Eq_i^{(2)},
\be^{-1}(\Eq_i^{(2)}+\Eq_i^{(3)})\right\},
$$
in the same unknowns as before, we obtain the determinant
$$
16i^5q ^2(q-4)w(q\nu+\be^2)^5,
$$
proving this time that the coefficients are not singular at $\be=0$.


\medskip
Finally, to rule out poles at $w=0$, we resort to a different
argument. First, we return to the case $q=q_m:=2+2\cos(2\pi/m)$
studied from Section~\ref{sec:cat-M} to Section~\ref{sec:M1-PQR}. A first observation is
that the series $C_r$ involved in the invariant equation~\eqref{eq:inv} are
series in $t$ with coefficients in $\qs[q,\nu,w]$. This follows from
the proof of Lemma~16 in~\cite{bernardi-mbm-alg}, using the fact that
the objects denoted by $L_r(t;y)$ belong to
$\qs[q,\nu,w,t,1/y]$. A second observation is  that the series $Y_i$ of Lemma~\ref{lem:count} satisfy $Y_i=1+t(1+\beta
v) +O(t^2)$  where $v\in \rs^*$, and that their coefficients belong to
$\rs(\nu)[w]$. This follows easily from the proof
of Lemma~\ref{lem:count}, since $P_j'(z_0)$ does not depend on $w$. Consequently, $I_i:=I(t,Y_i)$ is a series in $t$, with
a non-zero constant term \emm that does not depend on $w$,, and
its other coefficients lie in $\rs(\nu)[w]$. Using these two observations, \eqref{hP-expr} now implies that
$ \hP(t,x)$ is  a polynomial in $x$ with coefficients in
$\rs(\nu)[w][[t]]$.  By~\eqref{PQR-nohat}, the same holds for
$P(t,x)$. Similarly, the polynomials $Q(t,x)$ and $R(t,x)$ have
coefficients in $\rs(\nu)[w][[t]]$.

In the next subsection, we prove (without assuming polynomiality of
the coefficients of the series $P_j$, $Q_j$ and $R_j$), that these
series, once specialized at $q=q_m$, are indeed the coefficients of the
polynomials $P(t,x), Q(t,x)$ and $R(t,x)$ constructed for
$q=q_m$. Hence, if one of these series had a pole at $w=0$, this would
remain the case for infinitely many values  $q_m$, and this contradicts the
fact that  $P(t,x), Q(t,x)$ and $R(t,x)$ have
coefficients in $\rs(\nu)[w][[t]]$.

\subsection{Conclusion of the proof}
\label{sec:conclusion}
We can now conclude the proof of Theorem~\ref{thm:ED}. Let
$q$ be an indeterminate. We have  proved in Section~\ref{sec:unique-general} that the differential
system~\eqref{de-encore} and the initial conditions~\eqref{init-general} define the  series
$P_j$, $Q_j$ and $R_j$ uniquely as formal power series in $t$, and that their
coefficients lie in
$\qs[q,\beta, w]$.
%

Let us temporarily denote these series
by $\bar P_j$, $\bar Q_j$ and
$\bar R_j$, to avoid confusion with the series denoted $P_j$,
$Q_j$ and $R_j$ above, which depend on a specific value of $q$ of the form
$ q_m:=2+2\cos (2\pi/m) $.
Specializing the indeterminate $q$ to $q_m$ in the series $\bar P_j$,
$\bar Q_j$ and $\bar R_j$ gives series in~$t$ satisfying the
differential system and its initial conditions. But we have also proved
that this system has a unique power series solution when $q=q_m$. Thus
$\bar P_j$ evaluated at $q=q_m$ coincides with $P_j$, and similar
statements relate the series $Q_j$ and $R_j$ and their barred versions. We have proved in  Subsection~\ref{sec:M1-PQR} that when $q=q_m$, the Potts \gf\ $M_1$ is related
to the $P_j$'s, $Q_j$'s and $R_j$'s by~\eqref{M11-PQ} and~\eqref{M1-alt}. This
means that when $q=q_m$,
$$
12  \left( {\beta}^{2}+q\nu \right) \tM_1  +
 \bar P_3 ^{2}/4
+2 t \left(1+\nu -w(2 \beta+q) \right)\bar P_3  -\bar P_2
 +2 \bar Q_0
=4 t \left(1+  w(3 \beta +q) \right)
$$
and
$$
2\left( {\beta}^{2}+q\nu\right) \tM_1'
+ \left( 1+\nu- w\left( 2\beta+q \right)  \right) \bar P_3 /2-\bar R_1
=2+2 \beta w .
$$
Since all series involved in these identities have
polynomial
coefficients in $q$, and coincide for infinitely many values of $q$,
they must hold for $q$ an indeterminate.

Let us finally prove that $M_1$ is differentially algebraic. This
follows from the uniqueness of the solution of our differential system
in terms of power series, via  an
approximation theorem due to Denef and
Lipshitz~\cite[Thm.~2.1]{denef-lipshitz}. This theorem generalizes to
differential systems one of Artin's approximation theorems for
algebraic systems, and implies, in our context,  that each of the
series $P_j$, $Q_j$ and $R_j$ is differentially
algebraic. The expression of
$\tM_1=t^2M_1$ given by~\eqref{M11-PQ}, and the fact that differentially
algebraic series form a ring, implies that $M_1$ is also
differentially algebraic.

\section{Differential system for coloured triangulations}
\label{sec:triang}

We now consider triangulations, and more generally  \emm
near-triangulations,, which are planar maps in
which every  non-root face has degree 3.
We weight these maps by the number of
vertices (variable $w$), the degree of the root-face ($y$), and by
their Potts polynomial (divided by $q$). We denote by
$\Q(q,\nu,w;y)\equiv \Q(y)$ the associated \gf:
$$
\Q(q,\nu,w;y)= \frac 1 q \sum_{M} \Ppol_M(q,\nu) w^{\vv(M)}
y^{\df(M)},
$$
where the sum runs over all near-triangulations. We ignore the number
of edges, which would be redundant: a near-triangulation with $v$ vertices and
outer degree $d$ has $3v-d-3$ edges.
In our first
paper on coloured maps~\cite{bernardi-mbm-alg}, we counted edges (with
a variable $t$)
rather than vertices, using a \gf\
$Q(q,\nu,t;x, y)$
involving two catalytic variables $x$ and $y$.
It is related to $T(y)$ by:
$$
\Q(q,\nu,w;y)\equiv T(y)= w\,Q(q,\nu,w^{1/3}; 0, w^{1/3}y).
$$

Our objective is to establish a differential system for the Potts \gf\ of
near-triangulations of outer degree 1, denoted by $T_1\equiv
T_1(w)$. Note that this is
the coefficient of $y$ in $T(y)$.
More generally, we write
$$
T(q,\nu, w;y)= \sum_{d\ge 0} T_d(w) y^d,
$$
hoping that no confusion arises with the $m$th Chebyshev polynomial
$\Tpol_m$. The root-edge of near-triangulations counted by $T_1$ is a
loop. Its deletion gives a near-triangulation of outer degree 2, and
thus
$$
T_1=\nu T_2,
$$
an identity that will be useful later. The expansion of $T_1$ at order
3 reads
\begin{multline*}
T_1=\nu(q-1+\nu)w^2+\nu\left((q-1)(q-2+2\nu)
+\nu^2(q-1+\nu^2)\right.\\
\left. +2\nu(q-1+\nu)(q-1+\nu^2)
+\nu^2(q-1+\nu)^2
\right)w^3+O(w^4),
\end{multline*}
as illustrated in Figure~\ref{fig:small-triangulations}.

\begin{figure}[h]
  \centering
  \includegraphics{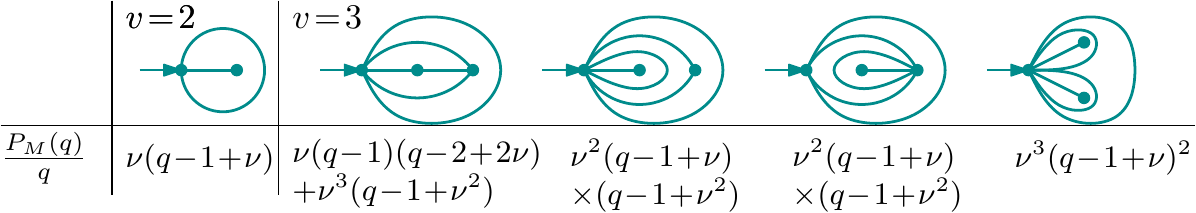}
\caption{The rooted near-triangulations of outer degree 1 with $v=2$
  and $v=3$ vertices, and their Potts polynomials (divided by $q$).}
\label{fig:small-triangulations}
\end{figure}

Our differential system for near-triangulations is  very
similar to the one obtained for general maps (Theorem~\ref{thm:ED}),
but a bit simpler.
\begin{Theorem}\label{thm:ED-Triang}
Let $q$ be an indeterminate, $\beta=\nu-1$ and
$$
\bD(w,x)=q\nu^2x^2+\beta(4\beta +q)x+ q\beta\nu(q-4)w+\beta^2.
$$
There exists a unique triple $(\bP(w,x), \bQ(w,x),\bR(w,x))$ of
 polynomials in $x$ with coefficients in
$\qs[q,\nu][[w]]$,
 having
 degree $3, 2$ and $1$ respectively in $x$, such that
\beq\label{init-T-thm}
 \begin{array}{ll}
 [x^3]\bP(w,x)=1,& \quad \bP(0,x)=x^2(x+1/4),
\\
&\quad \bQ(0,x)= x(2\nu x+1),
 \end{array}
\eeq
and
\beq\label{deT}
\frac 1{\bR}\frac{\partial }{\partial x} \left( \frac{\bR^2}{\bP
    \bD^2}\right)= \frac {1}{\bQ}\frac{\partial }{\partial w} \left(
  \frac{ \bQ^2}{\bP \bD^2}\right).
\eeq

Let $\bP_j(w)\equiv P_j$ (resp. $\bQ_j$, $\bR_j$) denote the coefficient of $x^j$ in
$\bP(w,x)$ (resp. $\bQ(w,x)$, $\bR(w,x)$).
The Potts \gf\ of
near-triangulations of outer degree $1$, denoted by
$
\Q_1(w)\equiv T_1$, can be expressed in terms of these series
using
\beq\label{Q1-PQt}
20 {\nu}^{2}q\Q_1-4{\nu}^{2} { P_1} +4  \nu{Q_0}+
 \left( { Q_1}-1 \right)  \left( { Q_1}+\nu-3 \right)+2 \nu  \left( q\nu-24 \beta-6 q \right) w =0.
\eeq
An alternative characterization of $T_1$ is
\beq\label{T1-alt}
2\nu q T'_1  =R_1- q(\beta-1)+8\beta.
  \eeq
The series $T_1$ is differentially algebraic, that is,  satisfies a
non-trivial differential equation  with respect to the vertex variable
$w$. The same holds for each series $P_j$, $Q_j$ and $R_j$.
\end{Theorem}

Eq.~\eqref{deT} is, in compact form, a system of eight differential
equations in $w$ relating the 9 series $P_j$, $Q_j$, $R_j$. Since
$P_3$ is given explicitly, we have in fact as many equations as
unknown series. We will  see that $Q_2=2\nu$. In fact, we will prove that~\eqref{deT}, combined with the initial
conditions~\eqref{init-T-thm}, determines uniquely all series $P_j$,
$Q_j$, and $R_j$.
For instance,
$$
  P_0=-\be w+\be\left( 8 q +q(q-12)\be/4-(q+6)\be^2-3\be^3\right)w^2+ O(w^3).
$$

We expect $T_1$ to satisfy a differential equation of order 4 (Section~\ref{sec:simpl-gen}). We will
work out in details several special cases in which $T_1$ satisfies a
second order DE
(Sections~\ref{sec:nu=0} to~\ref{sec:q=0}).

The proof of the above theorem is  similar to that of
Theorem~\ref{thm:ED}, and we mostly tell where these proofs differ,
without giving details otherwise.

\subsection{An equation with one catalytic variable}
As in the case of general planar maps, we  begin with
 an equation for the series $T(y)$, taken from~\cite{bernardi-mbm-alg}.
We assume  that $q=2+2\cos (2k\pi/m)$ with $k$ and $m$ coprime and
$0<k<2m$. We still write $\beta=\nu-1$. We introduce the following
notation:
\begin{itemize}
\item $I(w,y)\equiv I(q,\nu,w;y)$ is a variant of the \gf\
  $T(q,\nu,w;y)$:
$$
I(w,y)=yq\Q(q,\nu,w;y)-\frac{1}{y}+\frac{1}{y^2},
$$
\item $N(y,x)$ and $D(w,x)$ are the following (Laurent) polynomials:
$$
\bN(y,x)= \beta(4-q)\by+q \nu x +\beta(q-2),
$$
with $\by=1/y$, and
\beq\label{D-exprT}
\bD(w,x)=q\nu^2x^2+\beta(4\beta +q)x+ q\beta\nu(q-4)w+\beta^2.
\eeq
\end{itemize}
We still denote by $\Tch_m$ the $m$th Chebyshev polynomial of the first
kind.
\begin{Proposition}
  There exist $m+1$ \fps\ in $w$ with coefficients in
$\qs(q,\nu)$, denoted $C_0(w), \ldots, C_m(w)$, such that
\beq\label{eq:invT}
\bD(w,I(w,y))^{m/2} \,\Tch_m\left(\frac{\bN(y,I(w,y))}{2\sqrt
    {\bD(w,I(w,y))}}\right)= \sum_{r=0}^m C_r(w) I(w,y)^r.
\eeq

\end{Proposition}

\begin{proof}
This follows from Corollary~12 and Lemma~19 of~\cite{bernardi-mbm-alg}.
One must prove that the series in $t$ denoted $C_r$ in~\cite{bernardi-mbm-alg}
(which coincide with the series $C_r$ of the present paper, with
$t=w^{1/3}$), are not only series in $t$ but in fact series in
$t^3$. This can be done by following carefully the proof of Lemma~19
of~\cite{bernardi-mbm-alg}, using the fact that the series denoted
$K(ty)$ is a series in $t^3$. This implies that the series denoted
$t^jS_{i,j}$ and $t^jT_{i,j}$ are also series in $t^3$, and one
concludes using Eq.~(93)
of~\cite{bernardi-mbm-alg}.
\end{proof}

By expanding the invariant equation~\eqref{eq:invT} near $y=0$, we can
express the series $C_r$ in terms of the derivatives of $\Q$ with
respect to $y$, evaluated at $y=0$. We will need the following
expressions of $C_m, \ldots, C_{m-3}$. The first three are explicit in
terms of $q, \nu$ and $w$, but the last one involves the series
$
\Q'_y(0)\equiv \Q_1$.

\begin{Lemma}
 \label{lem:initial-CrT}
Denote $\beta=\nu-1$,
and recall that $q=2+2\cos(2k\pi/m)$.
We have:
\begin{align*}
\frac{(-1)^k \,C _{m}}{(q\nu^2 )^{m/2}} &=1,
\\
\frac{(-1)^k \, C _{m-1}}{(q\nu^2 )^{m/2-1}}
&=\frac{m\beta }{2}\left(4\be+q+\be m(q-4)\right),
\\
 \frac{(-1)^k\,C _{m-2}}{(q\nu^2 )^{m/2-2}}
&=\frac {m\beta } {24} \bigg( 12q^2\nu^3w(q-4)+48 \nu^2\be(m-1)
  +\be(m-1)(q-4) \times
\\
&\ \ \ \ \ \ \ \Big( 6q+12  +6\be (mq+4) + \be^2(m^2q-4m^2+mq+20m-12)\Big)\bigg),
\\
 \frac{(-1)^k\, C _{m-3}}{(q\nu^2 )^{m/2-3}} &=-\frac 1 2 \,{m}^{2}{q}^{3}{\nu}^{4}
\beta^2  \left( q-4 \right)
\Q_1
\\
&
+ \frac{m(m-1)\beta^2}{720}\bigg(   180q^2\nu^3w(q-4)\Big((q-4)\be m +8\be
  +2q\Big)
+960\be \nu ^3(m-2)  \\
& +720\be \nu^2(m-2)(q-4)(\be m -\be +2)
+\be (m-2) (q-4)^2 \Big(\be^3m^3(q-4) \\
&+3\be^2m^2(\be q+16\be +5q) +\be
m(2\be^2 q -8\be^2+15\be q+360 \be +60 q+180)\\
& -180\be^2+360\be +60 q+300\Big)\bigg),
%
\end{align*}
where $
\Q_1$ is the Potts
generating function for near-triangulations with outer degree $1$.
\end{Lemma}
\noindent{\bf Remark.} When comparing this lemma with its
counterpart for general planar maps, Lemma~\ref{lem:initial-Cr}, we observe
that we have no  term $\Tch_m(\gamma)$ nor $\Tch_m'(\gamma)$ here. This is because
$$
\gamma= \lim_{y\rightarrow 0} \frac{N(y,I(w,y))}{2\sqrt{D(w,I(w,y))}}=
\frac {\sqrt q}2=\cos(k\pi/m),$$
so that $\Tch_m(\gamma)=(-1)^k$ and $\Tch'_m(\gamma)=0$.
\begin{proof}
 We multiply the invariant equation~\eqref{eq:invT} by $y^{2m}$,
so that it
 becomes not singular at $y=0$, and expand it around $y=0$. We
 extract successively the coefficients of $y^0,y^2,y^4$ and $y^6$
 and obtain in this way expressions of $C_m, \ldots, C_{m-3}$ (the
 coefficients of odd powers of $y$ do not give more information). For
 instance, since
 \begin{align*}
 y^2 I(w,y)&= 1+O(y),\\
y^2N(y, I(w,y))&= q\nu +O(y),
\\
y^4 D(w,I(w,y))&= q\nu^2+O(y),
 \end{align*}
extracting the coefficient of $y^0$ gives
$$
q^{m/2}\nu^m \Tch_m(\sqrt q /2)= C_m,
$$
and the expression of $C_m$ follows since $\Tch_m(\sqrt q /2)=(-1)^k$. As we extract the
coefficients of higher powers of $y$, derivatives of $\Tch_m$, taken
at the point $\sqrt q/2=\cos(k\pi/m)$, occur. We can find explicit
 expressions for them in terms of $q$, $k$ and $m$ using $\Tch'_m(\sqrt
 q/2)=0$ and the differential equation
$$
(1-u^2)\Tch''_m(u)-u\Tch'_m(u)+m^2\Tch_m(u)=0.
$$
Similarly, derivatives of $y^2 I(w,y)$ with respect to $y$ occur, and
this is why the expression of $C_{m-3}$ involves the series $
\Q_1$.
\end{proof}

\subsection{Some special values of $\boldsymbol{y}$}
We now work out the counterpart of Lemma~\ref{lem:count}. We still
denote $n=\lfloor m/2\rfloor$.
\begin{Lemma}\label{lem:countT}
Let
$q=2+2\cos(2\pi/m)$, and denote $\be=\nu-1$.
There exist  $m-2$ distinct formal power series in $w$, denoted
$Y_1,\ldots, Y_{m-2}$, that have
constant term $1+O(\beta)$ as $\beta \rightarrow 0$
and satisfy
\begin{equation}\label{eqj-squareT}
\bN(Y,I(w,Y))^2 =4\cos(j\pi/m)^2\bD(w,I(w,Y)),
\end{equation}
for some $j \in \llbracket 1,n \rrbracket$. Their coefficients are
algebraic functions of $\beta$ over $\rs$.

Let us denote  $I_i(w):=I(w,Y_i)$. This is a \fps\ in $w$ with
coefficients in the algebraic closure of $\rs(\be)$, denoted by
$\overline{\rs(\be)}$. The $(m-2)$
series  $I_i$ are distinct, and for $1\le i\le m-2$,
$$
I_i(0)\not\in\{0, -1/4, -1/(2\nu)\},
\qquad D(w,I_i)\not = 0 , \qquad
\frac{\partial I}{\partial y}(w,Y_i) \not = 0.
$$
\end{Lemma}

The key difference with Lemma~\ref{lem:count} is that the coefficients
of the series $Y_i$ do not
have their coefficients in $\rs(\beta)$, but in its algebraic closure. Let us
illustrate this by an example.

\medskip

\noindent{\bf Example: bicoloured triangulations.} We take $q=2$, that
is, $m=4$ (and $n=2$). The relevant values of $j$ are 1 and 2, and we
want to solve
$$
N(Y,I(w,Y))^2=2D(w,I(w,Y))
$$
for $j=1$, and
$$
N(Y,I(w,Y))=0
$$
for $j=2$. Using the definitions of $N$ and $D$, this reads, for
$j=1$,
$$
4 {Y}^{3} \left( 2 \beta Y-2 \beta+Y-2 \right)
T(Y) -4 {Y}^{3} \left( 1+\beta \right) w+ \left( Y-2 \right)
 \left( \beta {Y}^{2}-2 \beta Y+2 \beta-2 Y+2 \right)
=0,
$$
and for $j=2$,
$$
2 {Y}^{3} \left( 1+\beta \right)T(Y)+\beta+1-Y=0.
$$
Recall that we are interested in series $Y(w)$ with constant term
$1+O(\beta)$. We find one such solution for each of the above
equation, which satisfy
\begin{align*}
  Y_1(w)&= {\frac {\beta+1-\sqrt {1-{\beta}^{2}}}{\beta}}+O(w),
\\
Y_2(w)&=1+\beta+O(w).
\end{align*}
\qee

\begin{proof}
We denote
$q_j=4\cos(j\pi/m)^2=2+2\cos(2j\pi/m)$. Observe that $q=q_1$. We
assume that $m=2n+1$ is odd. The proof is similar in the even case.

\medskip
\noindent\paragraph{\bf Existence of the series $\boldsymbol{Y_i}$.}
We are looking for series $Y=Y(w)$ solutions of~\eqref{eqj-squareT}.
Multiplying~\eqref{eqj-squareT} by $Y^4$, we obtain the equation
$
\Phi_j(w,Y)=0,
$
where $$
\Phi_j(w,y)=y^4\bN(y,I(w,y))^2-q_jy^4\bD(w,I(w,y)).
$$
Observe that $y^2 I(w,y)$ is a \fps\ in $w$ with coefficients in
$\rs[\nu,y]$. Hence the same holds for $y^2\bN(y,I(w,y))$ and
$y^4\bD(w,I(w,y))$, and finally for $\Phi_j(w,y)$. Expanding $I$, $N$
and $D$ at first order in $w$ gives:
 \begin{align}
y^2I(w,y)&=1-y+O(w),\label{I-exp-T}
\\
 y^2\bN(y,I(w,y))&= q\nu-y(q+2\beta(q-2))+(q-2)\beta y^2+O(w),
\nonumber
\\
 y^4\bD(w,I(w,y))&=q\nu^2
-2\,q \nu ^{2}y
+\left(q+ 3q\beta+(q+4){\beta}^{2} \right)y^2
-\beta\left( 4\beta+q \right) y^3
+{\beta}^{2}y^4 +O(w).\label{D-exp-T}
 \end{align}
The equation $\Phi_j(w,Y)=0$ thus reads
\beq\label{eq-Phi-S-T}
P_j(Y)+wS_j(w,Y)=0,
\eeq
where
\begin{multline*}
  P_j(y)=\big( q\nu-y(q+2\beta(q-2))+(q-2)\beta y^2\big)^2\\
-q_j
\big( q\nu^2
-2\,q \nu ^{2}y
+\left(q+ 3q\beta+(q+4){\beta}^{2} \right)y^2
-\beta\left( 4\beta+q \right) y^3
+{\beta}^{2}y^4 \big)
\end{multline*}
and $S_j(w,y)$ is a power series in $w$ with coefficients in
$\rs[\nu,y]$. In particular, the coefficient of $w^0$ in $Y$, denoted
$y_0$, must satisfy $P_j(y_0)=0$.
The roots of this quartic polynomial can be seen as Puiseux series in
$\beta$ with coefficients in  $\cs$ (see for
instance~\cite[Ch.~6]{stanley-vol2}). Let us
  focus on the roots
that are finite at $\beta=0$ and have constant term 1. Using Newton's
polygon method, we find that they read $y_0=1+ \beta v+O(\beta^3)$,
where $v$ must satisfy  the equation~\eqref{eq-quadratic} that we  studied
when constructing the series $Y_i$ for general planar maps. Thus when
$j>1$ we
find for $P_j$  two distinct roots with constant term 1, of the form $y_0=1+\beta
v_j^\pm+O(\beta^3)$,
and only one such root $y_0=1+\beta
v_1^++O(\beta^3)$ when $j=1$ (with $v_1^+=1/q$).

Now having fixed one root $y_0$ of $P_j$, let us return to~\eqref{eq-Phi-S-T} and
extract from it the coefficient of $w^p$, for $p\ge 1$. This gives
$$
P'_j(y_0)[w^p]Y + \big(\hbox{expression involving only } [w^i]Y \hbox{ for }
i<p\big) =0.
$$
Since $y_0$ is not a double root of $P_j$, this allows us to compute
the coefficient of $w^p$ in $Y$ by induction on $p$, and thus
determines $Y$ completely. Moreover, $[w^p]Y$
is an algebraic function of $\beta$ over $\rs$.

As argued in the proof of Lemma~\ref{lem:count}, the $(m-2)$ values $v_1^+,
\ldots, v_n^+, v_2^-, \ldots, v_n^-$ are distinct. They give rise to
$(m-2)$ distinct series $Y(w)$ satisfying $Y(0)=1+O(\beta)$
and~\eqref{eqj-squareT}  for some $j\in \llbracket 1,n \rrbracket
$. We denote them $Y_1, \ldots, Y_{m-2}$ with $Y_i(0)=1+\beta
v_i +O(\beta^3)$,
 where as before,
$$
\{v_1, \ldots, v_{m-2}\}=\{v_1^+,
\ldots, v_n^+, v_2^-, \ldots, v_n^-\}.
$$
Clearly these series are non-zero, and we have proved the first
statement of the lemma.

\medskip
\paragraph{\bf Properties of $\boldsymbol{Y_i}$ and
  $\boldsymbol{I_i:=I(w,Y_i)}$.}
Let us prove that the series $I_i$ are distinct. Returning to~\eqref{I-exp-T}
gives
$$
I_i(0)=-\beta v_i+O(\beta^2).
$$
This proves that the series $I_i$ are distinct (since the $v_i$'s are
distinct), and also that $I_i(0)$ is distinct from $0, -1/4$ and
$-1/(2\nu)$.

Let us now prove that
$D(w,I_i(w))\neq 0$. Returning to~\eqref{D-exp-T}, we find that:
$$
D(0,I_i(0)) = \beta^2(qv_i^2-qv_i+1)+ O(\beta^3).
$$
Comparing with~\eqref{D-constant-term} shows that this is non-zero.

We finally check that
$$
\displaystyle \frac{\partial I}{\partial y}(0,Y_i(0))=-1+4\beta
v_i+O( \beta^2)
$$
is non-zero, as stated in the lemma.
\end{proof}

\subsection{Some polynomials with common roots}
We still assume that $q=2+2\cos(2\pi/m)$.
We  denote  as before:
$$
\bC(w,x)=\sum_{r=0}^m C_r(w)\, x^r.
$$
Thus
$C(w,x)$ is a  polynomial in $x$ with coefficients in
$\rs(\nu)[[w]]$. This is also true of $D(w,x)$
(see~\eqref{D-exprT}).
 The coefficients of
$D$ are explicit, but those of $C$ are unknown. Here is now the
counterpart of Lemma~\ref{lem:rootsCD}.
\begin{Lemma}\label{lem:rootsCDT}
Each of the series $I_i:=I(w,Y_i)$ defined by Lemma~{\rm\ref{lem:countT}} satisfies
\begin{align*}
\bC(w,I_i)^2&=\bD(w,I_i)^m,
\\
\bD(w,I_i) \frac{\partial \bC}{\partial x}(w,I_i)&= \frac m 2 \bC(w,I_i)
\frac{\partial \bD}{\partial x}(w,I_i),
\\
\bD(w,I_i) \frac{\partial \bC}{\partial w}(w,I_i)&= \frac m 2 \bC(w,I_i) \frac{\partial \bD}{\partial w}(w,I_i).
  \end{align*}
The first two identities imply
that $I_i$ is actually a double root of $C^2-D^m$.
\end{Lemma}
The proof is identical to that of Lemma~\ref{lem:rootsCD}, with the
variable $t$ replaced by $w$.

\begin{Proposition}\label{prop:factorT}
 Let $I_i:=I(w,Y_i)$ be the $(m-2)$
series defined in  Lemma~{\rm\ref{lem:countT}}.
 There exists a  triple $(\hP(w,x), \hQ(w,x),\hR(w,x))$ of
 polynomials in $x$ with coefficients in $\overline{\rs(\be)}[[w]]$, having
 degree at most $3$, $2$ and $1$ respectively in $x$, such that
\begin{align*}
\bC(w,x)^2-\bD(w,x)^m&= \displaystyle \hP(w,x) \prod_{i=1}^{m-2}(x-I_i)^2,
\\
\bD(w,x)\bC'_x(w,x)- \frac m 2 \bD'_x(w,x)\bC(w,x)&= \displaystyle \hQ(w,x) \prod_{i=1}^{m-2}(x-I_i),
\\
\bD(w,x)\bC'_w(w,x)- \frac m 2 \bD'_w(w,x)\bC(w,x)&= \displaystyle \hR(w,x) \prod_{i=1}^{m-2}(x-I_i).
 \end{align*}
\end{Proposition}
\begin{proof}
The proof is almost the same as that of
Proposition~\ref{prop:factor}, with the variable $t$ replaced by~$w$.
The only difference is that here the polynomial
$\bC(w,x)^2-\bD(w,x)^m$ has degree at most $2m-1$ (instead of $2m$), and the polynomial
$\bD(w,x)\bC'_w(w,x)- \frac m 2 \bD'_w(w,x)\bC(w,x)$ has degree at
most $m-1$ (instead of $m$). This is because  the coefficient of $x^{2m}$ in
$\bC(w,x)^2$ and $\bD(w,x)^m$ is  $q^m\nu^{2m}$ (by
Lemma~\ref{lem:initial-CrT}), and  the
coefficient of  $x^{m}$ in  $\bD(w,x)\bC'_w(w,x)$ and $\frac m 2
\bD'_w(w,x)\bC(w,x)$ is
$\displaystyle -\frac{m}{2}\beta(q-4) q^{m/2+1}\nu^{m+1}$ (again, by Lemma~\ref{lem:initial-CrT}).
\end{proof}

\subsection{Differential system}

We still assume that
$q=2+2\cos(2\pi/m)$.

\medskip
\noindent
{\bf A differential equation relating {\boldmath$\hP, \hQ$} and {\boldmath$\hR$}.}
Starting from Proposition~\ref{prop:factorT}, one first  proves that
$$
\frac{1}{\hQ} \frac{\partial}{\partial w} \left(\frac{ \hQ^2}{\hP\bD^2}\right)
=\frac{1}{\hR} \frac{\partial}{\partial x}
\left(\frac{\hR^2}{\hP\bD^2}\right).
$$
This argument is the same as  in
Section~\ref{sec:diff-general}.
The above
equation coincides with~\eqref{deT}, but with hats over the letters
$P, Q$ and $R$.  The series $\bP, \bQ$ and $\bR$ occurring in Theorem~\ref{thm:ED-Triang} will  be normalizations of $\hP, \hQ$ and $\hR$ by
multiplicative constants, independent from $w$ and $x$.

\medskip
\noindent {\bf The leading coefficients of {\boldmath$\hP,\hQ,\hR$}.}
Proposition~\ref{prop:factorT} gives
\begin{align*}
\displaystyle [x^3]\hP(w,x)&=  [x^{2m-1}]\left(\bC(w,x)^2-\bD(w,x)^m\right),\\
\displaystyle [x^2]\hQ(w,x)&= [x^m]\left(\bD(w,x)\bC'_x(w,x)- \frac m 2 \bD'_x(w,x)\bC(w,x)\right),
\\
\displaystyle [x^1]\hR(w,x)&=  [x^{m-1}]\left(\bD(w,x)\bC'_w(w,x)- \frac m 2 \bD'_w(w,x)\bC(w,x)\right).
 \end{align*}
Recall that $\bD$ is given explicitly by~\eqref{D-exprT}, and that
Lemma~\ref{lem:initial-CrT} gives the leading coefficients of
$C(w,x)$. This allows to us determine the leading coefficients of
$\hP,\hQ,\hR$:
\beq\begin{cases}\label{tCPQRT}
 \hP_3:=[x^3] \hP(w,x)=  m^2\,\beta^2 (q-4) q^{m-1}\nu^{2m-2},
\\
\hQ_2:=[x^2]\hQ(w,x)=\displaystyle \frac{m^2\,\beta^2}{2} (q-4) q^{m/2}\nu^{m},
\\
\hR_1:=[x^1]\hR(w,x)= \displaystyle  \frac{m^2\,\beta^2}{2} (q-4) q^{m/2}\nu^{m}
\left(  q T_1'
+ \frac {q(\beta-1)-8\beta}{2\nu}\right).
\end{cases}\eeq
Let us define
\beq\label{normalT}
\bP(w,x)=\frac{\hP(w,x)}{\hP_3}, \quad
\bQ(w,x)=\frac{2\nu \hQ(w,x)}{\hQ_2} \quad \hbox{ and } \quad
\bR(w,x)=\frac{2\nu \hR(w,x)}{\hQ_2}.
 \eeq
Then $\bP$, $\bQ$, $\bR$ satisfy~\eqref{deT} and  moreover,
\beq\label{P3Q2T}
[x^3]\bP(w,x)=1, \qquad [x^2]\bQ(w,x)=2\nu, \qquad
[x^1]\bR(w,x)=  2\nu q T_1'+ {q(\beta-1)-8\beta}.
\eeq
In particular, the first  of the initial
conditions~\eqref{init-T-thm} holds. (We do not give $Q_2$
explicitly in the statement of the theorem, as its value follows from
the differential system and the initial conditions.)
\medskip

\noindent
{\bf The case {\boldmath$w=0$}.}
It remains to determine the values $\bP(0,x)$ and
$\bQ(0,x)$. Proposition~\ref{prop:factorT}
gives
\begin{align}
\bC(0,x)^2-\bD(0,x)^m&= \displaystyle \hP(0,x) \prod_{i=1}^{m-2}(x-I_i(0))^2,\label{e1t0T}
\\
\bD(0,x)\bC'_x(0,x)- \frac m 2 \bD'_x(0,x)\bC(0,x)&= \displaystyle \hQ(0,x) \prod_{i=1}^{m-2}(x-I_i(0)).\label{e2t0T}
 \end{align}
We will now combine our knowledge of $\bC(0,x)$ obtained
in~\cite{bernardi-mbm-alg} with the properties of the series $I_i(w)$
gathered in Lemma~\ref{lem:countT} to determine $\hP(0,x)$ and $\hQ(0,x)$.

The series $\bC(0,x)$ can be determined using results
in~\cite[Lemma~19]{bernardi-mbm-alg}
(with $t=w^{1/3}$).
 There, it is proved that
$$
\bC(0,x^2-x)\equiv \sum_{r=0}^m C_r(0)(x^2-x)^r=\hD^{m/2}\,
\Tch_m\left(\frac{\hN}{2\sqrt{\hD}}\right),
$$
where
$$
\hN=\bN(1/x,x^2-x)=\beta(4-q)x+q \nu (x^2-x) +\beta(q-2)
$$
and
$$
\hD=\bD(0,x^2-x)=q(1+\beta)^2(x^2-x)^2+\beta(4\beta +q)(x^2-x)+\beta^2.
$$
Recall the expression~\eqref{T2m1} of $\Tch_m(u)^2-1$. Thus \eqref{C0-explicitT} gives
\begin{align}
\bC(0,x^2-x)^2-\bD(0,x^2-x)^m
&=  \frac 1 4 (\hN^2-4\hD)
\left(\hN^2\right)^{ \mathbbm{1}_{m=2n}} \prod_{j=1}^{ \lfloor {\frac{m-1}2}\rfloor}
 \left( \hN^2-4\hD\cos^2 \frac{j\pi} m\right)^2 \label{fact1T}\\
&=\hP(0,x^2-x)\prod_{i=1}^{m-2}\left(x^2-x-I_i(0)\right)^2\nonumber
\end{align}
by~\eqref{e1t0T}.
We want to prove that $x^2(x+1/4)$ is a factor of $\hP(0,x)$, or
equivalently, that $x^2(x-1)^2(x-1/2)^2$ is a factor of
$\hP(0,x^2-x)$.
Since by Lemma~\ref{lem:countT}, $I_i(0)\not \in\{ 0, -1/4\}$, it suffices to prove
that $x^2(x-1)^2(x-1/2)^2$ divides the right-hand side
of~\eqref{fact1T}. Using the above values of $\hN$ and $\hD$, we first find
a factor $(x-1)^2$ in  $(\hN^2-4\hD)$. Then, using
  $4\cos^2 ({\pi}/ m)=q$, we find that the term obtained for $j=1$
  contains a factor $(x-1/2)^2$. Finally,   using $4\cos^2 (2{\pi}/ m)=(q-2)^2$, we find that the term obtained for $j=2$
  contains a factor $x^2$,
provided  $m\ge 5$. When
  $m=3$, that is, $q=1$, a factor $x^2$ is found in the term obtained
  for $j=1$.  When $m=4$, that is $q=2$, a factor $x^2$ comes out from
  the term $\hN^2$.

\medskip
We now wish to prove that $x(2\nu x+1)$ divides  $\hQ(0,x)$. We have
just seen that $x=0$ is a double root of $\hP(0,x)$. Using the same
argument as in Section~\ref{sec:diff-general},
this implies that $x$
divides $\hQ(0,x)$. For the remaining factor, we have  to prove that
$(2\nu(x^2-x)+1)$ divides $\hQ(0,x^2-x)$.
Let us combine~\eqref{e2t0T} (taken at $x^2-x$)
with~\eqref{C0-explicitT}. We obtain:
\begin{align*}
\hQ(0,x^2-x)\prod_{i=1}^{m-2}\left(x^2-x-I_i(0)\right)&=\bD(0,x^2-x)\bC'_x(0,x^2-x)- \frac m 2 \bD'_x(0,x^2-x)\bC(0,x^2-x)\\
&=\frac{\hD^{1+m/2}}{2x-1} \frac {\partial} {\partial x}\left( \frac{\hN}{2\sqrt{\hD}}\right)
\Tch'_m\left(\frac{\hN}{2\sqrt{\hD}}\right).
\end{align*}
A factor  $(2\nu(x^2-x)+1)$ is found in
$\frac {\partial} {\partial x}\left( \frac{\hN}{2\sqrt{\hD}}\right)$.
 Lemma~\ref{lem:countT} ensures that $I_i(0)\neq -1/(2\nu$), and thus
 this must be a factor of $\hQ$.

\subsection{The Potts \gf\ of triangulations}
It follows from the first two identities of Proposition~\ref{prop:factorT} that
$$
\hP\left( \bD\bC'_x- \frac m 2 \bD'_x\bC\right)^2
= \hQ^2\left(\bC^2-\bD^m\right).
$$
Using the normalisation~\eqref{normalT}, this gives
\beq\label{idT}
\bP\left(\bD\bC'_x- \frac m 2 \bD'_x\bC\right)^2
= \frac{qm^2\beta^2}4
(q/4-1) \bQ^2\left(\bC^2-\bD^m\right),
\eeq
since $\hQ_2^2= qm^2\beta^2(1+\beta)^2(q/4-1)\hP_3$ by \eqref{tCPQRT}.

Recall that the leading coefficients of $\bC$ are given in
Lemma~\ref{lem:initial-CrT}, and that $C^2-D^m$ has degree $2m-1$. Recall also the known values~\eqref{P3Q2T} of $P_3$ and $Q_2$.
Extracting the coefficients of $x^{2m+3}$, $x^{2m+2}$, and $x^{2m+1}$ in~\eqref{idT}
 first gives (for the coefficient of $x^{2m+3}$) a tautology, then (for the coefficient of $x^{2m+2}$) an interesting relation between $\bP_2$ and $\bQ_1$, namely
\beq\label{P2Q1-init}
4\nu \bP_2= 4 \bQ_1  +{\nu-4},
\eeq
and finally (for the coefficient of $x^{2m+1}$) and expression of
$T_1$ which we transform into~\eqref{Q1-PQt} thanks to~\eqref{P2Q1-init}.

The expression~\eqref{T1-alt} of $T_1'$ was obtained in~\eqref{P3Q2T}.

\subsection{Uniqueness of the solution}
\label{sec:uniqueT}
We now want to show that the differential system of Theorem~\ref{thm:ED-Triang},
together with its initial conditions, uniquely defines the nine series
$P_0$, $P_1$, $P_2$, $P_3$, $Q_0$, $Q_1$, $Q_2$, $R_0$, $R_1$, whether
$q$ is an indeterminate  or a real number of the form $2 + 2\cos
(2\pi/m)$. The proof parallels the case of general planar
maps given in Section~\ref{sec:unique-general} (with $t$ replaced by $w$), and is
actually a bit simpler. We merely sketch the main steps.

 The system can be written as in~\eqref{syst-line}:
\beq \label{syst-line-T}
2Q_w'PD-QP_w'D-2QPD_w'=2R_x'PD-RP_x'D-2RPD_x'.
\eeq
Both sides are
polynomials of degree at most 7 in $x$. Recall that $P_3=1$. Extracting the coefficient of
$x^7$ gives $Q_2'(w)=0$, which, with the initial condition
$Q_2(0)=2\nu$, implies $Q_2(w)=2\nu$. We are left with seven equations and
seven unknown series.

 We denote
$P_{i,j}=[w^i]P_j$, and similarly for $Q$ and $R$. Let
$\cC_i=(P_{i,0},P_{i,1},P_{i,2};Q_{i,0},Q_{i,1};R_{i-1,0}$, $R_{i-1,1})$.
The right-hand side of~\eqref{init-T-thm} gives
$$
\cC_0=(0,0,1/4;0,1;0,0).
$$
%
The next steps are the same as in Section~\ref{sec:unique-general}. For $i\ge 1$ and $j\in
\llbracket 0, 6\rrbracket$, we denote by $\Eq_{i,j}$ the
equation obtained by extracting the coefficient of $w^{i-1} x^j$ in~\eqref{syst-line-T}. Again, the equation $\Eq_{i,0}$  only
involves series from   $\cup_{s<i} \cC_i$ because $P_
{0,0}=P_{0,1}=Q_{0,0}$ as in the case of 
general  maps. However, we do
not have here the second difficulty that came from the factor $(x-1)$
in $P(0,x)$, $P'_x(0,x)$
 and $Q(0,x)$, and we consider the system
$$
\cS_i=\left\{
\Eq_{i+1,0},\Eq_{i,1},\Eq_{i,2},\ldots,\Eq_{i,6}\right\}.
$$
Solving the system $\cS_1$ for the unknowns of $\cC_1$ gives
$$
\cC_1=(-\be, {\be}^{2}+\be q/2-4\be,  4{\be}^{2}+2\be q-2q;
8\be+2q-2{\be}^{2}+\be q,  4{\be}^{3}+2q{\be}^{2}+4{\be}^{2}-2q;
-2\be, \be q-8\be-q).
$$
Then for $i>1$,  $\cS_i$ is a system of seven linear equations for the seven
unknowns in $\cC_i$, the  determinant of which is:
$$
{i}^{5}{q}^{3}{\be}^{7} \left( q-4 \right)  \left( \be+1 \right) ^{4}
 \left( \be-1 \right) ^{3} \left( 4\,{\be}^{2}-q \right)/2.
$$
We conclude as in  Section~\ref{sec:unique-general}.

\medskip
From this point on, we argue as in Section~\ref{sec:conclusion} to
prove that, for $q$ an indeterminate,  the series $P_j$, $Q_j$ and
$R_j$ are differentially algebraic and satisfy~\eqref{Q1-PQt}
and~\eqref{T1-alt}. It then follows from~\eqref{Q1-PQt} that  $T_1$
is  differentially algebraic too.

\subsection{About possible singularities}
It remains to prove  that  the coefficients of the series $P_j,Q_j,R_j$ are
polynomials in $q$ and $\nu$.

First, we note that~\eqref{T1-alt} implies that this is true for
$R_1$.
Consider the system of (only) six equations:
$$
\{\Eq_{i+1,0}, \Eq_{i,1}, \ldots, \Eq_{i,5}\}.
$$
 It relates polynomially the coefficients in $\cup_{s\le
  i} \cC_s$. For $i>1$, it is linear in $ P_{i,0}, P_{i,1}, P_{i,2}, Q_{i,0}$, $Q_{i,1}$,
$R_{i-1,0}$ once the values $P_{1,0}$ and $R_{0,0}$ are known, with determinant:
$$
- { i^5} q \be^9(\be-1)
\left(
  q^3-4q^2(\be^2-\be+1)+2q(\be-1)^2(\be^2-4\be+1)-16\be^2(\be-1)^2\right).
$$
 This determinant contains no factor $(q-4)$, $(\be +1)$ nor
$(4\be^2-q)$. By induction on $i$, we  conclude that the
denominators of the coefficients of $P_j$, $Q_j$ and $R_j$ consist of
factors $q, \be$ and $(\be-1)$.

To show that these factors do not occur, we use two identites that we will prove in Section~\ref{sec:simp-T}
(without assuming polynomiality of the coefficients...):
$$
 \nu q(q-4)Q_0-(4\beta+q)R_0+2(\nu qw(q-4)+\beta)R_1=
 2\beta(q-4)(\nu qw(q-4)+\beta)
$$
and
$$
\nu \beta q(q-4)Q_1-2\nu^2qR_0+\beta(4\beta+q)R_1=\ 2(q-4)\beta^2(4\beta+q).
$$
Let $\Eq'_{i}$ and $\Eq''_i$ be obtained by extracting the coefficient
of $w^i$ in these equations. Consider the system
$$
\{\Eq_{i+1,0}, \Eq_{i,1},\Eq_{i,2} , \Eq_{i,3}, \Eq_i', \Eq_i''\}.
$$
It relates polynomially the coefficients $P_{s,j}, Q_{s,j} $ and $R_
{s,j}$ for $s\le i$. For $i>1$, it is linear in $P_{i,0}, P_{i,1},
P_{i,2}, Q_{i,0}, Q_{i,1}, R_{i,0}$, with determinant
$$
-2 { i^4} q \be^{10}(\be+1)(q-4)(4\be+q).
$$
We have thus ruled out
factors $(\be -1)$. To rule out the factor $q$, we note that $\Eq_i''$
is a multiple of $q$ (because by~\eqref{T1-alt}, all coefficients $R_{i,1}$,
for $i>0$, are multiples of $q$). So if we replace  the last equation of the previous
system by $\Eq_i''/q$, the factor $q$ disappears
from the determinant.

We are left with the factor $\be$. This time we form the system
$$
\{\Eq_{i,3},\Eq_{i,4}, \Eq_{i,5}, \Eq_{i,6}, \Eq_i', \Eq_i''\},
$$
which is linear in the unknowns $P_{i,0}, P_{i,1}, P_{i,2}, Q_{i,0}, Q_{i,1},
R_{i,0}$, with determinant
$$
-2i^4q^5(q-4)(\be-1)(\be+1)^{10}(4\be^2-q).
$$
There is no factor $\be$, and we conclude that the coefficients of $P_j, Q_j$ and $R_j$ belong to $\qs[q,\nu]$.
\section{Simplifying the  system: five non-differential equations}
\label{sec:non-diff}
In this section, we derive 
five
equations between the unknown series $P_j$, $Q_j$ and $R_j$, which do
not involve their derivatives (however, one of them
involves the derivative of $M_1$ or $T_1$, depending on the family of
maps we consider).

Since the system obtained for triangulations is a bit simpler than for
general maps (the degrees of $P$ and $R$ being smaller), we begin with
this case.

\subsection{Triangulations}
\label{sec:simp-T}
We start from the system of Theorem~\ref{thm:ED-Triang}. It involves nine series in
$t$ denoted $P_0, \ldots, P_3$, $Q_0, \ldots, Q_2$, $R_0, R_1$. One
of them is given explicitly in the theorem: $P_3=1$. Moreover, we have
derived from the system in Section~\ref{sec:uniqueT} that
$Q_2=2\nu$. This leaves us with seven unknown series, related by seven
differential equations. These equations are the coefficients of $x^0,
\ldots, x^6$ in~\eqref{syst-line-T}.

We are going to derive three linear (non-differential) relations
between the unknown series by letting $x$ tend to $\infty$ or to
one of the two roots (in $x$) of $D(w,x)$.
Extracting the coefficient of $x^6$ from~\eqref{syst-line-T}  gives
$$
Q'_1=\nu P'_2.
$$
Using the initial conditions~\eqref{init-T-thm}, we obtain
\beq \label{P2Q1-T}
\nu P_2=Q_1+\nu/4-1.
\eeq
 This identity was  already obtained by
expanding~\eqref{idT} in the proof of Theorem~\ref{thm:ED-Triang} (see~\eqref{P2Q1-init}).

\medskip
Let us now specialize~\eqref{syst-line-T} at
 the two roots $\delta_1(w)$ and $\delta_2(w)$ of $D(w,x)$. These
roots, seen as series in $w$, satisfy:
$$
\delta_{1,2}(w)= \beta
\frac{\pm\sqrt{(4-q)(4\beta^2-q)}-4\beta-q}{2q\nu^2}+O(w),
$$
with $\beta=\nu-1$ as above. Since $P(0,x)=x^2(x+1/4)$,  the
$\delta_i$'s are not roots of $P$, and thus taking the limit
$x\rightarrow \delta_i$ in~\eqref{syst-line-T} gives,
for $i=1,2$,
$$
Q(w,\delta_i(w))D'_w(w,\delta_i(w))= R(w,\delta_i(w))D'_x(w,\delta_i(w)).
$$
Since $D$ and its roots are explicit, this is a linear system in
$Q_0$, $Q_1$, $R_0$ and $R_1$, symmetric in $\delta_1$ and $\delta_2$. Solving for $Q_0$ and $Q_1$ gives
\beq
 \nu q(q-4)Q_0-(4\beta+q)R_0+2(\nu qw(q-4)+\beta)R_1=
 2\beta(q-4)(\nu qw(q-4)+\beta)
\label{Q0R-T}
\eeq
and
\beq
\nu
\beta q(q-4)Q_1-2\nu^2qR_0+\beta(4\beta+q)R_1=\ 2(q-4)\beta^2(4\beta+q).
\label{Q1R-T}
\eeq
Let us finally recall the two characterizations of $T_1$ obtained
in Theorem~\ref{thm:ED-Triang}:
\beq\label{Q1-PQt-bis}
20 {\nu}^{2}q\Q_1-4{\nu}^{2} { P_1} +4  \nu{Q_0}+
 \left( { Q_1}-1 \right)  \left( { Q_1}+\nu-3 \right)+2 \nu  \left( q\nu-24 \beta-6 q \right) w =0.
\eeq
and
\beq\label{Tp1-T}
2\nu q T_1' =R_1 - q(\beta-1)+8\beta.
 \eeq
We call Eqs.~(\ref{P2Q1-T}-\ref{Tp1-T})  \emm the five generic
non-differential equations,: they involve the derivative $\tT_1'$, but
no derivative of the other unknown series $P_j$, $Q_j$ and $R_j$.
And of course they hold for indeterminates $q$ and $\nu$.

\medskip\noindent{\bf Prediction of the order of  $\boldsymbol{T_1}$.}
Using (in this order), equations~\eqref{P2Q1-T}, \eqref{Tp1-T}, \eqref{Q1R-T}, \eqref{Q0R-T}
and \eqref{Q1-PQt-bis}, we obtain expressions of $Q_1$, $R_1$, $R_0$, $Q_0$ and finally
$P_1$ as linear combinations of $P_2$, $T_1$ and $T'_1$ (with
coefficients in $\qs(q,\nu,w)$). We only give the connection between
$Q_0$, $P_2$ and $T_1$, which we will use later:
\beq\label{T2prime}
2\left( 4\nu ^3q^2w +\beta(4\beta^2-q)\right)  \tT _1'/\nu+2q\nu
Q_0-\beta(q+4\beta)P_2 = (4\beta+q)(4q\nu w-\be/4).
\eeq
Recall that $T_1=\nu T_2$ where $T_2$ counts near-triangulations with outer degree 2.

We have thus got rid of five unknown
series, but introduced a new one, namely $T_1$. Reporting the
expressions of $Q_1, R_1, R_0, Q_0$ and $P_1$  in  our differential system gives a system
of three equations in  $P_0$, $P_2$ and $T_1$, of order 1 in
$P_0$ and  $P_2$ and order 2 in $T_1$. Thus the field generated over
$\qs(q,\nu,w)$ by $P_0, P_2, T_1$ and their
derivatives has transcendence degree at most 4, and we expect $T_1$ to
satisfy a differential equation of order at most 4. We have been able to eliminate
$P_0$ and to obtain two differential equations in $P_2$ and $T_1$: one
of order 1 in $P_2$ and 2 in $T_1$, the other of respective orders 2
and 3.  Going further with the elimination seems to require  heavy
computer algebra, and we have failed to obtain an explicit equation for
$T_1$. We study 
in Sections~\ref{sec:nu=0} to~\ref{sec:q=0} three
special cases, of combinatorial interest, where the order of $T_1$ is only 2.

\subsection{General planar maps}
\label{sec:simpl-gen}
We now proceed similarly with the differential system of
Theorem~\ref{thm:ED}. It involves eleven series in $t$ denoted $P_0, \ldots,
P_4$, $Q_0, \ldots, Q_2$, $R_0, \ldots, R_2$. Two of them are given
explicitly: $P_4=1$ and $R_2=\nu+1-w(q+2\beta)$. We have also seen in
Section~\ref{sec:unique-general} that $Q_2=1$. This leaves us with eight unknown series, related by a system
of eight differential equations.

As in the case of triangulations, we can derive three
linear (non-differential) relations between them by considering the system~\eqref{syst-line} as
$x$ approaches $\infty$, or one of the two roots (in $x$) of the
quadratic polynomial $D(t,x)$. The derivation is the same as in the
previous subsection, and we simply give the resulting three identities.

The first one was already obtained in the proof of Theorem~\ref{thm:ED} by
expanding~\eqref{id} in $x$ (see~\eqref{P3Q1-init}):
\beq\label{P3Q1}
\bP_3=2\bQ_1+4t(1+\nu)-4tw(2\beta+q).
\eeq
The other two are the counterparts of~\eqref{Q0R-T} and~\eqref{Q1R-T}, and they
read:
\begin{multline}
  \beta  \left( wq+\beta \right)  \left( q-4 \right) Q_0
+q \left( \beta+2 \right) R_0
+ 2\left(
 \beta  \left( q-4 \right)  \left( wq+\beta \right) t+ q \right)
R_1 =\\
2 \beta  \left( q-4 \right)  \left( wq-2
 \right)  \left( wq+\beta \right) t+2 q \left( wq-2 \right) ,
\label{Q0R-M}
\end{multline}
\begin{multline}
\beta  \left( wq+\beta \right)  \left( q-4 \right) Q_1
-2 \left(  {\beta}^{2}+ q\beta+ q \right) R_0
  -q \left( \beta+2 \right)R_1 = \\
2 \beta  \left( q-4 \right)  \left( 2 \beta w+wq-\beta-2 \right)
 \left( wq+\beta \right) t-2 q \left( \beta qw-2 \beta w+wq-\beta-
2 \right) .
\label{Q1R-M}
\end{multline}
Let us finally recall the two characterizations of $\tM_1=t^2M_1$ obtained
in Theorem~\ref{thm:ED}:
\begin{multline}\label{M11-PQ-encore}
12  \left( {\beta}^{2}+q\nu \right) \tM_1  +
 P_3 ^{2}/4
+2 t \left(1+\nu -w(2 \beta+q) \right)P_3  -P_2
 +2 Q_0
=4 t \left(1+  w(3 \beta +q) \right) ,
\end{multline}
and
\beq\label{Mt1-expr}
2\left( {\beta}^{2}+q\nu\right) \tM_1'
+ \left( 1+\nu- w\left( 2\beta+q \right)  \right) P_3 /2-R_1
=2+2 \beta w .
\eeq
Again, we call Eqs.~(\ref{P3Q1}-\ref{Mt1-expr})  \emm the five generic
non-differential equations,: they involve the derivative $\tM_1'$, but
no derivative of the other unknown series $P_j$, $Q_j$ and $R_j$.
And of course they hold when  $q$,
$\nu$ and $w$ are indeterminates.

\medskip
\noindent{\bf Prediction of the order of  $\boldsymbol{M_1}$.}
Using (in this order) the five
equations~\eqref{P3Q1},~\eqref{Mt1-expr},~\eqref{Q1R-M},
\eqref{Q0R-M}, and~\eqref{M11-PQ-encore},  we can express $Q_1$, $R_1$,
$R_0$, $Q_0$,
and finally $P_2$   as polynomials
in $P_3$, $\tM_1$ and $\tM_1'$
(with coefficients in $\qs(q,\nu,w,t)$).
We  thus get rid of five
unknown series (but introduce a new one, namely $\tM_1$). Reporting
these five expressions in our differential system gives a system
of four equations in
$P_0, P_1,  P_3$ and $ \tM_1$, of order 1 in
$P_0, P_1,  P_3$ and order 2 in $\tM_1$.
We thus expect $\tM_1$ to satisfy a differential equation of order at
most 5. We study in Sections~\ref{sec:q=4} and~\ref{sec:sd} two
special cases, of combinatorial or physical
interest, where the order is only 3.

\section{Properly coloured triangulations: $\boldsymbol{\nu=0}$}
\label{sec:nu=0}
We 
 specialize the differential system obtained for triangulations to  $\nu=0$, and recover the solution of the problem studied by Tutte between 1973 and 1984. Recall
that, for  generic values of $q$ and $\nu$, we have  $ \tT_1=
\nu \tT _2$, where $\tT _i$ counts
$q$-coloured triangulations with outer degree~$i$. Thus $\tT _1$ vanishes when
$\nu=0$,  and we focus on $ T_2$ instead, as Tutte did.

\begin{Theorem}[{\bf Tutte~\cite{tutte-differential}}]
\label{thm:nu0}
Let $\tT_2\equiv \tT_2(q,w)$ denote the \gf\ of properly $q$-coloured
near-triangulations of
outer degree $2$,  counted by the number of
vertices (as before, the root-vertex is coloured in a prescribed colour). This series is characterized by
$$
2(1-q)w+(w+10 \tT_2-6w\tT'_2) \tT''_2+(4-q)(20
\tT_2-18w\tT_2'+9w^2\tT_2'')=0,
$$
with the initial conditions $\tT_2=O(w^2)$.
Equivalently, $T_2=\sum_{n\ge 2} a_n(q)w^n$, where $a_n\equiv a_n(q)$
is given by $a_2=q-1$ and for $n>0$,
$$
(n+1)(n+2)a_{n+2}=(q-4)(3n-1)(3n-2)a_{n+1} + 2\sum_{i=1}^n
i(i+1)(3n-3i+1)a_{i+1}a_{n+2-i}.
$$
\end{Theorem}
\noindent{\bf Remark.} Take a properly
coloured near-triangulation of outer degree 2, having at least one
finite face (that is, having at least three vertices). If we delete
its root-edge, we obtain a properly coloured triangulation with the
same number of vertices. This means that when $\nu=0$,
$$
T_2=(q-1)w^2+T_3,
$$
where $T_3$ counts properly $q$-coloured triangulations (with the root-vertex coloured in a prescribed colour). The series
$H$ occurring in~\eqref{Tutte-ED} is $qT_2$, and Tutte's differential
equation~\eqref{Tutte-ED} is  equivalent to the above theorem.
\begin{proof}
 We specialize Theorem~\ref{thm:ED-Triang} to  $\nu=0$, that is
 $\beta=-1$. The polynomial $D(w,x)$ becomes
independent of $w$ and has degree  one in $x$:
\beq\label{D-nu0}
D(w,x)=(4-q)x+1.
\eeq
Let us now specialize the five generic non-differential identities
(\ref{P2Q1-T}-\ref{Tp1-T}). The first three  give
\beq\label{R-val-4-T}
Q_1=1, \quad R_0=2 \quad\hbox{and}\quad R_1=2(4-q).
\eeq
Observe that  $R(w,x)=2D(w,x)$. The identities~\eqref{Q1-PQt-bis}
and~\eqref{Tp1-T} then become tautologies at $\nu=0$, and we still need to
relate $T_2$ to the series $Q_j$ and $P_j$.  We use
for that the equation~\eqref{T2prime}.
Once specialized at $\nu=0$,
it reads:
\beq\label{T2t-nu0}
2\tT _2'+P_2=1/4.
\eeq
This will replace the identity~\eqref{Tp1-T} relating $T_1'$ and $R_1$. We
proceed without deriving a counterpart of~\eqref{Q1-PQt-bis}.

We now go back to the differential system~\eqref{deT}, which simplifies into
\beq\label{syst-T-4}
2\,\frac\partial{\partial x} \left( \frac1 P\right)= \frac 1 {QD}
\frac{\partial}{\partial w}\left( \frac{Q^2}{P}\right).
\eeq
This is a rational expression in $x$.
Given the form~\eqref{D-nu0} of
$D$, we find convenient to expand the
numerator of~\eqref{syst-T-4} in powers of $X=x+1/(4-q)$ rather than $x$. We write
accordingly:
$$
P(w,x)= X^3+\tilde P_2X^2+\tilde P_1X+\tilde P_0,
\quad 
Q(w,x)= X+\tilde Q_0, \quad
R(w,x)=2(4-q) X,
$$
where the $\tP_j$'s and $\tQ_0$ are series in $w$ with rational
coefficients in $q$. Since we have taken
into account the values~\eqref{R-val-4-T} of $Q_1$, $R_0$ and
$R_1$, we have only four unknown series left. The  identity~\eqref{T2t-nu0} gives
\beq\label{Tt2-nu0-tilde}
\tP_2=\frac{q+8}{4(q-4)}-2\tT'_2.
\eeq
The following   initial conditions are translated from~\eqref{init-T-thm}:
\beq\label{init-0-T}
\tP_0 (0)= \frac q{4(q-4)^3}, \qquad
\tP_1(0)= \frac {2+q}{4(q-4)^2}, \qquad
\tP_2(0)= \frac {8+q}{4(q-4)},\qquad
\tQ_0(0)= \frac 1{q-4}.
\eeq
Expanding in $X$ the numerator of~\eqref{syst-T-4} gives a system of
four differential equations between the
four series $\tP_0, \tP_1$, $\tP_2$ and $\tQ_0$:
\begin{align*}
  2\tP_0 \tQ_0'-\tQ_0\tP_0'&=0,
\\
2\tQ_0'\tP_1-\tQ_0\tP'_1-\tP_0'+2(4-q)\tP_1 &=0,
\\
2\tQ_0'\tP_2-\tQ_0\tP'_2-\tP_1'+4(4-q)\tP_2 &=0,
\\
2\tQ_0'-\tP'_2+6(4-q)&=0.
\end{align*}
The first and last equations are readily solved.  With the initial
conditions~\eqref{init-0-T}, the last one gives
\begin{align}
\tQ_0&=\frac 1 2  \tP_2+3(q-4)w-\frac q{8(q-4)},\nonumber
\\
&=-\tT_2'+3(q-4)w+\frac 1{q-4},\label{Q0T2-T-nu0}
\end{align}
thanks to~\eqref{Tt2-nu0-tilde},  while the first one gives
\beq\label{P0Q0-T-nu0}
4(q-4)\tP_0=q\tQ_0^2=  q\left(-\tT_2'+3(q-4)w+\frac 1{q-4} \right)^2.
\eeq
 Reporting the  expressions~\eqref{Tt2-nu0-tilde} and~\eqref{Q0T2-T-nu0} of
 $\tP_2$ and $\tQ_0$ in the third differential equation gives an
equation in $\tP_1$ and $\tT_2$ that we can integrate (using the
initial conditions~\eqref{init-0-T}):
$$
\tP_1= (\tT_2')^2+\tT'_2\left(
6(q-4)w-\frac{q+4}{2(q-4)}\right)+10(4-q)\tT_2+\frac w 2 (q+8)
+\frac{2+q}{2(q-4)^2}.
$$
Reporting  this, as well as~\eqref{Q0T2-T-nu0} and~\eqref{P0Q0-T-nu0}   in the second differential equation finally gives
Tutte's equation for the \gf\ of properly $q$-coloured
near-triangulations of outer degree~2. The rest of the proof is straightforward.
\end{proof}


\section{Four colours: $\boldsymbol{q=4}$}
\label{sec:q=4}
%
Several signs suggest that the case of four colours should be
simpler. First, this is visible on the differential equation and
recurrence relation of Theorem~\ref{thm:nu0}. Then, the polynomial
$D(t,x)$, given by~\eqref{D-encore} (or its counterpart $D(w,x)$ given
by~\eqref{D-exprT} for triangulations) becomes a square when $q=4$, independent
of the size variable $t$ (or $w$). We  use this simplification to
derive a differential equation for the \gf\ of four-coloured maps, of
a relatively small order (3 for general maps, 2 for triangulations).

\subsection{Triangulations}
\label{sec:tri-4}
\begin{Theorem}
  Let $\tT_1\equiv \tT_1(\nu,w)$ be the \gf\ of four-coloured
near-triangulations of outer degree $1$, counted by the number of
  vertices ($w$) and the number of monochromatic edges $(\nu)$ (as before, the root-vertex is coloured in a prescribed colour). This
  series is characterized by the initial conditions
  $\tT_1=\nu(3+\nu)w^2+O(w^3)$  and the following differential
  equation of order $2$ and degree $6$, which we write in terms of
 $S=2\tT_1-w$:
$$
P(S, S', S'')=0
$$
with
\begin{multline*}
 P(X,Y,Z)=48  \beta ^4  \nu  ^2  w^2  Y^4 Z^2 +3  \beta
^5  \alpha   w  Y^4 Z^2+10240  \nu  ^4  \beta  ^2 w^3 Y^3  Z^2 -5  \beta ^5  \alpha   X   Y^3 Z^2\\
-8  \beta
^4 w  (10  X  \nu  ^2-\alpha  ^2)   Y^3 Z^2
 +768  \nu  ^2  \alpha   \beta
^3  w^2   Y^3 Z^2+ 589824  \nu ^6 w^4 Y^2 Z^2 +49152  \nu  ^4  \beta    \alpha   w^3  Y^2  Z^2\\
 -15  \beta ^4  \alpha ^2  X  Y^2 Z^2 -288  \beta
^2  \nu  ^2  w^2 (80  X  \nu  ^2-3  \alpha  ^2)  Y^2  Z^2 -6  \alpha   \beta
^3 w  (280  X  \nu  ^2-\alpha  ^2)  Y^2  Z^2 \\
+15   \alpha   \beta ^3  (40  X  \nu  ^2-\alpha  ^2) X Y  Z^2
-138240   \nu  ^4  \beta
  \alpha  w^2 X  Y Z^2 -2160  \nu  ^2 \beta
^2  \alpha ^2  w  X  Y  Z^2 \\
-6144  \nu  ^4 w^3  (320  X  \nu  ^2+\alpha  ^2)  Y Z^2 +5   \beta ^2  (160  X  \nu  ^2-\alpha  ^2)   (20  X  \nu  ^2+\alpha  ^2) X  Z^2\\
+80 \nu ^2w^2 (128 X \nu ^2+\alpha ^2) (160 X \nu ^2-\alpha ^2) Z^2 +
\beta     \alpha   (160  X  \nu  ^2-\alpha  ^2)  (560  X  \nu
^2+\alpha  ^2) w  Z^2 \\
+24 \nu ^2 \beta ^4 wY^5 Z -40  \beta
^4  \nu  ^2   X Y^4  Z
+40  \beta  ^3  \nu  ^2  \alpha   w   Y^4 Z
+3584  \nu  ^4  \beta ^2 w^2   Y^4 Z +98304 Z Y^3 \nu ^6 w^3\\
+1536  \nu  ^4  \beta    \alpha   w^2  Y^3  Z-24  \nu  ^2  \beta
^2 w  (320  X  \nu  ^2-\alpha  ^2)   Y^3 Z
-60  \nu ^2 \alpha  \beta ^3 X Y^3 Z -8  \beta   \alpha   \nu  ^2 w  (160  X  \nu  ^2-\alpha  ^2)    Y^2 Z\\
+4 \nu ^2 (160  X  \nu  ^2-\alpha  ^2) (5 \beta ^2 X-256 \nu ^2 w^2) Y^2 Z
+4096 \nu  ^6  w^2  Y^4+ 128 \nu ^4 \alpha  \beta  Y^4 w+160 X \nu ^4 \beta ^2 Y^4,
\end{multline*}
where we have written $\beta=\nu-1$ and $ \alpha=\nu-2$.
\end{Theorem}
\begin{proof}
  We specialize Theorem~\ref{thm:ED-Triang} to $q=4$. The polynomial
  $D(w,x)$ becomes a  square, and is independent of $w$:
$$
D(w,x)=(2\nu x+\be)^2.
$$
We we use extensively the notation $\be=\nu-1$, $\al=\nu-2$.
Let us specialize to $q=4$ the five non-differential identities
(\ref{P2Q1-T}-\ref{Tp1-T}).
Equation~\eqref{P2Q1-T}  holds verbatim. The
identities~\eqref{Q0R-T} and~\eqref{Q1R-T} both specialize to
$$
2\nu R_0=\be R_1,
$$
so that $R(w,x)$ is a multiple of $\sqrt{D(w,x)}=2\nu
x+\be$. However, we can also derive  a second relation from~\eqref{Q0R-T}
and~\eqref{Q1R-T}: if we first eliminate $R_0$ between them, and then set
$q=4$, we  obtain
\beq\label{RQ-T-q4}
(16\nu ^2w+\al\be)R_1-4\nu  \be Q_1+8\nu ^2Q_0+4\nu \be^2=0.
\eeq
We also have
the two relations~\eqref{Q1-PQt-bis}
and~\eqref{Tp1-T} relating $\tT_1$ to the other unknown series.

We now go back to the differential system~\eqref{deT}. We will expand its
numerator in powers of $X=x+\be/(2\nu)$. 
We write accordingly
$$
P(w,x)= X^3+\tilde P_2X^2+\tilde P_1X+\tilde P_0,
\qquad
Q(w,x)= 2\nu X^2+\tQ_1 X+\tilde Q_0, \qquad  R(w,x)=R_1 X,
$$
where the $\tP_j$'s and $\tQ_j$ are series in $w$ with rational
coefficients in $\nu$. Note that we have six unknown series instead of seven
because $R(w,x)$ is a multiple of $X$ (in other words, $\tR_0=0$). The
remaining four generic identities, namely~\eqref{P2Q1-T},
\eqref{RQ-T-q4}, \eqref{Q1-PQt-bis}
and~\eqref{Tp1-T}, translate as follows:
\begin{align}
\label{Q1P2t-4}
\tQ_1= &\ \nu \tP_2-3\al/4,
\\
\label{Q0sol-4-T}
8\,{\nu}^{2}\tQ_0 =&\ - \left( 16\,{\nu}^{2}w+\al\be \right)R_1  ,
\\
\label{P1sol-4-T}
\nu \tP_1     =&\
10\nu(2\tT_1-w)  +{ {\tQ_0  }}+\nu  {\tP_2  } ^{2}/4\,+3\al \tP_2 /8\,+{ {5}}{ {\al ^{2}}{}}/(64\nu),
\\
\label{R1sol-4-T}
R_1   =&\ 4\nu (2\tT_1 '-1).
\end{align}
The following  initial conditions are translated from~\eqref{init-T-thm}:
\beq\label{init-4-T}
  \tP_0(0)=\displaystyle-{\frac { \be ^{2}\al }{16{\nu
}^{3}}}
,
\qquad \tP_1(0)=
\displaystyle{\frac { \be  \left( \al+\be \right) }{4{\nu}^{
2}}}
,\qquad
 \tQ_0(0)= \displaystyle{\frac { \be \al }{2\nu}}
,
\qquad \tQ_1(0)= -\al-\be.
\eeq

 Expanding in $X$ the numerator
of~\eqref{deT} gives a system of five differential equations in the six
series $\tP_0, \tP_1, \tP_2, \tQ_0, \tQ_1, \R_1$:
\begin{align*}
  2\tP_0 \tQ_0'-\tQ_0\tP_0'+2\tP_0R_1&=0,
\\
2\tQ_0'\tP_1+2\tP_0\tQ_1'-\tP_0'\tQ_1-\tP_1'\tQ_0 +3\tP_1R_1&=0,
\\
2\tQ_0'\tP_2+2\tP_1\tQ_1'-\tP_1'\tQ_1-\tP_2'\tQ_0 -2\nu\tP_0'+4\tP_2R_1&=0,
\\
2\tQ_1'\tP_2-\tQ_1\tP'_2-2\nu\tP_1'+2\tQ_0'+5R_1 &=0,
\\
\tQ_1'-\nu\tP'_2&=0.
\end{align*}
 We will need to inject an additional equation
 (namely~\eqref{Q0sol-4-T}) to complete this system.

The fifth (and last) equation of the system is solved by~\eqref{Q1P2t-4}.

The fourth equation is a consequence of~\eqref{Q1P2t-4}, \eqref{P1sol-4-T}
and~\eqref{R1sol-4-T}.

In the first equation of the system, let us replace $\tR_1$ by its
expression in terms of $\tQ_0$ derived from~\eqref{Q0sol-4-T}. Using
the initial conditions~\eqref{init-4-T}, the  resulting equation (in $\tP_0$ and
$\tQ_0$) can now be integrated into:
\beq\label{P0sol-4-T}
4 \nu  \left( 16\,{\nu}^{2}w+\al\be \right) \tP_0
= -\be \tQ_0    ^{2}.
\eeq

Let us finally  eliminate $\tR_1$ between the  first and second equations of the
system. The resulting  equation can
be integrated into
$$
2\tQ_1= \frac{\tP_1 \tQ_0}{\tP_0} + c \sqrt {\tP_0},
$$
for some constant $c$. Using the initial
conditions~\eqref{init-4-T}, this constant is found to be zero, so
that
\beq\label{Q1sol-4-T}
2\tQ_1\tP_0=\tP_1\tQ_0.
\eeq

This is as far as we have been able to go in the integration of the
system. At this stage, we have seven unknown series (the $\tP_j$'s,
$\tQ_j$'s, $R_1$ and $\tT_1$) related by one differential equation (the
third one in the system), and six ``non-differential'' equations (namely
(\ref{Q1P2t-4}--\ref{R1sol-4-T}),
\eqref{P0sol-4-T} and \eqref{Q1sol-4-T}) in which the only derivative is
$\tT_1'$. Using in the following order
\eqref{Q1P2t-4},~\eqref{R1sol-4-T},~\eqref{Q0sol-4-T},~\eqref{P1sol-4-T},
and~\eqref{P0sol-4-T}, we can now express $\tQ_1$, $R_1$, $\tQ_0$, $\tP_1$,
$\tP_0$  in terms of $\tP_2$, $\tT_1'$  and
$\tT_1$. Plugging these expressions in the remaining non-differential
equation~\eqref{Q1sol-4-T} and in the remaining differential equation
gives a pair of differential  equations
between $\tP_2$ and $\tT_1$. Given the form of~\eqref{P1sol-4-T} and~\eqref{R1sol-4-T},
it makes sense to write them in terms of $S=2T_1-w$. They read:
%
$$
4\left(4  \be \nu \tP_2+128 {\nu}^{
2}w+5  \al\be \right)
S'=
 16 {\nu}^{2} { \tP_2  } ^{2}+24 \nu  \al \tP_2 +
640{\nu}^{2}S  +5 \al^{2},
$$
\begin{multline*}
2\left( 16 {\nu}^{2}w+\al\be \right)
\left(  4 \nu   \tP_2 +3 \al-2 \be   S'     \right) S''
-32 {
\nu}^{2} \be{S'} ^{2}\\
+8 \nu  \left(  \left( 16 {\nu}^{2}w+\al\be \right) \tP_2' +4 {\nu}^{2}\tP_2 -3 {\nu}\al \right) S'
- \left( 12 {\nu
}^{2}\al \tP_2 +320  {\nu}^{3}S+7 \al ^{2}\nu \right) \tP_2 '=0.
\end{multline*}
We finally eliminate $\tP_2$ as follows: we differentiate the first
equation above to obtain  a
total of three \emm polynomial, equations in $S, S', S''$, $\tP_2$
and $\tP_2'$, from which we eliminate $\tP_2$
and $\tP_2'$ using resultants. This yields the
differential equation
of  the theorem.

To finish, one checks that the recurrence relation that underlies the
differential equation for $T_1$ determines all coefficients of this
series once  we prescribe the first three coefficients of $T_1$.
\end{proof}
\subsection{General planar maps}
\begin{Theorem}
  Let $M_1\equiv M_1(\nu,w;t)$ be the \gf\ of four-coloured planar maps, counted by
  vertices ($w$), edges ($t$) and  monochromatic edges $(\nu)$ (as
  before, the root-vertex is coloured in a prescribed colour). This
  series is characterized by the initial conditions
  $M_1=w+w(\nu +3w+\nu w)t  +O(t^2)$  and a  differential
  equation of order $3$ and degree $11$.
\end{Theorem}
The differential equation can be seen in the Maple session
accompanying this paper, on the authors' web pages.
\begin{proof}
We  specialize Theorem~\ref{thm:ED} to $q=4$. The polynomial $D(t,x)$ is
again a perfect square:
$$
D(t,x)=\left( (\nu+1)x-2\right)^2.
$$
Equation~\eqref{P3Q1} still holds of course. The identities~\eqref{Q0R-M}
and~\eqref{Q1R-M} both specialize to
$$
 \left( \nu+1 \right) R_0 +2R_1 +4(1-2w)=0.
$$
Given that $R_2=(\nu+1)(1-2w)$  (this is one of the initial conditions
in Theorem~\ref{thm:ED}), this means that $R(t,x)$ has a factor
$\sqrt{D}= (\nu+1)x-2$. As in the case of triangulations, we can
derive a second identity from~\eqref{Q0R-M}
and~\eqref{Q1R-M} by eliminating $R_1$
between them:
\begin{multline}\label{QR-M4}
\left(\beta - \left( \beta+2 \right) ^{2} \left( \beta+4 w \right) t
 \right)R_0
+  \left( \beta+2 \right)
 \left( \beta+4 w \right) Q_0+
2\left( \beta+4w \right)Q_1 =\\
4  \left( 2 w-1 \right)
 \left( \beta+2 \right)  \left( \beta+4 w \right) t-8w,
\end{multline}
with $\be=\nu-1$. The two characterizations~\eqref{M11-PQ-encore} and~\eqref{Mt1-expr} of
$\tM_1=t^2M_1$ still hold.

We now go back to the differential system~\eqref{de-encore}. We will expand its
numerator in powers of $X=x-2/(\nu+1)$ instead of $x$. We write
accordingly
$$
P(t,x)= X^4+\tP_3 X^3+\tP_2 X^2+\tP_1 X+\tP_0, \qquad
Q(t,x)=X^2+\tQ_1X+\tQ_0,
$$
and
$$
R(t,x)=(\nu+1)(1-2w)X^2+\tR_1X,
$$
where the $\tP_j$, $\tQ_j$ and $\tR_j$ are series in $t$ with rational
coefficients in $\nu$ and $w$. We have taken into account the fact
that $R$ is  a multiple of $X$.
The remaining
four identities, namely~\eqref{P3Q1}, \eqref{QR-M4}, \eqref{M11-PQ-encore}
and~\eqref{Mt1-expr},
translate as follows:
\beq\label{Q1P3t}
\tQ_1= \tP_3/2+2(2w-1)(\beta+2)t,
\eeq
\beq\label{Q0sol-4}
2((\beta+2)^2 (\beta+4 w) t-\beta) \tR_1+(\beta+2)^2 (\beta+4 w)\tQ_0=0,
\eeq
\beq\label{P2sol-4}
\tP_2 =
12  (\beta+2)^2 \tM_1+2 \tQ_0+ \tP_3^2/4-2 (2 w-1) (\beta+2) t \tP_3-12 t w \beta-12 t,
\eeq
\beq\label{R1sol-4}
\tR_1= 2  (\beta+2)^2 \tM_1'-(2 w-1) (\beta+2) \tP_3/2-2 \beta w-2.
\eeq
Eq.~\eqref{P2sol-4} actually results from the elimination of $\tQ_1$
between~\eqref{M11-PQ-encore} and~\eqref{Q1P3t}.

The following initial conditions are translated from~\eqref{init-general}:
\beq\label{init-M4}
\displaystyle \tP_0(0) =\frac{4\beta^2}{(2+\beta)^4},\quad
\displaystyle \tP_1(0)=\frac{4\beta(\beta-2)}{(2+\beta)^3} ,\quad
\displaystyle \tQ_0(0)= -\frac{2\beta}{(2+\beta)^2}, \quad
\displaystyle \tQ_1(0) = \frac{2-\beta}{2+\beta}.
\eeq
Expanding in $X$ the numerator
of~\eqref{de-encore} gives a system of (only) six equations in the seven unknown
series $\tP_0$, $\tP_1$, $\tP_2$, $\tP_3$, $\tQ_0$, $\tQ_1$, $\tR_1$:
\begin{align*}
  2\tP_0 \tQ_0'-\tQ_0\tP_0'+2\tP_0\tR_1 &=0,
\\
2\tQ_0'\tP_1+2\tP_0\tQ_1'-\tP_0'\tQ_1-\tP_1'\tQ_0 +3\tP_1\tR_1 &=0,
\\
2\tQ_0'\tP_2+2\tP_1\tQ_1'-\tP_1'\tQ_1-\tP_2'\tQ_0
-\tP_0'+4\tP_2\tR_1 -(2w-1)(\beta+2)\tP_1&=0,
\\
2\tQ_0'\tP_3+2\tP_2\tQ_1'-\tP_2'\tQ_1-\tP_3'\tQ_0 -\tP_1'+5\tP_3\tR_1 -2(2w-1)(\beta+2)\tP_2&=0,
\\
2\tQ_1'\tP_3-\tQ_1\tP'_3-\tP_2'+2\tQ_0'+6\tR_1 -3(2w-1)(\beta+2)\tP_3 &=0,
\\
2\tQ_1'-\tP'_3-4(2w-1)(\beta+2)&=0.
\end{align*}
We will need another equation (namely~\eqref{Q0sol-4}) to complete
this system.

The sixth (and last) equation of the system is solved by~\eqref{Q1P3t}.

The fifth equation is a consequence of~\eqref{Q1P3t}, \eqref{P2sol-4}
and~\eqref{R1sol-4}.

In the first equation of the system, we now replace $\tR_1 $ by its
expression in terms of $\tQ_0$ derived from~\eqref{Q0sol-4}. The
resulting equation (in $\tP_0$ and
$\tQ_0$) can now be integrated into:
\beq\label{P0sol-4}
((\beta+2)^2 (\beta+4 w) t-\beta)\tP_0+\beta\tQ_0^2=0.
\eeq
We have used the initial conditions~\eqref{init-M4} to determine the
integration constant.

Let us finally eliminate $\tR_1$ between  the first and
second equations of the
system.  The resulting equation  can
be integrated into
$$
2\tQ_1= \frac{\tP_1 \tQ_0}{\tP_0} + c \sqrt {\tP_0},
$$
for some constant $c$. Using the initial
conditions~\eqref{init-M4},
this constant is found to be zero, and
\beq\label{P1sol-4}
2\tQ_1\tP_0=\tP_1\tQ_0.
\eeq

This is as far as we have been able to go in the integration of the
system. Using (in this
order)~\eqref{Q1P3t},~\eqref{R1sol-4},~\eqref{Q0sol-4},~\eqref{P2sol-4},~\eqref{P0sol-4},
and~\eqref{P1sol-4}, we can now express $\tQ_1$, $\tR_1$, $\tQ_0$, $\tP_2$,
$\tP_0$ and finally $\tP_1$ in terms of $\tP_3$, $\tM_1'$  and
$\tM_1$. Plugging these expressions in the remaining equations of the
system (the third and fourth) gives a pair of differential equations
between $\tP_3$ and $\tM_1$.
Both have order $1 $ in $\tP_3$ and order 2 in $\tM_1$.
Denoting $\la=2+\be$ and $p= \la^2t (\be+4w)$, they read:
\begin{multline*}
  16\la ^2\left(p-\beta\right)\left(t\tM_1'-3 \tM_1\right)
  \tM_1''
-p\tP_3^2\tM_1''
+ 4\la t(2w-1)\left(p-3\beta\right)\tP_3
\tM_1''
\\
+16t \left( t\la ^2 (2\beta^2w+\beta(2w+1)^2+8w)-2\beta(\beta w+1)\right)\tM_1''
+8\la ^2p\,(\tM_1')^2
\\
-4t\la (2w-1) \left(
2
p-3\be \right)\tP_3'\tM_1'
+ p\left( \tP_3' -2\la(2w-1)\right)\tP_3 \tM_1'
+12\la (2w-1)\left(p-\beta\right)\tP_3' \tM_1
\\
-8t\la^2\left(4\be w^2+2w(\be^2+2\be+4)+\be\right) \tM_1'
-tw\la^2 \left( \tP_3+4\la t(2w-1)\right)
\tP_3'
\\
+2tw\la^3(2w-1) \tP_3 +8tw\la^2(1+w\be)=0,
\end{multline*}
and
\begin{multline*}
   \left(p \tP_3/\la
-4t(2 w-1)(p-2\be)  \right)\tM_1''\\
+p \left(\tP_3'/\la+12w-6 \right)\tM_1'
-3p/t\left(\tP_3'/\la+4w-2 \right) \tM_1
+ 2tw\la  \tP_3'=0.
\end{multline*}
By eliminating $\tP_3$ between them, we obtain a differential equation
of order 3 for $\tM_1$ (we first
eliminate $\tP_3'$ between the two equations to obtain one equation
between $\tP_3$, $\tM_1$, $\tM_1'$ and $\tM_1''$; then we follow the
steps described at the end of Section~\ref{sec:tri-4}).

To finish, one checks that the recurrence relation that underlies the
differential equation for $\tM_1$ determines all coefficients of this
series once the coefficients of $t^2$ and $t^3$ in $\tM_1$ are
prescribed.
\end{proof}

\section{Spanning forests of cubic maps:  $\boldsymbol{q=0}$}
\label{sec:q=0}
When $q=0$, the  series $T_1$ has an interesting combinatorial interpretation in
terms of spanning forests, as discussed in
Section~\ref{sec:specializations}.  Moreover, the polynomial $D(w,x)$ involved in the differential
system for triangulations
simplifies
greatly, and this will allow us to characterize $T_1$ by a second
order equation.
We say a map is \emm near-cubic, if all non-root vertices have degree 3.

\begin{Theorem}
\label{thm:forestT}
Let $G(\beta,w)$ be the \gf\ of near-cubic planar maps with a root-vertex of degree $1$, equipped with a spanning forest, counted by the number of
faces ($w$) and  the number of connected components of the forest,
minus one ($\beta$). This series is characterized by the initial
conditions $G=O(w^2)$ and the following differential equation, written in
terms of $W=2G-w/\beta$:
\begin{multline*}
0= \left( 3 {\beta }^{4}w {{W}'} ^{4}-{\beta }^{3} ( 5 W
      \beta -\beta w+w )  {{W}'} ^{3}+4\, ( \beta +1 ) (
      5\,W \beta -\beta w+w ) ^{2} \right) {W}''\\
 -48\,{\beta }^{2}w ( \beta +1)  {{W}'} ^{3}+8\,\beta  ( \beta +1 )  ( 5\,W \beta -\beta w+w )  {{W}'} ^{2}+4\, ( \beta ^2-1) ( 5\,W \beta -\beta w+w ){{W}'}.
\end{multline*}
\end{Theorem}
\noindent{\bf Remark.} Theorem~\ref{thm:forestT} has been  rederived in~\cite{mbm-courtiel} using a simpler, and
more combinatorial, argument.
\begin{proof}
Let us  apply~\eqref{forests} to the set $\cM$  of
near-triangulations with outer degree 1.  We obtain:
\beq\label{T10G}
 \tT_1\big|_{q=0}
=\frac 1 {\be} G(\be, w\be),
\eeq
where $\be=\nu-1$ and $G(\beta,w)$ is defined above.

Hence let us specialize Theorem~\ref{thm:ED-Triang} to $q=0$. The polynomial
    $D(w,x)$ has degree 1 in $x$, and is independent of $w$:
$$
D(w,x)=\beta^2(1+4x).
$$
Let us now specialize the five generic non-differential equations
(\ref{P2Q1-T}-\ref{Tp1-T}). Equation~\eqref{P2Q1-T}  holds verbatim.
The identities~\eqref{Q0R-T} and~\eqref{Q1R-T}  give $R_0=-2\be$ and
$R_1=-8\be$, so that:
\beq\label{R-val-forest-T}
R(w,x)=-2\beta(1+4x)=-2D(w,x)/\beta.
\eeq
Equation~\eqref{Q1-PQt-bis} does not involve $\tT_1$, but it
expresses $P_1$ in terms of $Q_0$ and $Q_1$:
\beq\label{P1Q-q0}
-48\nu \be w-4P_1\nu ^2+4\nu Q_0+(Q_1-1)(Q_1+\nu -3)=0.
\eeq
Finally,~\eqref{Tp1-T} gives again $R_1=-8\be$, and  we still need to relate $\tT_1$ (or equivalently
$\tT_2=\tT_1/\nu$) to the series $P_j$ and $Q_j$. This can be
done by specializing~\eqref{T2prime} to $q=0$:
\beq\label{Tt1P2-q0}
8\be \tT_1'/\nu-4P_2+1=0.
\eeq

We now go back to the differential system~\eqref{deT}. We will expand its
numerator in powers of $X=x+1/4$ instead of $x$.
We write accordingly
$$
P(w,x)= X^3+\tilde P_2X^2+\tilde P_1X+\tilde P_0,
\qquad
Q(w,x)= 2\nu X^2+\tQ_1 X+\tilde Q_0, \qquad R(w,x)=-8\beta X,
$$
where the $\tP_j$'s and $\tQ_j$ are series in $w$ with rational
coefficients in $\nu$. Note that we only have five unknown series
instead of seven, because $R$ is completely determined
(see~\eqref{R-val-forest-T}). The remaining three generic identities,
namely~\eqref{P2Q1-T}, \eqref{P1Q-q0} and \eqref{Tt1P2-q0}, give
\beq\label{P2Q1-T-q0}
\tQ_1=\nu \tP_2 +1-\nu /2,
\eeq
\beq
4 \nu \tQ_0=   4 {\nu}^{2}\tP_1 +2 {\nu}^{2}\tP_2 -   { \tQ_1 }
  ^{2}- 4\be \tQ_1
  +48 {\nu}\be w-(7\nu-6)(\nu-2)/4,
\label{Q0P-T-q0}
\eeq
and
\beq\label{P2T1-T-q0}
2\nu \tP_2={4\be} \tT_1'-\nu.
\eeq
We will need the values of several series at $w=0$, which follow from the initial conditions~\eqref{init-T-thm}:
\beq\label{init-T-q0}
  \tP_0(0)=0,
\qquad \tP_1(0)=1/16,
\qquad \tP_2(0)= -1/2, \qquad
 \tQ_0(0)= (\nu-2)/8.
\eeq
Expanding in $X$ the numerator of~\eqref{deT} gives
a system of five differential equations between the five series
$\tP_0$, $\tP_1$, $\tP_2$, $\tQ_0$ and $\tQ_1$:
\begin{align*}
  2\tP_0 \tQ_0'-\tQ_0\tP_0'&=0,
\\
2\tQ_0'\tP_1+2\tP_0\tQ_1'-\tP_0'\tQ_1-\tP_1'\tQ_0 -8\beta \tP_1&=0,
\\
2\tQ_0'\tP_2+2\tP_1\tQ_1'-\tP_1'\tQ_1-\tP_2'\tQ_0 -2\nu\tP_0'-16\beta \tP_2&=0,
\\
2\tQ_1'\tP_2-\tQ_1\tP'_2-2\nu\tP_1'+2\tQ_0'-24 \beta&=0,
\\
\tQ_1'-\nu\tP'_2&=0.
\end{align*}

The fifth (and last) equation is solved by~\eqref{P2Q1-T-q0}.

The fourth one is a consequence of~\eqref{Q0P-T-q0} and~\eqref{P2Q1-T-q0}.

The first one is readily integrated into $\tP_0= c\, \tQ_0^2$, for some
constant $c$.  The initial conditions~\eqref{init-T-q0} give $c=0$: in other words, $\tP_0=0$ and $P(w,x)$ is a
multiple of $D(w,x)$.

In the third equation, let us replace $\tP_0$
by 0, then  $\tQ_0$ by its expression  derived
from~\eqref{Q0P-T-q0}, and  $\tQ_1$ by its expression~\eqref{P2Q1-T-q0}.
Finally, let us introduce the primitive $S$ of $\tP_2$:
\beq\label{P2-S-def}
S:=\int \tP_2\, \mathrm{d}w = \frac{2\beta}\nu \tT_1 -\frac w 2
\eeq
(the second equality follows from~\eqref{P2T1-T-q0}, with $T_1(0)=0$).
The resulting equation (in $\tP_1$ and $S$) can now be integrated
 into:
\beq\label{P1S-T-q0}
-2  \left( 2 \nu S'   +\nu-2
 \right) \tP_1  +\nu  {  S'  } ^{3}+ \left( \nu-2 \right)
{ S'  } ^{2}/2+48 w \beta S'  -80  \beta S   =0.
\eeq
We have used  the initial
conditions~\eqref{init-T-q0} to determine the integration constant.

This is as far as we have been able to  integrate the
system. We can now express $\tP_1$, $\tP_2$, $\tQ_1$, $\tQ_0$ in terms of $S$ and $S'$
using (in this order) \eqref{P1S-T-q0}, \eqref{P2-S-def},
\eqref{P2Q1-T-q0}, and~\eqref{Q0P-T-q0}. We plug these values in the
second equation of the system, together with  $\tP_0=0$,  and this
gives us a second
order equation for $S$, or, equivalently, for $\tT_1$:
$$
P(\tT_1, \tT_1', \tT_1'')=0
$$
with
\begin{multline*}
  P(X,Y,Z)=
48 \beta^4 w Z Y^4+8  \beta^3 w (\beta -8) Z Y^3-80 \beta^4 X Z Y^3-12
\beta^2 w(\beta -2)  Z Y^2+120 \beta^3 X Z Y^2
\\
+6 \beta^2 w Z Y-60
\beta^2 X Z Y+400 \beta^3 \nu  Z X^2+10 \beta Z X-80 \nu\beta^2 w  (\nu +3)
 Z X+4 \beta (\nu +3)^2 \nu  w^2 Z
\\-\nu  w Z
-192 \nu  w
\beta^3 Y^3-16 \nu \beta^2 w (\beta -14)  Y^2 +160 \beta^3 \nu  X Y^2
\\-4
\nu  \beta w (\beta ^2-\beta +16)  Y
+40 \nu \beta^2 (\beta -5)  X Y+2 \nu^2 w \beta-20 \nu \beta
     (\beta -3)          X.
\end{multline*}
We now return to~\eqref{T10G} to derive an equation for $G(\beta,w)$,
or equivalently for $W=2G-w/\be$.

To finish, one checks that the recurrence relation that underlies the
differential equation for $G$ determines all coefficients of this
series once  we prescribe $G=O(w^2)$.
\end{proof}

\section{The self-dual Potts model on general planar maps}
\label{sec:sd}
To finish, we derive a third order differential equation for the
self-dual Potts model defined at the end of  Section~\ref{sec:specializations}.

\begin{Theorem}
Let $S\equiv  S(\be,t)$ be the \gf\ for the  self-dual Potts model on
planar maps.
More precisely, let
$$
S(\be,t)=\be t^2\,
M(\be^2,\be+1,\be^{-1},t;1)=\sum_{M}t^{\ee(M)+2}
\sum_{S\subseteq E(M)}\beta^{2\cc(S)+\ee(S)-1-\vv(M)},
$$
where the first sum runs over all rooted planar maps $M$.
This series is characterized by the initial conditions
$S=t^2+O(t^3)$ and the following differential equation
of order $3$ and degree $4$:
$$
P(S,S', S'', S''')=0,
$$
where
\begin{multline*}
P(X,Y,Z,T)= 2 p (   \la \beta Y -2) (2  \beta Y -1)  (3   \la \beta X
 -\beta  t-4 t) T
-\beta ^2 p ^2 (3   \la\beta X -\beta  t-4 t) Z^3\\
-p \beta
(-2 \la \beta^2  p Y^2+\beta (16 \la t-\beta-6) Y+24  \la^2 \beta^2 X
Y+6\la \beta (\beta-10) X-2 \la\beta t+48 t+2)
Z^2
\\+2
 (\la \beta Y-2)
(
2 \la \beta^2  p Y^2-\beta (8 \beta^2 t+160 \beta t+\beta+288 t-10) Y
+48 \la^2 \beta^2  X Y-48 \la \beta^2  X\\
+8 \beta^2 t+80 \beta t+\beta+64 t-4)
Z
-8 (2  \beta Y -1) (  \la Y -1) (\la \beta  Y-2)^2,
\end{multline*}
with $\la=\be+2$ and $p= 8\la t-1$.
\end{Theorem}

\begin{proof}
We  specialize Theorem~\ref{thm:ED} to $q=\be^2$, $w=1/\be$. The argument
of Section~\ref{sec:unique-general} proves that in this case as well,
the differential system and its initial conditions define uniquely
the series  $P_j$, $Q_j$ and $R_j$. Observe that one of the initial
conditions reads $R_2=0$, so that $R(t,x)$ has degree~1 in $x$.
This is the key simplification in the self-dual Potts model.

The polynomial $D(t,x)$ does not simplify drastically.
With the notation of the theorem, it reads:
$$
D(x,t)=\la {\beta}^{2}  {X}^{2}+ (\beta-2
)  p {\beta}^{2}/4,
$$
with $X=x-1/2$. We write accordingly
\beq\label{PQRt-sd}
P(t,x)= \tP(t,X)=X^4 +\tP_3 X^3+\tilde P_2X^2+\tilde P_1X+\tilde P_0,
\eeq
\beq\label{PQRt-sd-bis}
Q(t,x)= \tQ(t,X)= X^2+\tQ_1 X+\tilde Q_0, \qquad \hbox{and} \qquad R(t,x)=\tR(t,X)=\tR_1 X+\tR_0.
\eeq
The initial conditions at $t=0$ are even functions of $X$:
$$
P(0,x)= (X^2-1/4)^2, \qquad Q(0,x)=X^2-1/4.
$$
The other two initial conditions
read $\tP_4=1$ and $\tR_2=0$, and we have already taken them into
account in~\eqref{PQRt-sd} and~\eqref{PQRt-sd-bis}. But then it is easy to check that
$(\tP(t,-X), \tQ(t,-X), -\tR(t,-X))$ is another solution of the
differential system, satisfying the same initial conditions.
The uniqueness of the solution
implies that $\tP$ and $\tQ$ are even functions of $X$, while $\tR$
must be odd. Hence
$$
\tP_1= \tP_3=\tQ_1=\tR_0=0,
$$
and we are left with four unknown series only.

With these values, the  generic equations~\eqref{P3Q1} and~\eqref{Q1R-M}
automatically hold, and the other three reduce to:
\beq\begin{cases}\label{conj2}
4\la\tQ_0+ p(\tR_1-\be+2)&=0,
\\
12\be\la S+2\tQ_0- \tP_2-4(\be+4)t&=0,
\\
-2\be\la S' +\tR_1+4&=0,
\end{cases}
\eeq
where
 $S=\be \tM_1 =\be t^2 M_1$ is the series defined in the theorem.

Expanding in $X$ the numerator of~\eqref{de-encore} gives a system of four
equations in the four series $\tP_0$, $\tP_2$, $\tQ_0$, $\tR_1$,
corresponding to the coefficients of $X^{2i}$ for $i$ from 0 to 3. The
other coefficients vanish due to the parity properties.

In these four equations, let us replace $\tP_2$, $\tQ_0$ and $\tR_1$
by their expressions in terms of $S$ derived from~\eqref{conj2}. The
fourth equation then vanishes.
The three others are found to be linearly dependent (over
$\qs(\be)$), and we consider those corresponding to the coefficients
of $X^0$ and $X^4$:
$$
4( p \beta  S '' +4 \la \beta   S '-8 ) \tP _0
=
p(2  \beta  S '-1 ) \tP _0'
$$
and
$$
 4   \tP _0'=(2 \beta ^2 p^2 S '-p \beta  (48 \la \beta S -8
\beta  t-48 t-1)) S '' +8 \la \beta ^2  p S'^2-4 \beta
(\beta +6)
pS'+64 \beta  t+128 t-8.
$$
It remains to eliminate $\tP_0$ to obtain the differential equation
of the theorem. To conclude, one checks that the underlying
recurrence relation on the coefficients of $S(t)$ determines all of
them once the coefficient of $t^2$  is prescribed.
\end{proof}

\section{Final comments}

At the moment, our solution does not solve the important question of
locating phase transitions and finding critical exponents of the Potts
model on planar maps. Partial results are known, for instance for the
Ising model~\cite{BK87,Ka86,bernardi-mbm-alg}, for maps equipped with
a spanning forest~\cite{bernardi-mbm-alg}, and
for properly coloured triangulations~\cite{odlyzko-richmond}. The latter work exploits
Tutte's differential equation~\eqref{Tutte-ED}. Moreover, a parametrized description
of the critical value of $\nu$ (depending on $q$) for planar maps is given in~\cite{borot3},
for $0<q<4$. The analysis builds on the catalytic variable that
records the degree of the outer face rather than on the size variable,
but this should not influence the location of the critical point. Also
the results of~\cite{eynard-bonnet-potts} involve the dependence of the series in this
catalytic variable.

It is natural to ask if the Potts \gf\ might also satisfy a \emm
linear, differential equation. This would make the analysis of
critical points much easier, since finding the location and nature of the
singularities would then be (almost) automatic~\cite{flajolet-sedgewick}. However, this
possibility has been ruled out (for triangulations) by an asymptotic
argument in~\cite{mbm-courtiel}.

The second order equations that we have obtained for triangulations in
special cases (Sections~\ref{sec:nu=0} to~\ref{sec:q=0}) probably deserve a further
treatment. Do they fit in Painlev\'e's classification?

Finally, our proofs are very heavy, and it would be nice to establish
differential algebraicity by simpler, and ideally more combinatorial means. This has been done in
several special cases (Ising~\cite{mbm-schaeffer-ising,bouttier-mobiles}, spanning trees~\cite{mullin-boisees}, spanning
forests~\cite{mbm-courtiel}, bipolar orientations~\cite{bonichon-mbm-fusy,mbm-survey}). Two particularly attractive
problems are three-coloured maps (known to be algebraic~\cite{bernardi-mbm-alg}) and
properly coloured triangulations (Tutte's recurrence
relation~\eqref{Tutte-ED}).

\bigskip

\bigskip\noindent
{\bf Acknowledgements.} We are grateful to Jérémie Bouttier for his
explanations on~\cite{borot3}, and to Alexey Ovchinnikov and Guillaume
Rond  for their help with differential systems.
We also acknowledge interesting discussions with Ga\"etan Borot and Paul Zinn-Justin.


\bibliographystyle{plain}
\bibliography{coloured.bib}

\end{document}